\documentclass[12pt]{amsart}
\usepackage{txfonts}      
\usepackage{amssymb}
\usepackage{eucal}
\usepackage{amsmath}
\usepackage{amscd}
\usepackage{xcolor}
\usepackage{multicol}
\usepackage[all]{xy}           
\usepackage{graphicx}
\usepackage{color}
\usepackage{colordvi}
\usepackage{xspace}
\usepackage{tikz}
\usepackage{makecell}
\usepackage{appendix}
\usepackage{enumerate,enumitem}
\usepackage{amsthm}
\usepackage[thicklines]{cancel}
\usepackage[compress]{cite}
\usepackage{ifpdf}
\ifpdf
\usepackage[colorlinks,final,backref=page,hyperindex]{hyperref}
\else
\usepackage[colorlinks,final,backref=page,hyperindex,hypertex]{hyperref}
\fi

\usepackage[active]{srcltx} 
\newtheorem{thm}{Theorem}[section]
\newtheorem{lem}[thm]{Lemma}
\newtheorem{cor}[thm]{Corollary}
\newtheorem{pro}[thm]{Proposition}
\theoremstyle{definition}
\newtheorem{defi}[thm]{Definition}
\newtheorem{ex}[thm]{Example}
\newtheorem{rmk}[thm]{Remark}

\title[Anti-pre-Novikov algebras, quasi-triangular and factorizable anti-pre-Novikov
 bialgebras] {Anti-pre-Novikov algebras, quasi-triangular and factorizable anti-pre-Novikov bialgebras}

\author{Qinxiu Sun}
\address{Department of Mathematics, Zhejiang University of Science and Technology, Hangzhou, 310023} \email{qxsun@126.com}
\author{Xingyu Zeng}
\address{Department of Mathematics, Zhejiang University of Science and Technology, Hangzhou, 310023} \email{948861157@qq.com}


\begin{document}
\begin{abstract}
	Firstly, we introduce a notion of anti-pre-Novikov algebras as a new framework for decomposing Novikov algebras. 
Anti-$\mathcal O$-operators
on Novikov algebras are developed to provide an algebraic framework for constructing anti-pre-Novikov algebras.
Secondly, we introduce a notion of anti-pre-Novikov bialgebras as the bialgebra
structures corresponding to a double construction of symmetric quasi-Frobenius Novikov algebras, which
is characterized by certain matched pairs of Novikov algebras as well 
as the compatible anti-pre-Novikov algebras. The study of the coboundary case induces the anti-pre-Novikov
 Yang-Baxter equation (APN-YBE), 
whose skew-symmetric solutions yield coboundary anti-pre-Novikov bialgebras.
The notion of $\mathcal O$-operators on anti-pre-Novikov algebras is studied to construct skew-symmetric solutions
of the APN-YBE. Thirdly, we investigate quasi-triangular and
factorizable anti-pre-Novikov bialgebras as a special class of coboundary anti-pre-Novikov bialgebras.
The solutions of the APN-YBE whose symmetric parts are invariant give rise to a quasi-triangular anti-pre-Novikov bialgebra.
Moreover, relative Rota-Baxter operators with weights are introduced to 
demonstrate solutions of the APN-YBE whose symmetric parts are invariant. Finally, we introduce a 
notion of quadratic Rota-Baxter anti-pre-Novikov algebras, which is one to one correspondence to a
 factorizable anti-pre-Novikov bialgebra.

\end{abstract}

\maketitle

\vspace{-1.2cm}

\tableofcontents

\vspace{-1.2cm}

\allowdisplaybreaks

\section{Introduction}  
Novikov algebras emerged during the study of Hamiltonian operators in formal variational calculus (\cite{13, 30})
and Poisson brackets of hydrodynamic type (\cite{6, 9, 10}).
Explicitly, a Novikov algebra is a vector space $A$ equipped with a binary operation
$\circ$ satisfying the following
conditions:
\begin{align}&\label{Na1}
(x\circ y)\circ z-x\circ (y\circ z)=(y\circ x)\circ z-y\circ (x\circ z),\\&
\label{Na2}(x\circ y)\circ z=(x\circ z)\circ y,\quad  x,y,z\in A.
\end{align}
Novikov algebras are a special class of pre-Lie algebras (also called left-symmetric algebras), 
which are tightly connected with  many fields in mathematics and physics such as affine manifolds and affine
structures on Lie groups \cite{21}, convex homogeneous cones \cite{29}, vertex algebras \cite{2,6}
and so on.

The notion of pre-Novikov algebras and quasi-Frobenius Novikov algebras naturally appeared in the study of Novikov bialgebras \cite{16}. 
In detail, a pre-Novikov algebra is a vector space $A$ together with 
two binary operations $\vartriangleleft,\vartriangleright: A\otimes A\rightarrow A$
satisfying
\begin{align*}
a\vartriangleright (b\vartriangleright c)&=(a\circ b)\vartriangleright c+b\vartriangleright(a\vartriangleright c)-(b\circ a)\vartriangleright c,\\
a\vartriangleright(b\vartriangleleft c)&=(a\vartriangleright b)\vartriangleleft c+b\vartriangleleft(a\circ c)-(b\vartriangleleft a)\vartriangleleft c,\\
(a\circ b)\vartriangleright c&=(a\vartriangleright c)\vartriangleleft b,\\
(a\vartriangleleft b)\vartriangleleft c&=(a\vartriangleleft c)\vartriangleleft b,  
\end{align*}
 for all $a,b,c\in A$, where $a\circ b=a\vartriangleleft b+a\vartriangleright b$.
The sum $\circ=\succ+\prec$ gives a Novikov algebra $(A,\circ)$. 
More importantly, there is a Novikov algebra associated with a pre-Novikov algebra and pre-Novikov algebras can 
give skew-symmetric solutions of Novikov Yang-Baxter equation and hence Novikov bialgebras \cite{16}. 
Furthermore, pre-Novikov algebras correspond to a class of left-symmetric conformal algebras \cite{17} and there have
related connections between pre-Novikov algebras and Zinbiel algebras with a derivation (see \cite{20,31})

In the study of the underlying
algebraic structures for non-degenerate commutative 2-cocycles on Lie algebras, G. Liu and C. Bai
introduced a notion of anti-pre-Lie algebras \cite{23}.
Later, Gao, Liu and Bai in \cite{12} introduced a notion of anti-dendriform algebras. Anti-dendriform algebras still
have the property of splitting the associativity, but it is the negative left and right multiplication operators that compose the
bimodules of the sum associative algebras, instead of the left and right multiplication operators doing so for dendriform algebras.
Motivated by this work, we introduce anti-pre-Novikov algebras as the underlying algebraic 
structures for symmetric quasi-Frobenius Novikov algebras. Analogously, anti-pre-Novikov algebras split the operations. 
However, in this case, it is the negative left and right multiplication operators that form the bimodules of the sum Novikov algebra, 
in contrast to pre-Novikov algebras, where the (positive) left and right multiplication operators fulfill this role.

A bialgebraic structure consists of an algebra structure and a coalgebra structure coupled by 
certain compatible conditions between the multiplication and comultiplication. In the early 1980s,
Drinfeld established Lie bialgebra theories in \cite{8}, which are closely
related to the classical Yang-Baxter equation and play an
important role in the infinitesimalization of a quantum group. Based on Drinfeld’s work, V. Zhelyabin 
introduced a notion of associative D-bialgebras \cite{32,33}. 
The associative analog of the Lie bialgebra
is antisymmetric infinitesimal bialgebras, which were developed by Aguiar \cite{1}, which is equivalent 
to double constructions of Frobenius algebras \cite{4}. Subsequently, 
analogous methods were extended to develop bialgebra theories for various algebraic structures, including
 pre-Lie algebras \cite{3}, Leibniz
algebras \cite{28}, perm algebras \cite{19}, Novikov algebras\cite{16}, 
3-Lie algebras \cite{5} and pre-Novikov algebra \cite{22}. 
A Manin triple of Poisson algebras is equivalent to
a Poisson bialgebra \cite{26}, which naturally fits into a framework to construct compatible Poisson
brackets in integrable systems. Transposed Poisson algebras are dual to Poisson algebras. 
The Poisson bialgebra approach fails to yield a bialgebra theory for transposed Poisson algebras.
 Alternatively,
Bai and Liu developed bialgebra theories for 
anti-pre-Lie algebras, transposed Poisson algebras and anti-pre-Lie Poisson algebras \cite{24}
through the Manin triples 
with respect to the commutative 2-cocycles on Lie algebras. In \cite{030}, we explore the bialgebra theory for
anti-dendriform algebras, which is characterized by double constructions of associative algebras with 
respect to the commutative Connes cocycles
and certain matched pairs of associative algebras as well as anti-dendriform algebras.

Under the context of Lie bialgebras, coboundary Lie bialgebras especially quasi-triangular Lie bialgebras play fundamental roles
in mathematical physics. Factorizable Lie bialgebras, a specialized subclass 
of quasi-triangular Lie bialgebras, establish a crucial connection between classical $r$-matrices and specific factorization problems. 
They find diverse applications in integrable systems, see \cite{022} and references therein.
 Recently, these results concerning factorizable and quasi-triangular structures have been successfully generalized to 
 antisymmetric infinitesimal bialgebras \cite{029}, pre-Lie bialgebras \cite{033} and Leibniz bialgebras \cite{07}

It is natural to investigate the bialgebra structures for anti-pre-Novikov algebras. This is another motivation
 for writing this paper. Explicitly,
 we introduce a notion of an anti-pre-Novikov bialgebra, which is interpreted in terms of a double construction
  of symmetric quasi-Frobenius Novikov algebras
and certain matched pairs of Novikov algebras.
The study of coboundary anti-pre-Novikov bialgebras leads to the introduction of the anti-pre-Novikov Yang-Baxter
equation (APN-YBE). Furthermore, quasi-triangular anti-pre-Novikov bialgebras and
 factorizable anti-pre-Novikov bialgebras
as a special class of coboundary anti-pre-Novikov bialgebras
 are studied, which ‌are induced by solutions of the APN-YBE without imposing skew-symmetry constraints.
The double space of any anti-pre-Novikov bialgebra
admits a factorizable anti-pre-Novikov bialgebra structure.
Relative Rota-Baxter operators with weights are introduced to characterize the solutions
of the APN-YBE whose symmetric parts are invariant.
 Furthermore, we give a characterization of factorizable anti-pre-Novikov bialgebras
in terms of quadratic Rota-Baxter anti-pre-Novikov algebras.

The paper is organized as follows. In Section 2, 
we introduce a notion of anti-pre-Novikov algebras and its representations.
The anti-$\mathcal O$-operators on anti-pre-Novikov algebras are considered to
interpret anti-pre-Novikov algebras. 
In Section 3, we introduce a notion of anti-pre-Novikov bialgebras, which is equivalent to
a double construction of symmetric quasi-Frobenius Novikov algebras and
certain matched pairs of Novikov algebras. The study of coboundary case leads to 
the introduction of the APN-YBE, whose
skew-symmetric solutions give coboundary anti-pre-Novikov bialgebras. 
The notion of $\mathcal O$-operators on anti-pre-Novikov
algebras is introduced to construct skew-symmetric solutions
of the APN-YBE. In section 4, we consider quasi-triangular anti-pre-Novikov bialgebras
and factorizable anti-pre-Novikov bialgebras, which are
a special class of quasi-triangular anti-pre-Novikov bialgebras. 
We show that the double of an anti-pre-Novikov bialgebra is naturally a factorizable anti-pre-Novikov bialgebra.
In section 5, we introduce a notion of quadratic Rota-Baxter anti-pre-Novikov algebras, which corresponds to
a factorizable anti-pre-Novikov bialgebra.
 
Throughout the paper, $k$ is a field.  All vector spaces and algebras are over $k$. 
 All algebras are finite-dimensional, although many results still hold in the infinite-dimensional case.

\section{Anti-pre-Novikov algebras}
We introduce a notion of anti-pre-Novikov algebras as a new approach of splitting operations,
whose negative left and right multiplication operators compose the bimodules of the
associated Novikov algebra. The notions of anti-$\mathcal O$-operators on Novikov algebras are developed
 to interpret anti-pre-Novikov algebras. 
\subsection{Anti-pre-Novikov algebras and Representations}
Let us begin to recall some basic knowledge on Novikov algebras.

A {\bf representation (bimodule)} of a Novikov algebra $(A,\circ)$ is a triple $(V,l,r)$, 
where $V$ is a vector space and $l,r: A\rightarrow \text{End}_{\bf k}(V)$ are linear maps satisfying
\begin{align}
\label{Nr1} &l(x\circ y-y\circ x)v=l(x)l(y)v-l(y)l(x)v,\\
\label{Nr2}&l(x)r(y)v-r(y)l(x)v=r(x\circ y)v-r(y)r(x)v,\\
\label{Nr3} &l(x\circ y)v=r(y)l(x)v,\ \ \
r(x)r(y)v=r(y)r(x)v, 
\end{align} for all $ x,y\in A,\;v\in V.$

By Eqs. \eqref{Nr1} and \eqref{Nr2}, we have
\begin{align}
\label{Nr4} l_{\star}(x\circ y-y\circ x)v=l_{\star}(x)l_{\star}(y)v-l_{\star}(y)l_{\star}(x)v, \ \ where\ \  l_{\star}=l_{\circ}+r_{\circ}.
\end{align} 

\begin{pro} \cite{16} \label{d1}
Let $(A,\circ)$ and $(B,\bullet)$  be Novikov algebras. If $(B,l_A,r_A)$ is a representation 
of $(A,\circ)$, $(A,l_B,r_B)$ is a representation of $(B,\bullet)$ and the following conditions are satisfied:
\begin{flalign}
&l_B(a)(x\circ y)=-l_B(l_A(x)a-r_A(x)a)y+(l_B(a)x-r_B(a)x)\circ y+r_B(r_A(y)a)x+x\circ (l_B(a)y),\label{Nm1}\\
&r_B(a)(x\circ y-y\circ x)=r_B(l_A(y)a)x-r_B(l_A(x)a)y+x\circ (r_B(a)y)-y\circ (r_B(a)x),\label{Nm2}\\
&l_A(x)(a\bullet b)=-l_A(l_B(a)x-r_B(a)x)b+(l_A(x)a-r_A(x)a)\bullet b+r_A(r_B(b)x)a+a\bullet(l_A(x)b),\label{Nm3}\\
&r_A(x)(a\bullet b-b\bullet a)=r_A(l_B(b)x)a-r_A(l_B(a)x)b+a\bullet (r_A(x)b)-b\bullet (r_A(x)a),\label{Nm4}\\
&(l_B(a)x)\circ y+l_B(r_A(x)a)y=(l_B(a)y)\circ x+l_B(r_A(y)a)x,\label{Nm5}\\
&(r_B(a)x)\circ y+l_B(l_A(x)a)y=r_B(a)(x\circ y),\label{Nm6}\\
&l_A(r_B(a)x)b+(l_A(x)a)\bullet b=l_A(r_B(b)x)a+(l_A(x)b)\bullet a,\label{Nm7}\\
&l_A(l_B(a)x)b+(r_A(x)a)\bullet b=r_A(x)(a\bullet b), \quad   x,y\in A,~a,b\in B,\label{Nm8}
\end{flalign}
then there is a Novikov algebra structure on the direct sum $A\oplus B$ of the underlying vector spaces of $A$ and $B$ given by
\begin{align*}
(x+a)\cdot (y+b)=(x\circ y+l_B(a)y+r_B(b)x)+(a\bullet b+l_A(x)b+r_A(y)a),\;a, b\in A,\;x, y\in B.
\end{align*}
 $(A,B,l_A,r_A,l_B,r_B)$ satisfying the above conditions is called a \textbf{matched pair of Novikov algebras.} 
 Conversely, any Novikov algebra that can be decomposed into a direct sum of 
 two Novikov subalgebras is obtained from a matched pair of Novikov algebras.
\end{pro}

\begin{defi}
Let $A$ be a vector space. An {\bf anti-pre-Novikov algebra} is a vector space $A$ together with two binary operations 
$\succ,\prec: A\otimes A \rightarrow A$ satisfying
\begin{align} \label{Aa1}
&(x\circ y-y\circ x)\succ z=y\succ(x\succ z)-x\succ(y\succ z),\\
\label{Aa2}&x\prec(y\circ z)=(y\succ x) \prec z-(x\prec y)\prec z-y\succ (x\prec z),\\
\label{Aa3}&(x\circ y)\succ z=-(x\succ z)\prec y,\\
\label{Aa4}&(x\prec y)\prec z=(x\prec z)\prec y,  \\
\label{Aa5}&(x\circ y-y\circ x)\prec z=x\succ(y\circ z)-y\succ(x\circ z),
\end{align}
for all $x,y,z\in A$, where $x\circ y=x\succ y+x\prec y$.
\end{defi}

\begin{ex}
Let $(A,\succ,\prec)$ be a 1-dimensional anti-pre-Novikov algebra with a basis $\{ e\}$.
Suppose that $e\succ e=pe,~e\prec e=qe$ with $p,q\in k$. Then
$(p+q)q=-q^{2},~p^{2}+pq=-pq$. Thus, $p+2q=0$ or $p=q=0$.
\end{ex}

\begin{ex}
Let $(A,\succ,\prec)$ be a 3-dimensional anti-pre-Novikov algebra with a basis $\{ e_1,e_2,e_3 \}$
whose non-zero products are given by $e_1\succ e_1=e_2,~e_1\succ e_2=e_3$. 
\end{ex}

\begin{pro}
Let $A$ be a vector space with two binary operations $\succ$ and $\prec$. Define $x\circ y=x\succ y+x\prec y, \forall~x,y\in A$.
Then the following conditions are equivalent:
\begin{enumerate}
	\item $(A,\succ,\prec)$ is an anti-pre-Novikov algebra.
 \item  $(A,\circ)$ is a Novikov algebra
 and Eqs.~(\ref{Aa1})-(\ref{Aa4}) hold for all $x,y,z\in A$.
 \item $(A,\circ)$ is a Novikov algebra and $(A,-L_{\succ},-R_{\prec})$ is a bimodule of $(A,\circ)$,
\end{enumerate}
where  $L_{\succ},R_{\prec} : A\longrightarrow \hbox{End} (A)$ are the
linear maps defined by $L_{\succ}(x)(y)=R_{\prec}(y)(x)=x\prec y$.
\end{pro}
\begin{proof}
It can be proved directly. 
\end{proof}

Assume that $(A,\succ,\prec)$ is an anti-pre-Novikov algebra.
By Eqs.~\eqref{Na2}, \eqref{Aa1} and \eqref{Aa4}, we have
\begin{equation}(x\succ z)\prec y=(x\succ y)\prec z, \ \ \ (x\circ y)\succ z=(x\circ z)\succ y, \ \ \
(x\circ y)\prec z=(x\circ z)\prec y.\end{equation}
Define \begin{equation} \label{Aa6} x\odot y=x\succ y+y\prec x,\ \ \ x\star y=x\circ y+y\circ x,\end{equation}
then we obtain
\begin{align}\label{Aa7} &x\star y=x\odot y+y\odot x, \ \ \ (x\circ y)\star z=x\star (z\circ y),\\
\label{Aa8} &x\odot (y\star z)-z\odot (x\star y)=y\odot(x\circ z-z\circ x), \ \ \ 
x\succ (y\circ z)=y\odot (x\circ z)-(y\circ x)\prec z.\end{align}

\begin{defi}
Let $(A,\succ,\prec)$ be an anti-pre-Novikov algebra, $V$ a vector space and
$l_{\succ},r_{\succ},l_{\prec},r_{\prec}:A\rightarrow \text{End}(V)$ be linear maps. $(V,l_{\succ},r_{\succ},l_{\prec},r_{\prec})$ is called a \textbf{representation (bimodule)} of $(A,\succ,\prec)$ if the following conditions hold:
\begin{flalign}
&l_{\succ}(x\circ y-y\circ x)=l_{\succ}(y)l_{\succ}(x)-l_{\succ}(x)l_{\succ}(y),  \label{rp1}\\
&r_{\prec}(x\circ y)=r_{\prec}(y)l_{\succ}(x)-r_{\prec}(y)r_{\prec}(x)-l_{\succ}(x)r_{\prec}(y), \label{rp2}\\
&l_{\succ}(x\circ y)=-r_{\prec}(y)l_{\succ}(x), \ \ \ \  l_{\prec}(x\prec y)=r_{\prec}(y)l_{\prec}(x), \label{rp3}\\
&l_{\prec}(x\circ y-y\circ x)=l_{\succ}(x)l_{\circ}(y)-l_{\succ}(y)l_{\circ}(x),  \label{rp4}\\
&r_{\succ}(x)(l_{\circ}(y)-r_{\circ}(y))=r_{\succ}(y\succ x)-l_{\succ}(y)r_{\succ}(x), \label{rp5}\\
&l_{\prec}(x)l_{\circ}(y)=l_{\prec}(y\succ x)-l_{\prec}(x\prec y)-l_{\succ}(y)l_{\prec}(x), \label{rp6}\\
&r_{\prec}(x)(l_{\circ}(y)-r_{\circ}(y))=l_{\succ}(y)r_{\circ}(x)-r_{\succ}(y\circ x), \label{rp7}\\
&r_{\succ}(x)l_{\circ}(y)=-l_{\prec}(y\succ x), \ \ \ \ r_{\prec}(x)r_{\prec}(y)=r_{\prec}(y)r_{\prec} (x), \label{rp8}
\\
&l_{\prec}(x)r_{\circ}(y)=r_{\prec}(y)r_{\succ} (x)-r_{\prec}(y)l_{\prec}(x)-r_{\succ}(x\prec y), \label{rp9}\\
&r_{\succ}(x)r_{\circ}(y)=-r_{\prec}(y)r_{\succ} (x),\label{rp10}
\end{flalign}
for all $x,y\in A$, where $x\circ y=x\succ y+x\prec y$.
\end{defi}

\begin{pro}
Let $(A,\succ,\prec)$  be an anti-pre-Novikov algebra and $V$ a vector space. Assume that
  $l_{\succ},r_{\succ},l_{\prec},r_{\prec}:A\rightarrow \text{End}(V)$ are linear maps.
   Define two binary operations $\succeq$ and $ \preceq$ on the direct sum $A\oplus V$ of vector spaces by
 \begin{align*}
&(x+u)\succeq(y+v)=x\succ y+l_{\succ}(x)v+r_{\succ}(y)u,\\
 &(x+u)\preceq(y+v)=x\prec y+l_{\prec}(x)v+r_{\prec}(y)u, \ \ \ \forall ~x,y\in A, u, v\in V.
\end{align*}
 Then $(V,l_{\succ},r_{\succ},l_{\prec},r_{\prec})$ is a representation of $(A,\succ,\prec)$ 
 if and only if $(A\oplus V, \succeq,\preceq)$ is an anti-pre-Novikov algebra, which is called the {\bf semi-direct product} of $A$ and $V$. 
 Denote it simply by $A\ltimes V$.
\end{pro}
\begin{proof}
It is a special case of Proposition \ref{M0}. 
\end{proof}

Let $A$ and $V$ be vector spaces. For
a linear map $f: A \longrightarrow \hbox{End} (V)$, define a linear
map $f^{*}: A \longrightarrow \hbox{End} (V^{*})$ by $\langle
f^{*}(x)u^{*},v\rangle=-\langle u^{*},f(x)v\rangle$ for all $x\in A,
u^{*}\in V^{*}, v\in V$, where $\langle \ , \ \rangle$ is the usual
pairing between $V$ and $V^{*}$.

\begin{pro} \label{zr} Let $(A,\succ,\prec)$ be an anti-pre-Novikov algebra and $(V,l_{\succ}, r_{\succ},l_{\prec},r_{\prec})$ be its
representation. Then
\begin{enumerate}
	\item $(V,-l_{\succ},-r_{\prec})$ is a representation of the associated
Novikov algebra $(A,\circ)$.
 \item $(V,l_{\circ},r_{\circ})$ is a representation of the 
 associated Novikov algebra $(A,\circ)$.
	\item $(V^*,-l_{\star}^*,-r_{\succ}^*,
r_{\odot}^*,r_{\circ}^*)$
is a representation of $(A,\succ,\prec)$. We call it the {\bf dual representation}.
	\item $(V^{*},l_{\star}^*,-r_{\circ}^*)
$ is a representation of the 
 associated Novikov algebra $(A,\circ)$.
\item $(V^{*},-l_{\odot}^{*},r_{\prec}^{*})$ is a
	representation of the 
 associated Novikov algebra $(A,\circ)$,
\end{enumerate}
where $l_{\circ}=l_{\prec}+l_{\succ},~r_{\circ}=r_{\prec}+r_{\succ},
 l_{\star}=l_{\circ}+r_{\circ},~l_{\odot}=l_{\succ}+r_{\prec}$ and $r_{\odot}=r_{\succ}+l_{\prec}.$
\end{pro}

\begin{proof} 
(a) It can be obtained by Eqs.~(\ref{rp1})-(\ref{rp3}) and (\ref{rp8}).

(b) It can be verified directly or follows by the semi-direct product of Novikov algebras.

(c) According to Eqs.~(\ref{Nr1})-(\ref{Nr3}), for all $x,y\in A,v^{*}\in V^{*}$ and $w\in V$ we have
\begin{align*}&
\langle l_{\star}^{*}(x\circ y-y\circ x)v^{*},w\rangle+ \langle l_{\star}^{*}(y)l_{\star}^{*}(x)v^{*}-l_{\star}^{*}(x)l_{\star}^{*}(y)v^{*}
,w\rangle
\\=&-\langle v^{*},l_{\star}(x\circ y-y\circ x)w \rangle- \langle v^{*}
,(l_{\star}(x)l_{\star}(y)v^{*}-l_{\star}(y)l_{\star}(x))w\rangle
\\=&0,\end{align*}
which indicates that
Eq.~(\ref{rp1}) holds for $l_{\succ}=-(l_{\circ}^{*}+r_{\circ}^{*})$.
Analogously, Eqs.~(\ref{rp2})-(\ref{rp10}) hold for $(-l_{\star}^*,-r_{\succ}^*,
r_{\odot}^*,r_{\circ}^*)$.

Items (d) and (e) are obtained directly from Items (a), (b) and (c).
\end{proof}

\begin{ex} Let $(A,\succ,\prec)$ be an anti-pre-Novikov algebra
	and $L_{\succ},R_{\succ},L_{\prec},R_{\prec} : A\longrightarrow \hbox{End} (A)$ be 
linear maps defined by $L_{\succ}(x)(y)=R_{\succ}(y)(x)=x\succ y,~~L_{\prec}(x)(y)=R_{\prec}(y)(x)=x\prec y$ 
for all $x,y\in A$. Then 
\begin{enumerate}
\item $ (A,L_{\succ},R_{\succ},L_{\prec},R_{\prec})$ is a representation of $(A,\succ,\prec)$,
which is called the {\bf regular representation} of $(A,\succ,\prec)$. 
Moreover, $(A^{*},-(L_{\prec}^{*}+L_{\succ}^{*}+R_{\prec}^{*}+R_{\succ}^{*}),-R_{\succ}^{*},
(L_{\prec}^{*}+R_{\succ}^{*}),(R_{\prec}^{*}+R_{\succ}^{*}))$ 
is the dual representation of $(A,L_{\succ},R_{\succ},L_{\prec},R_{\prec})$. 
\item $ (A,-L_{\succ},-R_{\prec})$ and $(A^{*},-(L_{\succ}^{*}+R_{\prec}^{*}),R_{\prec}^{*})$ are all representations of the associated
Novikov algebra $(A,\circ)$. 
\end{enumerate}
\end{ex}

\begin{pro} \label{M0} Let $(A_{1},\succ_{1},\prec_{1})$ and
$(A_{2},\succ_{2},\prec_{2})$ be two anti-pre-Novikov algebras. Suppose that there are linear maps
$l_{\succ_1},r_{\succ_1},l_{\prec_1},r_{\prec_1}:A_1\longrightarrow \hbox{End}(A_2)$
and $l_{\succ_2},r_{\succ_2},l_{\prec_2},r_{\prec_2}:A_2\longrightarrow \hbox{End}(A_1)$ such that 
$(A_2,l_{\succ_1},r_{\succ_1},l_{\prec_1},r_{\prec_1})$ is a representation of $(A_1,\succ_1,\prec_1)$
and $(A_1,l_{\succ_2},r_{\succ_2},l_{\prec_2},r_{\prec_2})$ is a representation of $(A_2,\succ_2,\prec_2)$.
Moreover, the following compatible conditions hold for all $x,y\in A_{1}$ and $a,b\in A_{2}$:
\begin{align*}
&r_{\succ_2}(a)(x\circ_1 y-y\circ_1 x)=y\succ_1 r_{\succ_2}(a)x-x\succ_1 r_{\succ_2}(a)y+r_{\succ_2}(l_{\prec_1}(x)a)y-r_{\succ_2}(l_{\prec_1}(y)a)x,
\\&
((r_{\circ_2}-l_{\circ_2})(a)x)\succ_1 y+l_{\succ_2}((l_{\circ_1}-r_{\circ_1})(x)a)y=l_{\succ_2}(a)(x\succ_1 y)-x\succ_{1}l_{\succ_2}(a)y-
r_{\succ_2}(r_{\succ_1}(y)a)x\\
&r_{\succ_2}(a)(x\circ_1 y)=-r_{\succ_2}(a)x\prec_1 y-l_{\prec_2}(l_{\succ_1}(x)a) y,\\
&r_{\circ_2}(a)x\succ_1 y+l_{\succ_2}(l_{\succ_1}(x)a)y=-r_{\prec_2}(a)(x\succ_1 y),\\
&l_{\circ_2}(a)x\succ_1 y+l_{\succ_2}(r_{\circ_1}(x)a)y=-l_{\succ_2}(a)y\prec_1 x-l_{\prec_2}(r_{\succ_1}(y)a)x,\\
&r_{\prec_2}(a)(x\prec_1 y)=r_{\prec_2}(a)x\prec_1 y+l_{\prec_2}(l_{\prec_1}(x)a)y,\\
&l_{\prec_2}(a)x\prec_1 y+l_{\prec_2}(r_{\prec_1}(x)a)y=l_{\prec_2}(a)y\prec_1 x+l_{\prec_2}(r_{\prec_1}(y)a)x,\\
&r_{\prec_2}(a)(x\circ_1 y-y\circ_1 x)=x\succ_1 r_{\circ_2}(a)y-y\succ_1 r_{\circ_2}(a)x+r_{\succ_2}(l_{\circ_1}(y)a)x-r_{\succ_2}(l_{\circ_1}(x)a)y,\\
&((r_{\circ_2}-l_{\circ_2})(a)x)\prec_1 y-l_{\prec_2}((l_{\circ_1}-r_{\circ_1})(x)a)y=x\succ_1 l_{\circ_2}(a)y+r_{\succ_2}
(r_{\circ_1}(y)a)x-l_{\succ_2}(a)(x\circ_1 y),\\&
x\prec_1 r_{\circ_2}(a)y+r_{\prec_2}(l_{\circ_1}(y)a)x=r_{\prec_2}(a)(y\succ_1 x-x\succ_1 y)-y\succ_1 r_{\prec_2}(a)x-r_{\succ_2}(l_{\prec_1}(x)a)y,\\&
x\prec_1 l_{\circ_2}(a)y+r_{\prec_2}(r_{\circ_1}(y)a)x=(l_{\succ_2}(a)-r_{\prec_2}(a))x\prec_1 y+l_{\prec_2}((r_{\succ_1}-l_{\prec_1})(x)a)y-
l_{\succ_2}(a)(x\prec_1y),\\
&l_{\prec_2}(a)(x\circ_1y)=((r_{\succ_2}-l_{\prec_2})(a)x)\prec_1 y+l_{\prec_2}((l_{\succ_1}-r_{\prec_1})(x)a)y-
x\succ_1 (l_{\prec_2}(a)y)-r_{\succ_2}(r_{\prec_1} (y)a)x,\\
&r_{\succ_1}(x)(a\circ_2 b-b\circ_2 a)=b\succ_2 r_{\succ_1}(x)a-a\succ_2 r_{\succ_1}(x)b+r_{\succ_1}(l_{\prec_2}(a)x)b-r_{\succ_1}(l_{\prec_2}(b)x)a,
\\&
(r_{\circ_1}(x)-l_{\circ_1}(x))a\succ_2 b+l_{\succ_1}((l_{\circ_2}-r_{\circ_2})(a)x)b=l_{\succ_1}(x)(a\succ_2 b)-a\succ_{2}l_{\succ_1}(x)b-
r_{\succ_1}(r_{\succ_2}(b)x)a,\\
&r_{\succ_1}(x)(a\circ_2 b)=-r_{\succ_1}(x)a\prec_2 b-l_{\prec_1}(l_{\succ_2}(a)x) b,\\
&r_{\circ_1}(x)a\succ_2 b+l_{\succ_1}(l_{\succ_2}(a)x)b=-r_{\prec_1}(x)(a\succ_2 b),\\
&l_{\circ_1}(x)a\succ_2 b+l_{\succ_1}(r_{\circ_2}(a)x)b=-l_{\succ_1}(x)b\prec_2 a-l_{\prec_1}(r_{\succ_2}(b)x)a,\\
&r_{\prec_1}(x)(a\prec_2 b)=r_{\prec_1}(x)a\prec_2 b+l_{\prec_1}(l_{\prec_2}(a)x)b,\\
&l_{\prec_1}(x)a\prec_2 b+l_{\prec_1}(r_{\prec_2}(a)x)b=l_{\prec_1}(x)b\prec_2 a+l_{\prec_1}(r_{\prec_2}(b)x)a,\\
&r_{\prec_1}(x)(a\circ_2 b-b\circ_2 a)=a\succ_2 r_{\circ_1}(x)b-b\succ_2 r_{\circ_1}(x)a+r_{\succ_1}(l_{\circ_2}(b)x)a-r_{\succ_1}(l_{\circ_2}(a)x)b,\\
&((r_{\circ_1}-l_{\circ_1})(x)a)\prec_2 b-l_{\prec_1}((l_{\circ_2}-r_{\circ_2})(a)x)b=a\succ_2 l_{\circ_1}(x)b+r_{\succ_1}
(r_{\circ_2}(b)x)a-l_{\succ_1}(x)(a\circ_2 b),\\&
a\prec_2 r_{\circ_1}(x)b+r_{\prec_1}(l_{\circ_2}(b)x)a=r_{\prec_1}(x)(b\succ_2 a-a\succ_2 b)-b\succ_2 r_{\prec_1}(x)a-r_{\succ_1}(l_{\prec_2}(a)x)b,\\&
a\prec_2 l_{\circ_1}(x)b+r_{\prec_1}(r_{\circ_2}(b)x)a=((l_{\succ_1}-r_{\prec_1})(x)a)\prec_2 b+l_{\prec_1}((r_{\succ_2}-l_{\prec_2})(a)x)b-
l_{\succ_1}(x)(a\prec_2 b),\\
&l_{\prec_1}(x)(a\circ_2 b)=((r_{\succ_1}-l_{\prec_1})(x)a)\prec_2 b+l_{\prec_1}((l_{\succ_2}-r_{\prec_2})(a)x)b-
a\succ_2(l_{\prec_1}(x)b)-r_{\succ_1}(r_{\prec_2} (b)x)a,
\end{align*} where
$l_{\circ_1}=l_{\succ_1}+l_{\prec_1},\ l_{\circ_2}=l_{\succ_2}+l_{\prec_2},\ 
r_{\circ_1}=r_{\succ_1}+r_{\prec_1},\ r_{\circ_2}=r_{\succ_2}+r_{\prec_2}.$
Define two binary operations $\succ$ and $\prec$ on the direct sum $A_1\oplus A_2$ of the underlying vector spaces 
of $A_1$ and $A_2$ by 
\begin{align*}&(x+a)\succ(y+b)=x\succ_{1}y+l_{\succ_2}(a)y+r_{\succ_2}(b)x+a\succ_{2}b+l_{\succ_1}(x)b+r_{\succ_1}(y)a
,\\&
(x+a)\prec(y+b)=x\prec_{1}y+l_{\prec_2}(a)y+r_{\prec_2}(b)x+a\prec_{2}b+l_{\prec_1}(x)b+r_{\prec_1}(y)a,
\end{align*}
for all $x,y\in A_1,a,b\in A_2.$
Then $(A_1\oplus A_2,\succ,\prec)$ is an anti-pre-Novikov algebra.
Denote this anti-pre-Novikov algebra by $A_1\bowtie A_2$ and 
$(A_{1},A_{2},l_{\succ_1},r_{\succ_1},l_{\prec_1},r_{\prec_1},l_{\succ_2},r_{\succ_2},l_{\prec_2},r_{\prec_2})$
satisfying the above conditions is called a {\bf matched pair of
anti-pre-Novikov algebras}. Conversely, any anti-pre-Novikov algebra that can be decomposed into a 
direct sum of two anti-pre-Novikov subalgebras is obtained from a matched pair of anti-pre-Novikov algebras.
\end{pro}
\begin{proof} It can be verified by direct computations.
\end{proof} 

\begin{cor}\label{Ma} If
$(A_{1},A_{2},l_{\succ_1},r_{\succ_1},l_{\prec_1},r_{\prec_1},l_{\succ_2},r_{\succ_2},l_{\prec_2},r_{\prec_2})$ is
a matched pair of anti-pre-Novikov algebras, then 
$(A_{1},A_{2},l_{\succ_1}+l_{\prec_1},r_{\succ_1}+r_{\prec_1},l_{\succ_2}+l_{\prec_2},r_{\succ_2}+r_{\prec_2})$
is a matched pair of Novikov algebras.
\end{cor}

\begin{proof}
In view of Proposition \ref{M0}, there is an anti-pre-Novikov algebra $(A_1\bowtie A_2,\succ,\prec)$, whose associated
Novikov algebra is defined by
\begin{align*}&(x+a)\circ(y+b)=(x+a)\succ(y+b)+(x+a)\prec(y+b)
\\=&x\succ_{1}y+l_{\succ_2}(a)y+r_{\succ_2}(b)x+a\succ_{2}b+l_{\succ_1}(x)b+r_{\succ_1}(y)a
+x\prec_{1}y+l_{\prec_2}(a)y+r_{\prec_2}(b)x\\&+a\prec_{2}b+l_{\prec_1}(x)b+r_{\prec_1}(y)a
\\=&x\circ_{1}y+(l_{\succ_2}+l_{\prec_2})(a)y+(r_{\succ_2}+r_{\prec_2})(b)x+a\circ_{2}b+(l_{\succ_1}+l_{\prec_1})(x)b+(r_{\succ_1}+r_{\prec_1})(y)a
.\end{align*} 
Combining Proposition \ref{d1}, we get the conclusion.
\end{proof}

\subsection{ Anti-$\mathcal O$-operators }
\begin{defi} Let $(V,l,r)$ be a representation of a Novikov algebra $(A,\circ )$. A linear map $T:V\rightarrow A$ is 
called an \textbf{anti-$\mathcal{O}$-operator} on $(A,\circ )$ associated to $(V,l,r)$ if $T$ satisfies
\begin{equation}\label{Ao10}
T(u)\circ T(v)=-T(l(T(u))v+r(T(v))u),\ \ \ \forall~ u,v\in V.
\end{equation}
Furthermore, $T$ is called strong if 
\begin{equation}\label{Ao1}
l(T(u)\circ T(v)-T(v)\circ T(u))w+r(T(u)\circ T(w))v-r(T(v)\circ T(w))u=0,\ \ \ \forall~ u,v,w\in V.
\end{equation}
In particular, an anti-$\mathcal{O}$-operator $T$ on $(A,\circ )$ associated to
the representation $(A,L_{\circ},R_{\circ})$ is called an anti-Rota-Baxter operator, that is, $T:A\longrightarrow A$
is a linear map satisfying 
\begin{equation}
T(x)\circ T(y)=-T(T(x)\circ y)+x\circ T(y)),\ \ \ \forall~ x,y\in A. \end{equation}
An anti-Rota-Baxter operator $T$ is called strong if $T$ satisfies 
\begin{equation}
[T(x),T(y)]\circ z+y\circ (T(x)\circ T(z))-x\circ (T(y)\circ T(z))=0,\ \ \ \forall~ x,y,z\in A,\end{equation}
where $[x,y]=x\circ y-y\circ x$.
\end{defi}

\begin{pro}\label{Am1} Let $(A,\circ)$ be a Novikov algebra and $(V,l,r)$ be a representation of $(A,\circ )$. Assume that
$T:V\longrightarrow A$ is an anti-$\mathcal{O}$-operator on $(A,\circ)$ associated to $(V,l,r)$. Define two
binary operations $\succ,\prec$ on $V$ respectively by
\begin{equation} \label{Ao2} u\succ v=-l(T(u))v, \ \ \ 
u\prec v=-r(T(v))u, \ \ \ \forall ~u,v\in V.\end{equation} Then we have
\begin{enumerate}
\item For all $u,v,w\in V$,  denote $u\cdot v=u\succ v+u\prec v$, the following equations hold:
\begin{align}&\label{Ao3}  (v\cdot u-u\cdot v)\succ w=u\succ(v\succ w)-v\succ(u\succ w), 
\\&\label{Ao4} (u\succ w)\prec v-u\succ(w\prec v)=w\prec(u\cdot v)+(w\prec u)\prec v,
\\&\label{Ao5} (u\cdot v)\succ w=-(u\succ w)\prec v,\ \ \ (w\prec v)\prec u=(w\prec u)\prec v.\end{align}
\item $(V,\succ,\prec)$ is an anti-pre-Novikov algebra if and only if $T$ is strong. In this case, $T$ 
is a homomorphism of Novikov algebras from $(V,\cdot)$ to $(A,\circ)$. Moreover, there is an induced
anti-pre-Novikov algebra structure on $T(V)=\{T(u)|u\in V \} \subseteq A$ given by
\begin{equation*}T(u)\succ T(V)=T(u\succ v),\ \ \ T(u)\prec T(V)=T(u\prec v),\ \ ~\forall~u,v\in V
\end{equation*}
and $T$ 
is a homomorphism of anti-pre-Novikov algebras.
\item If $T$ is invertible, then $T$ is strong. 
\end{enumerate}
\end{pro}

\begin{proof} (a)
Using Eqs. \eqref{Nr1} and \eqref{Ao2}, for all $u,v,w\in V$ we have
\begin{align*}& u\succ(v\succ w)-v\succ(u\succ w), 
\\=&l(Tu)l(Tv)w-l(Tv)l(Tu)w=l(Tu\circ Tv-Tv\circ Tu)w
\\=&-l(T(l(Tu)v+r(Tv)u)-l(T(l(Tv)u+r(Tu)v)w=-(u\cdot v-v\cdot u)\succ w,
\end{align*}
which yields that Eq.~\eqref{Ao3} holds. Analogously, we can prove that Eqs.~\eqref{Ao4}-\eqref{Ao5} hold.

(b) $(V,\succ,\prec)$ is an anti-pre-Novikov algebra if and only if Eqs.~\eqref{Ao3}-\eqref{Ao5} hold and
\begin{equation}\label{Ao6} (u\cdot v-v\cdot u)\prec w=u\succ(v\cdot w)-v\succ(u\cdot w),\ \ \ \forall~u,v,w\in V.\end{equation}
By Eqs.~\eqref{Nr1}-\eqref{Nr2} and \eqref{Ao10}-\eqref{Ao2}, we get
\begin{align*}&(u\cdot v-v\cdot u)\prec w-u\succ(v\cdot w)+v\succ(u\cdot w)
\\=&r(Tw)(l(Tu)v+r(Tv)u-r(Tw)(l(Tv)u+r(Tu)v)-l(Tu)(l(Tv)w+r(Tw)v)\\&+l(Tv)(l(Tu)w+r(Tw)u)
\\=&r(Tv\circ Tw)u-r(Tu\circ Tw)v-l(Tu\circ Tv-Tv\circ Tu)w.\end{align*}
Thus, Eq.~(\ref{Ao6}) holds if and only if Eq.~(\ref{Ao1}) holds.
The other conclusions follow immediately.

(c) In view of Eqs.~\eqref{Ao10} and \eqref{Ao2}, we have
\begin{align*}
&(Tu\circ Tv)\circ Tw
=-T(l(Tu)v+r(Tv)u)\circ Tw
\\=&T[l(T(l(Tu)v+r(Tv)u))w+r(Tw)(l(Tu)v+r(Tv)u)]
=T[(u\cdot v)\succ w+(u\cdot v)\prec w].
\end{align*}
By the same token,
\begin{align*} &(Tv\circ Tu)\circ Tw=T[(v\cdot u)\succ w+(v\cdot u)\prec w],
\\& Tu\circ (Tv\circ Tw)=T[u\succ(v\cdot w)+u\prec(v\cdot w)],\\&Tv\circ (Tu\circ Tw)=T[v\succ(u\cdot w)+v\prec(u\cdot w)].\end{align*}
Since $(A,\circ)$ is a Novikov algebra and $T$ is
invertible, we obtain
\begin{equation*}(u\cdot v)\cdot w-(v\cdot u)\cdot w=u\cdot( v\cdot w)-v\cdot( u\cdot w).\end{equation*}
Combining Items (a) and (b), we get that Eq.~\eqref{Ao6} holds. Thus, $(V,\succ,\prec)$ is an anti-pre-Novikov algebra.
It follows that $T$ is strong.
\end{proof}

\begin{ex} Let $(A,\circ)$ be the 2-dimensional Novikov algebra defined in \cite{7} with a basis $\{e_1, e_2 \}$, whose
non-zero multiplication is given by
$e_1 \circ e_1 = e_1, \ e_2 \circ e_1 = e_2.$ Define a linear map 
$T:A\longrightarrow A$ by $T(e_1 )=ae_2,~~T(e_2 )=0~(a\in k)$. Then $T$ is a strong anti-Rota-Baxter operator 
on $(A,\circ)$. Moreover, there is an anti-pre-Novikov algebra structure $(\succ,\prec)$ on $A$, 
whose non-trivial binary operation is given by
$e_1 \succ e_1 =-a e_2$.
\end{ex}

\begin{thm}\label{Am2}
Let $(A,\circ)$ be a Novikov algebra. Then there is a
compatible anti-pre-Novikov algebra structure on $(A,\circ)$ if and
only if there exists an invertible anti-$\mathcal{O}$-operator on
$(A,\circ)$.
\end{thm}

\begin{proof}
Assume that $(A,\succ,\prec)$ is a compatible anti-pre-Novikov algebra
 structure on $(A,\circ)$. Then
\begin{equation*}x\circ
y=x\succ y+x\prec y=-(-L_{\succ}
(x) y-R_{\prec} (y)x),\;\;\forall x,y\in A,\end{equation*} which means that the
identity map $I:A\rightarrow A$ is an invertible
anti-$\mathcal{O}$-operator on $(A,\circ)$ associated to the
representation $(A,-L_{\succ},-R_{\prec})$.

On the other hand, assume that $T:V\rightarrow A$ is an invertible
anti-$\mathcal{O}$-operator on $(A,\circ)$ associated to a
representation $(V, l, r)$ of $(A,\circ)$. 
In view of Proposition \ref{Am1}, there exists an anti-pre-Novikov algebra
 structure on $V$ given by
Eq.~\eqref{Ao2} and $T(V)=A$. For all $x,y\in
A$, due to $T$ being invertible, there exist $u,v\in V$ such that $x=T(u),y=T(v)$. Thus
we have
\begin{align*}
x\circ y&=T(u)\circ T(v)=-T(l(T(u))v+r(T(v))u)=T(u\succ v+u\prec v)\\
&=T(u)\succ T(v)+T(u)\prec T(v)=x\succ y+x\prec y,
\end{align*}
which indicates that $(A,\succ,\prec)$ is a compatible anti-pre-Novikov algebra on $(A,\circ)$.
\end{proof}

 \begin{defi}
Let $(A,\circ )$ be a Novikov algebra. A triple $(A,\circ ,\omega)$ is called a \textbf{symmetric quasi-Frobenius Novikov algebra}
if there is a symmetric non-degenerate bilinear form $\omega $ on $A$ satisfying
\begin{equation} \label{Qn}
\omega (x\circ y,z)-\omega (x\circ z+z\circ x,y)+\omega (z\circ y,x)=0 ,\quad x,y,z\in A,
\end{equation}
\end{defi}

\begin{defi}Let $(A,\succ,\prec)$ be an anti-pre-Novikov algebra and $\omega$ a non-degenerate symmetric bilinear form.
If $\omega$ is invariant, that is,
\begin{equation} \label{C2}\omega (x \prec y, z)=-\omega(x, z\circ y), \ \  \ 
\omega(x \succ y, z)=\omega(x\circ z+z\circ x, y), ~\forall~x, y, z \in A.\end{equation}
Then $(A,\succ,\prec,\omega)$ is called a \textbf{quadratic anti-pre-Novikov algebra}.
\end{defi}

\begin{thm}\label{Am3} Let $(A,\circ,\omega)$ be a symmetric quasi-Frobenius Novikov algebra.
 Then there exists a compatible anti-pre-Novikov algebra structure $(A,\succ,\prec)$
 on $(A,\circ)$ defined by Eq.~(\ref{C2}),
such that $(A,\circ)$ is the associated Novikov algebra of $(A,\succ,\prec)$. This anti-pre-Novikov algebra is called 
the compatible anti-pre-Novikov algebra of $(A,\circ,\omega)$. Moreover, $(A,\succ,\prec,\omega)$ is a quadratic anti-pre-Novikov algebra.
Conversely, assume that $(A,\succ,\prec,\omega)$ is a quadratic anti-pre-Novikov algebra.
Then $(A,\circ,\omega)$ is a symmetric quasi-Frobenius Novikov algebra.
\end{thm}

\begin{proof} Let $(A,\circ,\omega)$ be a symmetric quasi-Frobenius Novikov algebra.
Define a linear map 
\begin{equation} \label{Ao11}T: A^*\rightarrow A, \ \ \  \omega(T(u),x)=\langle u,x\rangle,\ \ \ u\in A^*,x\in A.\end{equation}
Due to $\omega$ being non-degenerate, $T$ is
invertible.
For all $x,y,z\in A$, we have
\begin{align*}
\omega (x\succ y,z)&=\omega(y,x\circ z+z\circ x)\\
&=\langle T^{-1}(y), x\circ z+z\circ x \rangle =\langle T^{-1}(y), (L_{\circ}+R_{\circ})(x)z \rangle\\
&=-\langle (L_{\circ}^*+R_{\circ}^*)(x)T^{-1}(y),z \rangle=-\omega(T((L_{\circ}^*+R_{\circ}^*)(x)T^{-1}(y)),z),
\end{align*}
and
\begin{align*}
\omega (x\prec y,z)&=-\omega (x,z\circ y)=-\langle T^{-1}(x),z\circ y \rangle\\
&=-\langle T^{-1}(x),R_{\circ}(y)z\rangle=\langle R_{\circ}^*(y)T^{-1}(x),z\rangle\\
&=-\omega(T((-R_{\circ}^*)(b)T^{-1}(a)),c).
\end{align*}
Thus,
\begin{eqnarray}
x\succ y=-T((L_{\circ}^*+R_{\circ}^*)(x)T^{-1}(y)),\;\;x\prec y=-T((-R_{\circ}^*)(y)T^{-1}(x)),\;\;x,y\in A.
\end{eqnarray}
Let $x=T(u),\ y=T(v)$. Define
\begin{align*}
u\rhd v :=-(L_{\circ}^*+R_{\circ}^*)(T(u))v, \ \ \  \ u\lhd v:=-(-R_{\circ}^*)(T(v))u,\quad u,v\in A^*.
\end{align*}
Then we have
\begin{align*}
x\succ y=T(u)\succ T(v)=T(u\rhd v),\ \ \ \ x\prec y=T(u)\prec T(v)=T(u\lhd v),\;\;x,y\in A.
\end{align*}
If $(A^{*}, \rhd,\lhd)$ is an anti-pre-Novikov algebra, 
then $(A, \succ,\prec)$ is an anti-pre-Novikov algebra and $T$ is an anti-pre-Novikov algebra isomorphism. 
In fact, by Eq. (\ref{Qn}), for all $u,v,w\in A^*$, we have
\begin{align*}
&\langle w,T(u)\circ T(v)+T((L_{\circ}^*+R_{\circ}^*)(T(u))v-R_{\circ}^*(T(v))u)\rangle\\
=&\omega(T(w),T(u)\circ T(v))+\omega(T((L_{\circ}^*+R_{\circ}^*)(T(u))v-R_{\circ}^*(T(v))u),T(w))\\
=&\omega(T(w),T(u)\circ T(v))+\langle(L_{\circ}^*+R_{\circ}^*)(T(u))v-R_{\circ}^*(T(v))u, T(w)\rangle\\
=&\omega(T(w),T(u)\circ T(v))-\langle v,T(u)\circ T(w)+ T(w)\circ T(u)\rangle+\langle u ,T(w)\circ T(v)\rangle\\
=&\omega(T(w),T(u)\circ T(v))-\omega(T(v),T(u)\circ T(w)+ T(w)\circ T(u))+\omega(T(u) ,T(w)\circ T(v))\\
=&0,
\end{align*}
which yields that
\begin{align*}
T(u)\circ T(v)+T((L_{\circ}^*+R_{\circ}^*)(T(u))v-R_{\circ}^*(T(v))u)=0,
\end{align*}
that is, $T:A^*\rightarrow A$ defined by Eq. \eqref{Ao11} is an invertible anti-$\mathcal{O}$-operator on $(A,\circ )$ associated to $(A^*,L_{\circ}^*+R_{\circ}^*,-R_{\circ}^*)$. 
By Proposition \ref{Am1}, $(A^{*}, \rhd,\lhd)$ is an anti-pre-Novikov algebra.
It follows that $(A,\succ,\prec,\omega)$ is a quadratic anti-pre-Novikov algebra.
Furthermore,
\begin{eqnarray*}
x\succ y+x\prec y&=&-T((L_{\circ}^*+R_{\circ}^*)(x)T^{-1}(y))+T((R_{\circ}^*)(y)T^{-1}(x))\\
&=&T(u)\circ T(v)=x\circ y.
\end{eqnarray*}
Thus, $(A,\circ)$ is the associated Novikov algebra of $(A, \succ,\prec)$.
The converse part is obviously.
We complete the proof.
\end{proof}

\begin{defi} Let $(A_{1},\circ_1)$ be a Novikov algebra. Suppose that there is a Novikov algebra $\ast$ on the dual space
$A_{1}^{*}$. We call $(A,\circ,\omega)$ a double construction of symmetric quasi-Frobenius Novikov algebras if it satisfies the conditions:
\begin{enumerate}
\item $ A=A_1\oplus A_{1}^{*}$ as direct sum of vector spaces.
\item $A$ is a Novikov algebra and $A_1, A_{1}^{*}$ are subalgebras of $A$.
\item $\omega$ is the natural symmetric bilinear form on $A_1\oplus A_{1}^{*}$ and $\omega$ satisfies Eq. (\ref{Qn}), where $\omega$
is given by 
\begin{equation}\label{Co1}\omega(x+a, y+b) =\langle x,b\rangle+\langle a,y\rangle, \forall~x, y\in A_1,~a, b\in A_{1}^{*}\end{equation}

\end{enumerate}
Denote it by $(A_1\oplus A_{1}^{*},A_1, A_{1}^{*},\omega)$. 
\end{defi}

\begin{pro}
 Let $(A\oplus A^{*},A, A^{*},\omega)$ be a double construction of symmetric quasi-Frobenius Novikov algebras.
  Then there is a compatible anti-pre-Novikov algebra structure $(\succ,\prec)$ on $A\oplus A^{*}$ given by Eq.~\eqref{C2}.
Moreover, $(A,\succ_A,\prec_A)$ and $(A^{*},\succ_{A^{*}},\prec_{A^{*}})$ 
are anti-pre-Novikov subalgebras whose associated Novikov algebras are $(A,\circ)$ and $(A^{*},\ast)$ respectively, where
$\succ_A=\succ|_{A\otimes A},~\prec_A=\prec|_{A\otimes A}$ and
 $\succ_{A^{*}}=\succ|_{A^{*}\otimes A^{*}},\prec_{A^{*}}=\prec|_{A^{*}\otimes A^{*}}$.
\end{pro}

\begin{proof}
The first part is given by Theorem \ref{Am3}. For all $x,y\in A$, since $x\succ y\in A\oplus A^{*}$, suppose that
$x\succ y=u+u^{*}$ with $u\in A,u^{*}\in A^{*}$. Then 
\begin{equation*} \langle u^{*}, z\rangle=\omega(x\succ_{A}y, z)=\omega(y,x\circ z+z\circ x)=0
,~\forall~ z \in A,\end{equation*}
which indicates that $u^{*}=0$, that is, $x\succ_{A}y\in A$. Analogously,
$x\prec_{A}y\in A$. Thus, $(A,\succ_A,\prec_A)$ is a subalgebra of $A\oplus A^{*}$.
By the same token, $(A^{*},\succ_{A^{*}},\prec_{A^{*}})$ is a subalgebra of $A\oplus A^{*}$.
It is straightforward to prove that the associated Novikov algebras of 
$(A,\succ_A,\prec_A)$ and $(A^{*},\succ_{A^{*}},\prec_{A^{*}})$ 
are $(A,\circ)$ and $(A^{*},\ast)$ respectively.
\end{proof}

\begin{thm} \label{Mp1}
 Let $(A,\succ_{A},\prec_{A})$ and $(A^{*},\succ_{A^{*}},\prec_{A^{*}})$ be two anti-pre-Novikov algebras and their associated
 Novikov algebras be $(A,\circ)$ and $(A^{*},\ast)$ respectively. Then the following conditions are equivalent:
\begin{enumerate}
\item There is a double construction $(A\oplus A^{*},A, A^{*},\omega)$ of symmetric quasi-Frobenius Novikov algebras
 such that the compatible anti-pre-Novikov algebra $(A\oplus A^{*},\succ,\prec)$ defined by Eq. \eqref{C2} contains
 $(A,\succ_{A},\prec_{A})$ and $(A^{*},\succ_{A^{*}},\prec_{A^{*}})$ as anti-pre-Novikov subalgebras.
 \item $(A,A^{*}, -(L_{\circ}^{*}+R_{\circ}^{*}),-R_{\succ_A}^{*},R_{\odot}^{*},
 R_{\circ}^{*},-(L_{\ast}^{*}+R_{\ast}^{*}),-R_{\succ_{A^*}}^{*},R_{\ominus}^{*},R_{\ast}^{*})$ 
 is a matched pair of anti-pre-Novikov algebras. 
 \item $(A,A^{*},-L_{\odot}^{*},R_{\prec_A}^{*},-L_{\ominus}^{*},R_{\prec_{A^{*}}}^{*})$ is a matched pair of
 Novikov algebras,
 \end{enumerate}
  where
 $L_{\circ}=L_{\prec_A}+L_{\succ_A},~R_{\circ}=R_{\prec_A}+R_{\succ_A},R_{\odot}=R_{\succ_A}+L_{\prec_A},L_{\odot}=L_{\succ_A}+R_{\prec_A},
 L_{\ast}= L_{\prec_{A^*}}+L_{\succ_{A^*}},R_{\ast}= R_{\prec_{A^*}}+R_{\succ_{A^*}},L_{\ominus}=L_{\succ_{A^*}}+R_{\prec_{A^*}},R_{\ominus}=R_{\succ_{A^*}}+L_{\prec_{A^*}}.$ 
\end{thm}
\begin{proof}
$(a)\Longrightarrow (b)$ By Proposition \ref{M0}, there are linear maps 
$l_{\succ_A},r_{\succ_A},l_{\prec_A},r_{\prec_A}:A\longrightarrow \hbox{End}(A^{*})$ and
$l_{\succ_{A^{*}}},r_{\succ_{A^{*}}},l_{\prec_{A^{*}}},r_{\prec_{A^{*}}}:A^{*}\longrightarrow \hbox{End}(A)$ such that
$(A,A^{*},l_{\succ_A},r_{\succ_A},l_{\prec_A},r_{\prec_A}, l_{\succ_{A^{*}}},r_{\succ_{A^{*}}},l_{\prec_{A^{*}}},r_{\prec_{A^{*}}})$ 
is a matched pair of anti-pre-Novikov algebras and
\begin{align*}
&x\succ b=r_{\succ_{A^{*}}}(b)x+l_{\succ_A}(x)b,\ \ \ b\succ x=l_{\succ_{A^{*}}}(b)x+r_{\succ_A}(x)b,
\\&
x\prec b=r_{\prec_{A^{*}}}(b)x+l_{\prec_A}(x)b,\ \ \ b\prec x=l_{\prec_{A^{*}}}(b)x+r_{\prec_A}(x)b,
\end{align*}
for all $x\in A$ and $b\in A^{*}$.
Then we obtain,
\begin{align*}\langle l_{\succ_A}(x)b,y\rangle &=\omega (x\succ b,y)=\omega(b,x\circ y+y\circ x)
\\&=\langle b,(L_{\succ_A}(x)+L_{\prec_A}(x))y+(R_{\succ_A}(x)+R_{\prec_A}(x))y
\rangle\\&=-\langle (L_{\succ_A}^{*}+L_{\prec_A}^{*}+R_{\succ_A}^{*}+R_{\prec_A}^{*})(x)b,y
 \rangle,\end{align*}
which indicates that
$l_{\succ_A}=-(L_{\succ_A}^{*}+L_{\prec_A}^{*}+R_{\succ_A}^{*}+R_{\prec_A}^{*})$.
Analogously, $r_{\succ_A}=-R_{\succ_A}^{*},~l_{\prec_A}=R_{\succ_A}^{*}+L_{\prec_A}^{*},~r_{\prec_A}=(R_{\prec_A}^{*}+R_{\succ_A}^{*}),~
l_{\succ_{A^{*}}}=-(L_{\succ_A^{*}}^{*}+L_{\prec_A^{*}}^{*}+R_{\succ_A^{*}}^{*}+R_{\prec_A^{*}}^{*})
,~r_{\succ_{A^{*}}}=-R_{\succ_A^{*}}^{*}
,~l_{\prec_{A^{*}}}=R_{\succ_A^{*}}^{*}+L_{\prec_A^{*}}^{*}
,~r_{\prec_{A^{*}}}=(R_{\prec_A^{*}}^{*}+R_{\succ_A^{*}}^{*})$. Thus, Item (b) holds.

$(b)\Longrightarrow (c)$ It can be obtained by Corollary \ref {Ma}.

$(c)\Longrightarrow (a)$  Assume that $(A,A^{*},-(L_{\succ_A}^{*}+R_{\prec_A}^{*}),R_{\prec_A}^{*},-(L_{\succ_A}^{*}+R_{\prec_{A^{*}}}^{*}),R_{\prec_{A^{*}}}^{*})$ is a matched pair of Novikov algebras. Then $(A\bowtie A^{*},\circ)$ is a Novikov algebra with $\circ$ given by
\begin{equation*}x\circ a=R_{\prec_{A^*}}^{*}(a)x-(L_{\succ_{A}}^{*}+R_{\prec_{A}}^{*})(x) a,\ \ \ 
a\circ x=R_{\prec_{A}}^{*}(x)a-(L_{\succ_{A^*}}^{*}+R_{\prec_{A^*}}^{*})(a) x, \ \ \forall ~x\in A, a\in A^{*}.\end{equation*}
In view of Eq. \eqref{Co1}, for all $x,y\in A, a\in A^{*}$ we get
\begin{align*}&\omega(a\circ x,y)+\omega (y\circ x,a)-\omega (a\circ y+y\circ a,x)\\
=&\omega(R_{\prec_{A}}^{*}(x)a-(L_{\succ_{A^*}}^{*}+R_{\prec_{A^*}}^{*})(a) x
,y)+\omega(y\circ x,a)\\&-\omega(R_{\prec_{A}}^{*}(y)a-(L_{\succ_{A^*}}^{*}+R_{\prec_{A^*}}^{*})(a) y+
R_{\prec_{A^*}}^{*}(a)y-(L_{\succ_{A}}^{*}+R_{\prec_{A}}^{*})(y) a,x)
\\=&\langle R_{\prec_{A}}^{*}(x)a,y\rangle+\langle y\circ x,a\rangle
+\langle L_{\succ_{A}}^{*}(y) a,x
\rangle
\\=&\langle a,-y\prec_A x\rangle+\langle y\circ x,a
\rangle-\langle a,y\succ_A x\rangle
\\=&0,\end{align*}
which means that $\omega$ satisfies Eq. \eqref{Qn}.
Therefore, $(A\oplus A^{*},A, A^{*},\omega)$ is a double construction of symmetric quasi-Frobenius Novikov algebras.
The proof is completed.
\end{proof}

\section{Anti-pre-Novikov
bialgebras  }

In this section, we introduce a notion of anti-pre-Novikov
bialgebras
 as the bialgebra structures corresponding 
   to double constructions of symmetric quasi-Novikov algebras. Both of them
are equivalent to certain matched pairs of Novikov algebras as well 
as the compatible anti-pre-Novikov algebras. The study of coboundary case leads to the introduction of the APN-YBE, whose
skew-symmetric solutions induce coboundary anti-pre-Novikov
bialgebras. 
The notion of $\mathcal O$-operators on
anti-pre-Novikov algebras is introduced to construct skew-symmetric solutions
of the APN-YBE. 

\subsection{Anti-pre-Novikov bialgebras }
\begin{defi} \label{DB1} An anti-pre-Novikov coalgebra is a triple $(A,\Delta_{\succ},\Delta_{\prec})$, where
$A$ is a vector space and $\Delta_{\succ},\Delta_{\prec}:A\longrightarrow A\otimes A$ are linear maps such that
the following conditions hold:
\begin{equation}\label{Ca1}
( \Delta\otimes I)\Delta_{\succ}-(\tau \otimes I)( \Delta\otimes I)\Delta_{\succ}=(\tau \otimes I)(I\otimes\Delta_{\succ})\Delta_{\succ}-(I\otimes\Delta_{\succ})\Delta_{\succ},\end{equation}
\begin{equation}\label{Ca2}
(I\otimes \Delta)\Delta_{\prec}=(\tau\otimes I)(\Delta_{\succ}\otimes I)\Delta_{\prec}-
( \Delta_{\prec}\otimes I)\Delta_{\prec}-(\tau \otimes I)(I\otimes\Delta_{\prec})\Delta_{\succ},\end{equation}
\begin{equation}\label{Ca3}
( \Delta\otimes I)\Delta_{\succ}=-(I \otimes \tau)( \Delta_{\succ}\otimes I)\Delta_{\prec},\end{equation}
\begin{equation}\label{Ca4}
( \Delta_{\prec}\otimes I)\Delta_{\prec}=(I \otimes \tau)( \Delta_{\prec}\otimes I)\Delta_{\prec},\end{equation}
\begin{equation}\label{Ca5}
( \Delta\otimes I)\Delta_{\prec}-(\tau\otimes I)( \Delta\otimes I)\Delta_{\prec}=(I \otimes \Delta)\Delta_{\succ}-(\tau\otimes I)(I\otimes \Delta)\Delta_{\succ},\end{equation}
where $\Delta=\Delta_{\succ}+\Delta_{\prec}$ and $\tau:A\otimes A\longrightarrow A\otimes A~~\tau(x\otimes y)=y\otimes x,~\forall~x,y\in A$.
\end{defi}

\begin{defi} \label{DB2} An anti-pre-Novikov bialgebra is a quintuple $(A,\succ,\prec,\Delta_{\succ},\Delta_{\prec})$
such that $(A,\succ,\prec)$ is an anti-pre-Novikov algebra, $(A,\Delta_{\succ},\Delta_{\prec})$ is 
an anti-pre-Novikov coalgebra and the following compatible conditions hold:
\begin{align}&
(\tau \Delta_{\prec}+\Delta_{\succ})(x\circ y)=(I\otimes L_{\circ}(x)-(L_{\succ}+2R_{\prec})(x)\otimes I)(\tau \Delta_{\prec}+\Delta_{\succ})(y)\label{B1}
\\&+(I\otimes R_{\circ}(y))(2\tau \Delta_{\prec}+\Delta_{\succ})(x)+(R_{\prec}(y)\otimes I)\tau\Delta_{\prec} (x),\nonumber\\
&\Delta_{\prec}([y,x])=( L_{\circ}(y)\otimes I-I\otimes L_{\odot}(x) )\Delta_{\prec} (x)
+( I\otimes L_{\odot}(y) -L_{\circ}(x)\otimes I)\Delta_{\prec} (y),\label{B2}\\
&\Delta(x\odot y)=((L_{\succ}+2R_{\prec})(x)\otimes I)\Delta (y)+(L_{\prec}(y)\otimes I)\Delta_{\prec}(x)
\label{B3}\\&+(I\otimes L_{\odot}(x))\Delta(y)-(I\otimes R_{\odot}(y))\Delta_{\succ}(x)
-2(I\otimes R_{\odot}(y))\tau\Delta_{\prec}(x),\nonumber\\
&(\tau\Delta-\Delta)(y\prec x)=(I\otimes L_{\prec}(y))(\Delta_{\succ}+\tau \Delta_{\prec})(x)
-(I\otimes R_{\prec}(x))\Delta(y)\label{B4}\\&+(R_{\prec}(x)\otimes I)\tau\Delta (y)
-(L_{\prec}(y)\otimes I)(\tau\Delta_{\succ}+ \Delta_{\prec})(x),\nonumber\\
&(I\otimes R_{\circ}(y)+R_{\prec}(y)\otimes I)(\Delta_{\succ}+ \tau\Delta_{\prec})(x)
=(I\otimes R_{\circ}(x)+R_{\prec}(x)\otimes I)(\Delta_{\succ}+ \tau\Delta_{\prec})(y)\label{B5},\\
&(I\otimes R_{\circ}(y))\tau\Delta_{\prec}(x)-L_{\odot}(x)\otimes I)(\Delta_{\succ}+ \tau\Delta_{\prec})(y)=\tau\Delta_{\prec})(x\circ y)
\label{B6},\\
&(R_{\odot}(y)\otimes I)\Delta_{\prec}(x)-(I\otimes L_{\odot}(x))\tau\Delta(y)=(I\otimes R_{\odot}(y)
)\tau\Delta_{\prec})(x)-( L_{\odot}(x)\otimes I)\Delta(y),\label{B7}\\
&(I\otimes R_{\odot}(y))(\Delta_{\succ}+ \tau\Delta_{\prec})(x)=(R_{\prec}(x)\otimes I)\Delta (y)-\Delta(y\prec x),\label{B8}
\end{align}
where $\circ=\succ+\prec,~[x,y]=x\circ y-y\circ x,~~x\odot y=x\succ y+y\prec x,~\Delta=\Delta_{\succ}+\Delta_{\prec}$ and $R_{\circ}=R_{\prec}+R_{\succ},
~L_{\circ}=L_{\prec}+L_{\succ},~L_{\odot}=R_{\prec}+L_{\succ},~R_{\odot}=L_{\prec}+R_{\succ}$.
\end{defi}

\begin{rmk} $(A,\Delta_{\succ},\Delta_{\prec})$ is 
an anti-pre-Novikov coalgebra 
 if and only if $(A^{*},\succ_{A^{*}},\prec_{A^{*}})$
 is an anti-pre-Novikov algebra, where $\succ_{A^{*}},\prec_{A^{*}}$ are the linear dual of $\Delta_{\succ},\Delta_{\prec}$
 respectively, that is,
 \begin{align}&\label{Dc1}\langle \Delta_{\succ}(x),\zeta\otimes \eta\rangle=\langle x,\zeta\succ_{A^{*}} \eta\rangle
\\&\label{Dc2}\langle \Delta_{\prec}(x),\zeta\otimes \eta\rangle=\langle x,\zeta\prec_{A^{*}} \eta\rangle,~\forall~x\in A,\zeta,\eta\in A^{*}.
\end{align}
 Thus, an anti-pre-Novikov bialgebra $(A,\succ,\prec,\Delta_{\succ},\Delta_{\prec})$ is sometimes
denoted by $(A,\succ,\prec,A^{*},\succ_{A^{*}},$ \ \ $\prec_{A^{*}})$, where the anti-pre-Novikov
 algebra structure $(A^{*},\succ_{A^{*}},\prec_{A^{*}})$
 on the dual space $A^{*}$
corresponds to the anti-pre-Novikov coalgebra $(A,\Delta_{\succ},\Delta_{\prec})$
through Eqs. (\ref{Dc1})-(\ref{Dc2}).
\end{rmk}

\begin{thm} \label{Mp2}
Let $(A,\succ,\prec)$ be an anti-pre-Novikov algebra and $(A,\circ )$ be the associated Novikov algebra of $(A,\succ_{A},\prec_{A})$. 
Suppose that there is an anti-pre-Novikov algebra $(A^{*}, \succ_{A^*},\succ_{A^*})$ which is induced from 
an anti-pre-Novikov coalgebra $(A, \Delta_{\prec}, \Delta_{\succ})$, whose associated Novikov algebra is denoted by $(A^{*}, \ast)$. Then $(A, A^*,-(L_{\succ_{A}}^*+R_{\prec_{A}}^*),R_{\prec_{A}}^*,-(L_{\succ_{A^{*}}}^*+R_{\prec_{A^{*}}}^*),R_{\prec_{A^{*}}}^*)$
 is a matched pair of Novikov algebras if and only if $(A,\succ,\prec,\Delta_{\succ},\Delta_{\prec})$ is an anti-pre-Novikov bialgebra.
\end{thm}

\begin{proof}  We need to prove that Eqs.~\eqref{Nm1}-\eqref{Nm8} are equivalent to
 Eqs.~\eqref{B1}-\eqref{B8}. In fact, let $l_{A}=-(L_{\succ_{A}}^*+R_{\prec_{A}}^*),~
 r_A=R_{\prec_{A}}^*,~ l_B=-(L_{\succ_{A^{*}}}^*+R_{\prec_{A^{*}}}^*),~r_B=R_{\prec_{A^{*}}}^*$
 for all $x,y\in A$ and $a,b\in A^{*}$, we obtain
\begin{align*}\langle -(L_{\succ_{A^{*}}}^*+R_{\prec_{A^{*}}}^*)(a)(x\circ y),b\rangle
&=\langle x\circ y,a\succ_A^{*}b+b\prec_A^{*}a )\rangle
\\&=\langle (\Delta_{\succ}+\tau\Delta_{\prec})(x\circ y), a\otimes b\rangle,
\end{align*}
\begin{align*}\langle -(L_{\succ_{A^{*}}}^*+R_{\prec_{A^{*}}}^*)((L_{\succ_{A}}^*+2R_{\prec_{A}}^*)(x)a)y,b\rangle
&=\langle y, (L_{\succ_{A}}^*+2R_{\prec_{A}}^*)(x)a \succ_{A^{*}}b+b \prec_{A^{*}} (L_{\succ_{A}}^*+2R_{\prec_{A}}^*)(x)a
)\rangle
\\&=\langle \Delta_{\succ}(y)+\tau\Delta_{\prec}(y),(L_{\succ_{A}}^*+2R_{\prec_{A}}^*)(x)a \otimes b\rangle
\\&=-\langle [(L_{\succ_{A}}^*+2R_{\prec_{A}}^*)(x) \otimes I](\Delta_{\succ}+\tau\Delta_{\prec})(y),a \otimes b\rangle,
\end{align*}
\begin{align*}\langle -[(L_{\succ_{A^{*}}}^*+2R_{\prec_{A^{*}}}^*)(a)x]\circ y,b\rangle
&=\langle (L_{\succ_{A^{*}}}^*+2R_{\prec_{A^{*}}}^*)(a)x,R_{\circ}^{*}(y)b\rangle
\\&=\langle -x,a \succ_{A^{*}}R_{\circ}^{*}(y)b+(2R_{\circ}^{*}(y)b)\prec_{A^{*}} a\rangle
\\&=\langle (I\otimes R_{\circ}(y))(\Delta_{\succ}+2\tau\Delta_{\prec})(x),a \otimes b\rangle,
\end{align*}
\begin{align*}\langle R_{\prec_{A^{*}}}^*)(R_{\prec_{A}}^*)(y)a)x,b\rangle
=-\langle x,b\prec_{A^{*}}(R_{\prec_{A}}^*)(y)a) \rangle
=\langle (R_{\prec_{A}}(y)\otimes I)\tau\Delta_{\prec}(x),a \otimes b\rangle,
\end{align*}
\begin{align*}\langle -x\circ[(L_{\succ_{A^{*}}}^*+R_{\prec_{A^{*}}}^*)(a)y],b\rangle
&=\langle (L_{\succ_{A^{*}}}^*+R_{\prec_{A^{*}}}^*)(a)y, L_{\circ}^{*}(x)b\rangle
\\&=-\langle y,a\succ_{A^{*}}L_{\circ}^{*}(x)b +L_{\circ}^{*}(x)b\prec_{A^{*}} a
\\&=\langle (I\otimes L_{\circ}(x))(\Delta_{\succ}+\tau\Delta_{\prec})(y),a \otimes b\rangle.
\end{align*}
Thus, Eq.~(\ref{Nm1}) $\Longleftrightarrow$ Eq.~(\ref{B1}). By the same token,
Eq. (\ref{Nm2}) $\Longleftrightarrow$ Eq. (\ref{B2}), Eq. (\ref{Nm3}) $\Longleftrightarrow$ Eq. (\ref{B3}),
 Eq. (\ref{Nm4}) $\Longleftrightarrow$ Eq. (\ref{B4}),
Eq. (\ref{Nm5}) $\Longleftrightarrow$ Eq. (\ref{B5}), Eq. (\ref{Nm6}) $\Longleftrightarrow$ Eq. (\ref{B6}),
 Eq. (\ref{Nm7}) $\Longleftrightarrow$ Eq. (\ref{B7}),
Eq. (\ref{Nm8}) $\Longleftrightarrow$ Eq. (\ref{B8}). 
\end{proof}

In the light of Theorem \ref{Mp1} and Theorem \ref{Mp2}, we have 

\begin{thm} 
Let $(A,\succ,\prec)$ be an anti-pre-Novikov algebra and $(A,\circ )$ be the associated Novikov algebra of $(A,\succ_{A},\prec_{A})$. 
Suppose that there is an anti-pre-Novikov algebra $(A^{*}, \succ_{A^*},\succ_{A^*})$ which is induced from 
an anti-pre-Novikov coalgebra $(A, \Delta_{\prec}, \Delta_{\succ})$, whose associated Novikov algebra is denoted by $(A^{*}, \ast)$. Then
the following conditions are equivalent:
\begin{enumerate}
\item There is a double construction $(A\oplus A^{*},A, A^{*},\omega)$ of symmetric quasi-Frobenius Novikov algebras
 such that the compatible anti-pre-Novikov algebra $(A\oplus A^{*},\succ,\prec)$ defined by Eq.~\eqref{C2} contains
 $(A,\succ_{A},\prec_{A})$ and $(A^{*},\succ_{A^{*}},\prec_{A^{*}})$ as anti-pre-Novikov subalgebras.
 \item $(A,A^{*}, -(L_{\circ}^{*}+R_{\circ}^{*}),-R_{\succ_A}^{*},R_{\odot}^{*},
 R_{\circ}^{*},-(L_{\ast}^{*}+R_{\ast}^{*}),-R_{\succ_{A^*}}^{*},R_{\ominus}^{*},R_{\ast}^{*})$ 
 is a matched pair of anti-pre-Novikov algebras. 
 \item $(A,A^{*},-L_{\odot}^{*},R_{\prec_A}^{*},-L_{\ominus}^{*},R_{\prec_{A^{*}}}^{*})$ is a matched pair of
 Novikov algebras.
\item  $(A,\succ,\prec,\Delta_{\succ},\Delta_{\prec})$ is an anti-pre-Novikov bialgebra.
 \end{enumerate}
\end{thm}

Let $(A, \succ,\prec\Delta_{\prec}, \Delta_{\succ})$ be
 an anti-pre-Novikov bialgebra. Then $(D=A\oplus A^{*},\succ_{D},\prec_{D})$
 is an anti-pre-Novikov algebra, where
 \begin{align}\label{Db1}(x+a)\succ_{D}(y+b)=&x\succ_A y-(L_{\succ_{A^*}}^{*}+L_{\prec_{A^*}}^{*}+R_{\succ_{A^*}}^{*}+R_{\prec_{A^*}}^{*})
(a)y-R_{\succ_{A^*}}^*
(b)x\\&+a\succ_{A^*}b-(L_{\succ_A}^{*}+L_{\prec_A}^{*}+R_{\succ_A}^{*}+R_{\prec_A}^{*}(x)b-R_{\succ_A}^{*}(y)a\nonumber
,\end{align}
\begin{align}\label{Db2}
(x+a)\prec_{D}(y+b)=&x\prec_A y+(L_{\prec_{A^*}}^{*}+R_{\succ_{A^*}}^{*})
(a)y+(R_{\succ_{A^*}}^{*}+R_{\prec_{A^*}}^{*})(b)x\\&+a\prec_{A^*}b+(L_{\prec_A}^{*}+R_{\succ_A}^{*})
(x)b+(R_{\prec_A}^{*}+R_{\succ_A}^{*})(y)a\nonumber,
\end{align}
for all $x,y\in A,a,b\in A^{*}$.
$(D=A\oplus A^{*},\succ_{D},\prec_{D})$ is called the double anti-pre-Novikov algebra.

\subsection{Coboundary anti-pre-Novikov bialgebras and the anti-pre-Novikov Yang-Baxter equations}
\begin{defi} 
An anti-pre-Novikov bialgebra $(A,\succ,\prec,\Delta_{\succ,s},\Delta_{\prec,s})$ is called coboundary 
 if $\Delta_{\succ,s},\Delta_{\prec,s}$ are defined by the following equations respectively:
 \begin{align}&\label{CB1}
\Delta_{\succ,s}(x)=(I\otimes L_{\star}(x)-L_{\succ}(x)\otimes I)s_{\succ},
\\&\label{CB2}\Delta_{\prec,s}(x)=(L_{\circ}(x)\otimes I-I\otimes L_{\odot}(x))s_{\prec},
\end{align}
 for all $x\in A$, where $s_{\prec}=\sum_{i}a_i\otimes b_i,s_{\succ}=\sum_{i}c_i\otimes d_i\in A\otimes A$
 and $\circ=\succ+\prec,~
 L_{\star}=L_{\circ}+R_{\circ},~L_{\odot}=L_{\succ}+R_{\prec}$.
\end{defi}
It is direct to show that Eqs.~\eqref{B2} and \eqref{B6} hold if $\Delta_{\succ,s},\Delta_{\prec,s}$ 
are given by Eqs.~\eqref{CB1}-\eqref{CB2}.

Let $W$ be a vector space with a binary operation $\ast$. Suppose that
 $s=\sum\limits_{i}a_i\otimes b_i,~s^{'}=\sum\limits_{i}c_i\otimes d_i \in W\otimes W$. Put
\begin{eqnarray*}
s_{12}\ast s_{13}^{'}:=\sum_{i,j}a_i\ast c_j\otimes b_i\otimes d_j,\;s_{13}\ast s_{12}^{'}:=\sum_{i,j}a_i\ast c_j\otimes d_j\otimes b_i,
\;s_{12}\ast s_{23}^{'}:=\sum_{i,j}a_i\otimes b_i\ast c_j \otimes d_j,\\
s_{13}\ast s_{23}^{'}:=\sum_{i,j}a_i\otimes c_j\otimes b_i\ast d_j,\;
 s_{21}\ast s_{13}^{'}:=\sum_{i,j}b_i\ast c_j\otimes a_i\otimes d_j,\;
s_{13}\ast s_{21}^{'}:=\sum_{i,j}a_i\ast d_j\otimes c_j\otimes b_i,\\
 s_{21}\ast s_{31}^{'}:=\sum_{i,j}b_i\ast d_j\otimes a_i\otimes c_j
 ,\; s_{21}\ast s_{23}:=\sum_{i,j}b_i\otimes a_i\ast c_j\otimes d_j,\; s_{21}\ast s_{32}^{'}:=\sum_{i,j}b_i\otimes a_i\ast d_j\otimes c_j,\\
s_{31}\ast s_{23}^{'}:=\sum_{i,j}b_i\otimes c_j\otimes a_i\ast d_j,\;s_{31}\ast s_{21}^{'}:=\sum_{i,j}b_i\ast d_j\otimes c_j\otimes a_i,\;
s_{31}\ast s_{32}^{'}:=\sum_{i,j}b_i\otimes d_j\otimes a_i\ast c_j,\\
s_{23}\ast s_{12}^{'}:=\sum_{i,j}c_j\otimes a_i\ast d_j\otimes b_i,\; s_{23}\ast s_{21}^{'}:=\sum_{i,j}d_j\otimes a_i\ast c_j\otimes b_i,\;
s_{23}\ast s_{13}^{'}:=\sum_{i,j}c_j\otimes a_i\otimes b_i\ast d_j,\\
s_{32}\ast s_{21}:=\sum_{i,j}d_j\otimes b_i\ast c_j\otimes a_i,\;\; s_{23}\ast s_{31}^{'}:=\sum_{i,j}d_j\otimes a_i\otimes b_i\ast c_j.
\end{eqnarray*}

\begin{pro} \label{DB4}Let $(A,\succ,\prec)$ be an anti-pre-Novikov algebra. Assume that
$\Delta_{\succ},\Delta_{\prec}$ defined by Eqs. \eqref{CB1}-\eqref{CB2}. Then
\begin{enumerate}
\item Eq. \eqref{Ca1} holds if and only if the following equation holds:
\begin{align}&\label{CB3}(L_{\succ}(x)\otimes I \otimes I)(s_{\succ,12}\succ s_{\succ,23}-s_{\succ,13}\star s_{\succ,23}+s_{\succ,21}\succ s_{\succ,13})
\\&+(I\otimes L_{\succ}(x)\otimes I)(s_{\succ,13}\star s_{\succ,23}
-s_{\succ,21}\succ s_{\succ,13}-s_{\succ,12}\succ s_{\succ,23})
\nonumber\\&+(I\otimes I\otimes L_{\star}(x))(s_{\succ,12}\star s_{\succ,23}
-s_{\succ,13}\succ s_{\succ,12}+s_{\succ,13}\circ s_{\prec,12}
-s_{\succ,23}\odot s_{\prec,12}\nonumber\\&+
s_{\succ,23}\succ s_{\succ,21}-s_{\succ,13}\star s_{\succ,21}
-s_{\succ,23}\circ s_{\prec,21}+s_{\succ,13}\odot s_{\prec,21}
+s_{\succ,23}\circ s_{\succ,13}-s_{\succ,13}\circ s_{\succ,23})
\nonumber\\&-I\otimes L_{\odot}(x\succ c_i)\otimes I)[( s_{\succ}-s_{\prec})\otimes d_i]
-(L_{\prec}(x\succ c_i)\otimes I\otimes I)[(\tau s_{\succ}+ s_{\prec})\otimes d_i]\nonumber\\&+\sum_i(L_{\odot}(x\succ c_i)\otimes I\otimes I
-I\otimes L_{\succ}(x\succ c_i)\otimes I)[(\tau s_{\succ}-\tau s_{\prec})\otimes d_i]
\nonumber\\&+\sum_i
(I\otimes L_{\prec}(x\succ c_i)\otimes  I)[(\tau s_{\prec}+ s_{\succ})\otimes d_i]
+(L_{\succ}(x\succ c_i)\otimes I\otimes I
=0\nonumber.\end{align}
\item Eq. \eqref{Ca2} holds if and only if the following equation holds:
\begin{align}\label{CB4}&
(L_{\circ}(x)\otimes I \otimes I)(s_{\prec,12}\circ s_{\prec,23}-s_{\prec,13}\odot s_{\prec,23}
+s_{\prec,13}\star s_{\succ,23}-s_{\prec,12}\succ s_{\succ,23}-s_{\succ,21}\circ s_{\prec,13})
\\&+(I\otimes I\otimes L_{\odot}(x))(s_{\prec,23}\odot s_{\prec,12}-s_{\prec,13}\circ s_{\prec,12}
+s_{\prec,13}\star s_{\succ,21}-s_{\prec,23}\succ s_{\succ,21}-
s_{\succ,23}\succ s_{\prec,13}\nonumber\\&+s_{\prec,13}\prec s_{\succ,23})
+(I\otimes L_{\succ}(x)\otimes I)(s_{\succ,23}\odot s_{\prec,13}-s_{\succ,21}\circ s_{\prec,13}
+s_{\prec,12}\prec s_{\prec,23})
\nonumber\\&+
\sum_{i}(L_{\circ}(x\circ a_i)\otimes I\otimes I+I\otimes L_{\succ}(x\circ a_i)\otimes I)[(\tau s_{\succ}+s_{\prec})\otimes b_i]
\nonumber\\&+
(I\otimes I\otimes L_{\odot}(x\odot b_i))[a_i\otimes (s_{\prec}- s_{\succ} )]
=0\nonumber.\end{align}
\item Eq. \eqref{Ca3} holds if and only if the following equation holds:
\begin{align}\label{CB5}&
(I \otimes I \otimes L_{\star}(x))(
s_{\succ,23}\star s_{\succ,12}-s_{\succ,13}\succ s_{\succ,12}+s_{\succ,13}\circ s_{\prec,12}-s_{\succ,23}\odot s_{\prec,12}
+s_{\succ,13}\circ s_{\prec,32})
\\&+(I  \otimes L_{\odot}(x)\otimes I)(s_{\prec,12}\succ s_{\succ,13}-s_{\prec,32}\star s_{\succ,13}-s_{\succ,12}\succ s_{\succ,23})
\nonumber\\&+(I  \otimes L_{\odot}(x\succ c_i)\otimes I-L_{\succ}(x\succ c_i)\otimes I\otimes I)[(s_{\prec}-s_{\succ})\otimes d_i]
=0\nonumber.\end{align}
\item  Eq.~\eqref{Ca4} holds if and only if the following equation holds:
\begin{align}\label{CB6}&
 (I\otimes I\otimes L_{\odot}(x))(s_{\prec,23}\odot s_{\prec,12}-s_{\prec,13}\prec s_{\prec,32}-s_{\prec,13}\circ s_{\prec,12})
\\&+(I\otimes L_{\odot}(x) \otimes I)(s_{\prec,12}\circ s_{\prec,13}+s_{\prec,12}\prec s_{\prec,23}-s_{\prec,32}\odot s_{\prec,13})
=0\nonumber.\end{align}
\item  Eq. \eqref{Ca5} holds if and only if the following equation holds:
\begin{align}\label{CB7}&
(I\otimes I\otimes L_{\odot}(x))(s_{\prec,23}\odot s_{\prec,12}-s_{\prec,13}\circ s_{\prec,12}-s_{\prec,23}\star s_{\succ,12}
+s_{\prec,13}\succ s_{\succ,12}+
s_{\prec,23}\circ s_{\prec,21}\\&-s_{\prec,13}\odot s_{\prec,21}+s_{\prec,13}\star s_{\succ,21}-s_{\prec,23}\succ s_{\succ,21}
+s_{\succ,13}\circ s_{\succ,23}-s_{\succ,23}\circ s_{\succ,13})
\nonumber\\&+
(L_{\succ}(x)\otimes I\otimes I)(s_{\succ,13}\star s_{\succ,23}-s_{\succ,12}\succ s_{\succ,23}+s_{\succ,12}\circ s_{\prec,23}
-s_{\succ,13}\odot s_{\prec,23}-s_{\succ,21}\circ s_{\prec,13})
\nonumber\\&+
( I\otimes L_{\succ}(x)\otimes I)(s_{\succ,21}\succ s_{\succ,13}-s_{\succ,13}\star s_{\succ,23}-s_{\succ,21}\circ s_{\prec,13}
+s_{\succ,23}\odot s_{\prec,13}+s_{\succ,12}\circ s_{\prec,23})
\nonumber\\&+(I\otimes  I\otimes L_{\star}(x))s_{\succ,23}\odot(s_{\succ,13}-s_{\prec,13})
+(L_{\star}(x)\otimes I\otimes I)s_{\succ,21}\succ(s_{\prec,13}-s_{\succ,13})
\nonumber\\&+\sum_i (L_{\prec}(x\circ a_i)\otimes I\otimes I)[(\tau s_{\succ}+ s_{\prec})\otimes b_i]
-(I\otimes L_{\prec}(x\circ a_i)\otimes I)[(s_{\succ}+\tau s_{\prec})\otimes b_i]
\nonumber\\&+
[L_{\succ}(x\circ a_i)\otimes I\otimes I-I\otimes L_{\odot}(x\circ a_i)\otimes I][(s_{\prec}-s_{\succ})\otimes b_i]
\nonumber\\&
+[L_{\odot}(x\circ a_i)\otimes I\otimes I+I\otimes L_{\succ}(x\circ a_i)\otimes I][(\tau s_{\prec}-\tau s_{\succ})\otimes b_i]
\nonumber\\&+[I\otimes L_{\succ}(x\star d_i)\otimes I-I\otimes I\otimes L_{\odot}(x\star d_i)][ c_i\otimes (s_{\succ}-s_{\prec})]
=0\nonumber.\end{align}
\item Eq.~\eqref{B1} holds if and only if the following equation holds:
\begin{align}\label{CB8}&
[I\otimes L_{\circ}(x\circ y)-I\otimes 2 R_{\circ}(y)L_{\circ}(x)-I\otimes L_{\circ}(x)L_{\circ}(y)-
L_{\succ}(x\circ y)\otimes I -(L_{\succ}+2R_{\prec})(x)L_{\succ}(y)\\&\otimes I
+(L_{\succ}+2R_{\prec})(x)\otimes L_{\circ}( y)+2L_{\odot}(x)\otimes R_{\circ}( y)
+L_{\succ}( y)\otimes L_{\circ}(x)]
( s_{\succ}+\tau s_{\prec})
=0.\nonumber\end{align}
\item Eq.~\eqref{B3} holds if and only if the following equation holds:
\begin{align}\label{CB9}&
[2R_{\prec}(x)\otimes L_{\odot}(y)+L_{\succ}(x)\otimes L_{\odot}(y)-L_{\succ}(y)\otimes L_{\odot}(x)
+L_{\succ}(x\odot y)\otimes I\\&-L_{\succ}(x)L_{\succ}(y)\otimes I-2R_{\prec}(x)L_{\succ}(y)\otimes I
-I\otimes L_{\odot}(x\odot y)+I\otimes L_{\odot}(x)L_{\odot}(y)](s_{\prec}-s_{\succ})
\nonumber\\&+2[I\otimes R_{\odot}(y)L_{\circ}(x)-L_{\odot}(x)\otimes R_{\odot}(y)]( s_{\succ}+\tau s_{\prec})
=0\nonumber.\end{align}
\item Eq. \eqref{B4} holds if and only if the following equation holds:
\begin{align}\label{CB10}&
[L_{\succ}(y)\otimes R_{\prec}(x)+I\otimes L_{\odot}(y\prec x)- I\otimes R_{\prec}(x)L_{\odot}(y)-L_{\succ}(y\prec x)\otimes I](s_{\succ}- s_{\prec})
\\&+[L_{\prec}(y)\otimes L_{\succ}(x)+ R_{\prec}(x)L_{\prec}(y)\otimes I- L_{\prec}(y)L_{\circ}(x)\otimes I]( \tau s_{\succ}+ s_{\prec})
\nonumber\\&+[I\otimes L_{\succ}(y\prec x)-R_{\prec}(x)\otimes L_{\succ}(y)+R_{\prec}(x)L_{\odot}(y)\otimes I-L_{\odot}(y\prec x)\otimes I]( \tau s_{\succ}- \tau s_{\prec})
\nonumber\\&+[I\otimes L_{\prec}(y)L_{\circ}(x)-I\otimes R_{\prec}(x)L_{\prec}(y)- L_{\succ}(x)\otimes L_{\prec}(y)]
( s_{\succ}+\tau s_{\prec})
=0\nonumber.\end{align}
\item Eq. \eqref{B5} holds if and only if the following equation holds:
\begin{align}\label{CB11}&
[I\otimes R_{\circ}(x)L_{\circ}(y)-I\otimes R_{\circ}(y)L_{\circ}(x)+
L_{\odot}(x)\otimes R_{\circ}(y)+R_{\prec})(x)\otimes L_{\circ}(y)-R_{\prec}(y)\otimes L_{\circ}(x)
\\&-L_{\odot}(y)\otimes R_{\circ}(x)+R_{\prec}(y)L_{\succ}(x)\otimes I-R_{\prec}(x)L_{\succ}(y)\otimes I]( s_{\succ}+\tau s_{\prec})
=0.\nonumber\end{align}
\item Eq. \eqref{B7} holds if and only if the following equation holds:
\begin{align}\label{CB12}&
(L_{\odot}(x)\otimes L_{\odot}(y)-L_{\odot}(x)L_{\succ}(y)\otimes I)(s_{\prec}-s_{\succ})
+(L_{\odot}(y)\otimes L_{\odot}(x)-I\otimes R_{\prec}(x)L_{\circ}(y))(\tau s_{\succ}-\tau s_{\prec})
\\&+(R_{\odot}(y)\otimes L_{\odot}(x))(\tau s_{\succ}+ s_{\prec})
-(L_{\odot}(x)\otimes L_{\odot}(y))( s_{\succ}+\tau s_{\prec})
=0.\nonumber\end{align}
\item Eq. \eqref{B8} holds if and only if the following equation holds:
\begin{align}\label{CB13}&
(L_{\prec}(y\succ x)\otimes I-L_{\succ}(y\prec x)\otimes I+I\otimes L_{\odot}(y\prec x)-R_{\prec}(x)\otimes
 L_{\odot}(y))(s_{\succ}-s_{\prec})
 \\&+(I\otimes R_{\odot}(y)L_{\circ}(x)-L_{\odot}(x)\otimes R_{\odot}(y))(s_{\succ}+\tau s_{\prec})=0,\nonumber\end{align}
\end{enumerate}
where $L_{\odot}=L_{\succ}+R_{\prec},R_{\odot}=R_{\succ}+L_{\prec},~L_{\star}=L_{\circ}+R_{\circ},~\circ=\succ+\prec$.
\end{pro}

\begin{proof} In view of Eqs. \eqref{CB1}-\eqref{CB2}, we get
\begin{align*} 
&( \Delta_{s}\otimes I)\Delta_{\prec,s}-(\tau\otimes I)( \Delta_{s}\otimes I)\Delta_{\prec,s}-(I \otimes \Delta_{s})\Delta_{\succ,s}
+(\tau\otimes I)(I\otimes \Delta_{s})\Delta_{\succ,s}(x)
\\=&\sum_{i,j}[(x\circ a_i)\circ a_j]\otimes b_j\otimes b_i-
a_j\otimes[(x\circ a_i)\odot b_j]\otimes b_i
-(a_i\circ a_j)\otimes b_j\otimes (x\odot b_i)\\&+a_j\otimes (a_i \odot b_j)\otimes(x\odot b_i)
+c_j\otimes [ (x\circ a_i)\star d_j]\otimes b_i-[(x\circ a_i)\succ c_j]\otimes d_j\otimes b_i
\\&-c_j\otimes (a_i\star d_j)\otimes (x\odot b_i)+(a_i\succ c_j)\otimes d_j\otimes (x\odot b_i)
-b_j\otimes [(x\circ a_i)\circ a_j]\otimes b_i
\\&+[(x\circ a_i)\odot b_j]\otimes a_j\otimes b_i
+ b_j\otimes (a_i\circ a_j)\otimes(x\odot b_i)- (a_i \odot b_j)\otimes a_j\otimes(x\odot b_i)
\\&- [ (x\circ a_i)\star d_j]\otimes c_j\otimes b_i+ d_j\otimes [(x\circ a_i)\succ c_j]\otimes b_i
+ (a_i\star d_j)\otimes c_j\otimes(x\odot b_i)\\&- d_j\otimes (a_i\succ c_j)\otimes (x\odot b_i)
-c_i\otimes c_j\otimes [(x\star d_i)\star d_j]+ c_i\otimes [(x\star d_i)\succ c_j]\otimes d_j
\\&+(x\succ c_i)\otimes c_j \otimes ( d_i\star d_j)-(x\succ c_i)\otimes (d_i\succ c_j )\otimes d_j
-c_i\otimes[ (x\star d_i)\circ a_j]\otimes b_j\\&+ c_i\otimes a_j\otimes [(x\star d_i)\odot b_j]
+(x\succ c_i)\otimes (d_i\circ a_j) \otimes b_j-(x\succ c_i)\otimes a_j \otimes (d_i\odot b_j)
\\&+ c_j\otimes c_i\otimes[(x\star d_i)\star d_j]-  [(x\star d_i)\succ c_j]\otimes c_i\otimes d_j
-c_j \otimes(x\succ c_i)\otimes  ( d_i\star d_j)\\&+ (d_i\succ c_j )\otimes(x\succ c_i)\otimes d_j
+[ (x\star d_i)\circ a_j]\otimes c_i\otimes b_j-  a_j\otimes c_i\otimes [(x\star d_i)\odot b_j]
\\&- (d_i\circ a_j) \otimes (x\succ c_i)\otimes b_j+ a_j \otimes (x\succ c_i)\otimes(d_i\odot b_j)
\\=&P(1)+P(2)+P(3),
\end{align*}
where $\Delta_{s}=\Delta_{\succ,s}+\Delta_{\prec,s}$,
\begin{align*} 
P(1)=&\sum_{i,j}[(x\circ a_i)\circ a_j]\otimes b_j\otimes b_i-[(x\circ a_i)\succ c_j]\otimes d_j\otimes b_i
+[(x\circ a_i)\odot b_j]\otimes a_j\otimes b_i\\&- [ (x\circ a_i)\star d_j]\otimes c_j\otimes b_i
+(x\succ c_i)\otimes c_j \otimes ( d_i\star d_j)-(x\succ c_i)\otimes (d_i\succ c_j )\otimes d_j
\\&+(x\succ c_i)\otimes (d_i\circ a_j) \otimes b_j-(x\succ c_i)\otimes a_j \otimes (d_i\odot b_j)
\\&-  [(x\star d_i)\succ c_j]\otimes c_i\otimes d_j
+[ (x\star d_i)\circ a_j]\otimes c_i\otimes b_j,\end{align*}
\begin{align*} P(2)=&
\sum_{i,j}c_j\otimes [ (x\circ a_i)\star d_j]\otimes b_i-a_j\otimes[(x\circ a_i)\odot b_j]\otimes b_i
-b_j\otimes [(x\circ a_i)\circ a_j]\otimes b_i\\&
+ d_j\otimes [(x\circ a_i)\succ c_j]\otimes b_i
+ c_i\otimes [(x\star d_i)\succ c_j]\otimes d_j
\\&-c_i\otimes[ (x\star d_i)\circ a_j]\otimes b_j
-c_j \otimes(x\succ c_i)\otimes  ( d_i\star d_j)+ (d_i\succ c_j )\otimes(x\succ c_i)\otimes d_j
\\&- (d_i\circ a_j) \otimes (x\succ c_i)\otimes b_j+ a_j \otimes (x\succ c_i)\otimes(d_i\odot b_j)
,\end{align*}
\begin{align*} P(3)=&
\sum_{i,j}
a_j\otimes (a_i \odot b_j)\otimes(x\odot b_i)-(a_i\circ a_j)\otimes b_j\otimes (x\odot b_i)
-c_j\otimes (a_i\star d_j)\otimes (x\odot b_i)
\\&+(a_i\succ c_j)\otimes d_j\otimes (x\odot b_i)
+ b_j\otimes (a_i\circ a_j)\otimes(x\odot b_i)- (a_i \odot b_j)\otimes a_j\otimes(x\odot b_i)
\\&+ (a_i\star d_j)\otimes c_j\otimes(x\odot b_i)- d_j\otimes (a_i\succ c_j)\otimes (x\odot b_i)
-c_i\otimes c_j\otimes [(x\star d_i)\star d_j]\\&+ c_i\otimes a_j\otimes [(x\star d_i)\odot b_j]
+ c_i\otimes c_j\otimes[(x\star d_j)\star d_i]-  a_j\otimes c_i\otimes [(x\star d_i)\odot b_j]
,\end{align*}
Using Eqs. (\ref{Aa5}) and (\ref{Aa8}), we have
\begin{align*} 
P(1)=&\sum_{i,j}(L_{\succ}(x\circ a_i)\otimes I\otimes I)(s_{\prec}\otimes b_i)+
[(x\circ a_i)\prec a_j]\otimes b_j\otimes b_i-(L_{\succ}(x\circ a_i)\otimes I\otimes I)(s_{\succ}\otimes b_i)
\\&+(L_{\odot}(x\circ a_i)\otimes I\otimes I)(\tau s_{\prec}\otimes b_i)
-(L_{\odot}(x\circ a_i)\otimes I\otimes I)(\tau s_{\succ}\otimes b_i)\\&-d_j\odot (x\circ a_i)\otimes c_j\otimes b_i
+(L_{\succ}(x)\otimes I\otimes I)(s_{\succ,13}\star s_{\succ,23}-s_{\succ,12}\succ s_{\succ,23}
+s_{\succ,12}\circ s_{\prec,23}\\&-s_{\succ,13}\odot s_{\prec,23})
-[(L_{\star}(x)\otimes I\otimes I)s_{\succ,21}]\succ s_{\succ,13} 
\\&+[(L_{\star}(x)\otimes I\otimes I)s_{\succ,21}]\succ s_{\prec,13}
+[ (x\star d_i)\prec a_j]\otimes c_i\otimes b_j
\\=&\sum_{i,j}(L_{\succ}(x\circ a_i)\otimes I\otimes I)[(s_{\prec}-s_{\succ})\otimes b_i]
+(L_{\odot}(x\circ a_i)\otimes I\otimes I)[(\tau s_{\prec}-\tau s_{\succ})\otimes b_i]
\\&+[(L_{\star}(x)\otimes I\otimes I)s_{\succ,21}]\succ (s_{\prec,13}-s_{\succ,13} )
+(L_{\prec}(x\circ a_i)\otimes I\otimes I)[(\tau s_{\succ}+s_{\prec})\otimes b_i ]
\\&+(L_{\succ}(x)\otimes I\otimes I)(s_{\succ,13}\star s_{\succ,23}-s_{\succ,12}\succ s_{\succ,23}
+s_{\succ,12}\circ s_{\prec,23}-s_{\succ,13}\odot s_{\prec,23}-s_{\succ,21}\circ s_{\prec,13})
,\end{align*}
\begin{align*} P(2)=&
\sum_{i,j} (I\otimes L_{\odot}(x\circ a_i)\otimes I)(s_{\succ}\otimes b_i)
+c_j\otimes [ d_j\odot(x\circ a_i) ]\otimes b_i-(I\otimes L_{\odot}(x\circ a_i)\otimes I)(s_{\prec}\otimes b_i)\\&
-(I\otimes L_{\succ}(x\circ a_i)\otimes I)(\tau s_{\prec}\otimes b_i)-b_j\otimes [(x\circ a_i)\prec a_j]\otimes b_i
\\&+(I\otimes L_{\succ}(x\circ a_i)\otimes I)(\tau s_{\succ}\otimes b_i)
+(I\otimes L_{\succ}(x\star d_i)\otimes I)( c_i\otimes s_{\succ})
\\&-(I\otimes L_{\succ}(x\star d_i)\otimes I)( c_i\otimes s_{\prec})
-c_i\otimes[ (x\star d_i)\prec a_j]\otimes b_j
\\&+( I\otimes L_{\succ}(x)\otimes I) (-s_{\succ,23}\star s_{\succ,13}+s_{\succ,21}\succ s_{\succ,13}
-s_{\succ,21}\circ s_{\prec,13}+s_{\succ,23}\odot s_{\prec,13})
\\=&\sum_{i,j} (I\otimes L_{\odot}(x\circ a_i)\otimes I)[(s_{\succ}-s_{\prec})\otimes b_i]
+c_i\otimes [ d_i\odot(x\circ a_j) ]\otimes b_j\\&
-(I\otimes L_{\succ}(x\circ a_i)\otimes I)[(\tau s_{\succ}-\tau s_{\prec})\otimes b_i]
-(I\otimes L_{\prec}(x\circ a_i)\otimes I)(\tau s_{\prec}\otimes b_i)
\\&+(I\otimes L_{\succ}(x\star d_i)\otimes I)[ c_i\otimes (s_{\succ}-s_{\prec})]
-(I\otimes L_{\prec}(x\circ a_i)\otimes I)(s_{\succ}\otimes b_i)
\\&-c_i\otimes[ ( d_i\circ x)\prec a_j]\otimes b_j
\\&+( I\otimes L_{\succ}(x)\otimes I) (-s_{\succ,23}\star s_{\succ,13}+s_{\succ,21}\succ s_{\succ,13}
-s_{\succ,21}\circ s_{\prec,13}+s_{\succ,23}\odot s_{\prec,13})
\\=&\sum_{i,j} (I\otimes L_{\odot}(x\circ a_i)\otimes I)[(s_{\succ}-s_{\prec})\otimes b_i]
-(I\otimes L_{\succ}(x\circ a_i)\otimes I)[(\tau s_{\succ}-\tau s_{\prec})\otimes b_i]
\\&+(I\otimes L_{\succ}(x\star d_i)\otimes I)[ c_i\otimes (s_{\succ}-s_{\prec})]
-(I\otimes L_{\prec}(x\circ a_i)\otimes I)[(s_{\succ}+\tau s_{\prec})\otimes b_i]
\\&+( I\otimes L_{\succ}(x)\otimes I) (s_{\succ,21}\succ s_{\succ,13}-s_{\succ,23}\star s_{\succ,13}
-s_{\succ,21}\circ s_{\prec,13}+s_{\succ,23}\odot s_{\prec,13}+s_{\succ,12}\circ s_{\prec,23})
,\end{align*}
\begin{align*} P(3)=&
(I\otimes I\otimes L_{\odot}(x))(s_{\prec,23}\odot s_{\prec,12}-s_{\prec,13}\circ s_{\prec,12}-s_{\prec,23}\star s_{\succ,12}
\\&+s_{\prec,13}\succ s_{\succ,12}+
s_{\prec,23}\circ s_{\prec,21}-s_{\prec,13}\odot s_{\prec,21}+s_{\prec,13}\star s_{\succ,21}-s_{\prec,23}\succ s_{\succ,21}
)\\&-\sum_{i,j}(I\otimes I\otimes L_{\odot}(x\star d_i))(c_i\otimes s_{\succ})-c_i\otimes c_j\otimes [d_j\odot(x\star d_i)]
\\&+
(I\otimes I\otimes L_{\odot}(x\star d_i))(c_i\otimes s_{\prec})
+[(I\otimes I\otimes L_{\star}(x))s_{\succ,23}]\odot s_{\succ,13}
\\&+c_i\otimes c_j\otimes[ d_i\odot(x\star d_j)]
-[(I\otimes I\otimes L_{\star}(x))s_{\succ,23}]\odot s_{\prec,13}
\\=&(I\otimes I\otimes L_{\odot}(x))(s_{\prec,23}\odot s_{\prec,12}-s_{\prec,13}\circ s_{\prec,12}-s_{\prec,23}\star s_{\succ,12}
\\&+s_{\prec,13}\succ s_{\succ,12}+
s_{\prec,23}\circ s_{\prec,21}-s_{\prec,13}\odot s_{\prec,21}+s_{\prec,13}\star s_{\succ,21}-s_{\prec,23}\succ s_{\succ,21}
)\\&+\sum_{i,j}(I\otimes I\otimes L_{\odot}(x\star d_i))[c_i\otimes (s_{\prec}-s_{\succ})]+
[(I\otimes I\otimes L_{\star}(x))s_{\succ,23}]\odot (s_{\succ,13}-s_{\prec,13})
\\&+c_i\otimes c_j\otimes[ x\odot(d_i\circ d_j-d_j\circ d_i)]
\\=&(I\otimes I\otimes L_{\odot}(x))(s_{\prec,23}\odot s_{\prec,12}-s_{\prec,13}\circ s_{\prec,12}-s_{\prec,23}\star s_{\succ,12}
+s_{\prec,13}\succ s_{\succ,12}+
s_{\prec,23}\circ s_{\prec,21}\\&-s_{\prec,13}\odot s_{\prec,21}+s_{\prec,13}\star s_{\succ,21}-s_{\prec,23}\succ s_{\succ,21}+s_{\succ,13}\circ s_{\succ,23}-s_{\succ,23}\circ s_{\succ,13}
)\\&+\sum_{i,j}(I\otimes I\otimes L_{\odot}(x\star d_i))[c_i\otimes (s_{\prec}-s_{\succ})]+
[(I\otimes I\otimes L_{\star}(x))s_{\succ,23}]\odot (s_{\succ,13}-s_{\prec,13})
,\end{align*}
which yields that Item (e) holds.
The remaining part can be verified analogously.
\end{proof}
In particular, we have
\begin{pro} \label{DB4}Let $(A,\succ,\prec)$ be an anti-pre-Novikov algebra and $s_{\succ}=s_{\prec}=\sum_ia_i\otimes b_i$. Assume that
$\Delta_{\succ},\Delta_{\prec}$ defined by Eqs.~\eqref{CB1}-\eqref{CB2}. Then
\begin{enumerate}
\item Eq.~\eqref{Ca1} holds if and only if the following equation holds:
\begin{align}\label{CB30}&(I\otimes I\otimes L_{\star}(x))(s_{12}\odot s_{23}
+s_{13}\prec s_{12}-s_{23}\prec s_{21}-s_{21}\odot s_{13}
+s_{23}\circ s_{13}-s_{13}\circ s_{23})
\\&+[L_{\succ}(x)\otimes I \otimes I-I\otimes L_{\succ}(x)\otimes I](s_{12}\succ s_{23}+s_{21}\succ s_{13}-s_{13}\star s_{23})
\nonumber\\&+\sum_i
[I\otimes L_{\prec}(x\succ a_i)-L_{\prec}(x\succ a_i)\otimes I]( s+ \tau (s))\otimes b_i
=0\nonumber.\end{align}
\item Eq.~\eqref{Ca2} holds if and only if the following equation holds:
\begin{align}\label{CB40}&(I\otimes I\otimes L_{\odot}(x))(s_{23}\odot s_{12}-s_{13}\circ s_{12}
+s_{13}\star s_{21}-s_{23}\succ s_{21}-
s_{23}\succ s_{13}+s_{13}\prec s_{23})
\\&+[L_{\circ}(x)\otimes I \otimes I+I\otimes L_{\succ}(x)\otimes I](s_{23}\odot s_{13}-s_{21}\circ s_{13}
+s_{12}\prec s_{23})
\nonumber\\&
+\sum_{i}(L_{\circ}(x\circ a_i)\otimes I\otimes I+I\otimes L_{\succ}(x\circ a_i)\otimes I)[(\tau (s)+s)\otimes b_i]
=0\nonumber.\end{align}
\item Eq.~\eqref{Ca3} holds if and only if the following equation holds:
\begin{align}\label{CB50}
&(I \otimes I \otimes L_{\star}(x))(
s_{12}\odot s_{23}+s_{13}\prec s_{12}
+s_{13}\circ s_{32})
\\&+(I  \otimes L_{\odot}(x)\otimes I)(-s_{32}\star s_{13}+s_{12}\succ s_{13}-s_{12}\succ s_{23})
=0\nonumber.\end{align}
\item Eq.~\eqref{Ca4} holds if and only if the following equation holds:
\begin{align}\label{CB60}&
 (I\otimes I\otimes L_{\odot}(x))(s_{23}\odot s_{12}-s_{13}\prec s_{32}-s_{13}\circ s_{12})
\\&+(I\otimes L_{\odot}(x) \otimes I)(s_{12}\circ s_{13}+s_{12}\prec s_{23}-s_{32}\odot s_{13})
=0\nonumber.\end{align}
\item Eq.~\eqref{Ca5} holds if and only if the following equation holds:
\begin{align}\label{CB70}&
(I\otimes I\otimes L_{\odot}(x))(-s_{12}\odot s_{23}-s_{13}\prec s_{12}
+s_{23}\prec s_{21}+s_{21}\odot s_{13}+s_{13}\circ s_{23}-s_{23}\circ s_{13})
\\&+\sum_i
(L_{\prec}(x\circ a_i)\otimes I\otimes I-I\otimes L_{\prec}(x\circ a_i)\otimes I )[(\tau (s)+ s)\otimes b_i]
\nonumber\\&
+(L_{\succ}(x)\otimes I\otimes I)(s_{23}\odot s_{13}+s_{12}\prec s_{23}
-s_{21}\circ s_{13})
\nonumber\\&+
( I\otimes L_{\succ}(x)\otimes I)(s_{12}\circ s_{23}-s_{13}\odot s_{23}-s_{21}\prec s_{13})
=0\nonumber.\end{align}
\item Eq.~\eqref{B1} holds if and only if the following equation holds:
\begin{align}\label{CB80}&
[I\otimes L_{\circ}(x\circ y)-I\otimes 2 R_{\circ}(y)L_{\circ}(x)-I\otimes L_{\circ}(x)L_{\circ}(y)-
L_{\succ}(x\circ y)\otimes I -(L_{\succ}+2R_{\prec})(x)L_{\succ}(y)\\&\otimes I
+(L_{\succ}+2R_{\prec})(x)\otimes L_{\circ}( y)+2L_{\odot}(x)\otimes R_{\circ}( y)
+L_{\succ}( y)\otimes L_{\circ}(x)]
( s+\tau (s))
=0.\nonumber\end{align}
\item Eq.~\eqref{B3} holds if and only if the following equation holds:
\begin{align}\label{CB90}2[I\otimes R_{\odot}(y)L_{\circ}(x)-L_{\odot}(x)\otimes R_{\odot}(y)]( s+\tau (s))
=0.\end{align}
\item Eq.~\eqref{B4} holds if and only if the following equation holds:
\begin{align}\label{CB100}&
[L_{\prec}(y)\otimes L_{\succ}(x)+ R_{\prec}(x)L_{\prec}(y)\otimes I- L_{\prec}(y)L_{\circ}(x)\otimes I+I\otimes L_{\prec}(y)L_{\circ}(x)
\\&-I\otimes R_{\prec}(x)L_{\prec}(y)- L_{\succ}(x)\otimes L_{\prec}(y)]( \tau (s)+ s)
=0\nonumber.\end{align}
\item Eq.~\eqref{B5} holds if and only if the following equation holds:
\begin{align}\label{CB110}&
[I\otimes R_{\circ}(x)L_{\circ}(y)-I\otimes R_{\circ}(y)L_{\circ}(x)+
L_{\odot}(x)\otimes R_{\circ}(y)+R_{\prec})(x)\otimes L_{\circ}(y)-R_{\prec}(y)\otimes L_{\circ}(x)
\\&-L_{\odot}(y)\otimes R_{\circ}(x)+R_{\prec}(y)L_{\succ}(x)\otimes I-R_{\prec}(x)L_{\succ}(y)\otimes I]( s+\tau (s))
=0\nonumber.\end{align}
\item Eq.~\eqref{B7} holds if and only if the following equation holds:
\begin{align}\label{CB120}
[R_{\odot}(y)\otimes L_{\odot}(x)-L_{\odot}(x)\otimes L_{\odot}(y)](\tau (s)+ s)
=0.\end{align}
\item Eq.~\eqref{B8} holds if and only if the following equation holds:
\begin{align}\label{CB130}[I\otimes R_{\odot}(y)L_{\circ}(x)-L_{\odot}(x)\otimes R_{\odot}(y)](s+\tau (s))=0,\end{align}
\end{enumerate}
where $L_{\odot}=L_{\succ}+R_{\prec},R_{\odot}=R_{\succ}+L_{\prec},~L_{\star}=L_{\circ}+R_{\circ},~\circ=\succ+\prec$.
\end{pro}

\begin{thm} \label{YE3} Let $(A,\succ,\prec)$ be an anti-pre-Novikov algebra.
Suppose that
$\Delta_{\succ,s},\Delta_{\prec,s}$ defined by Eqs.~\eqref{CB1}-\eqref{CB2}. Then $(A,\succ,\prec,\Delta_{\succ,s},\Delta_{\prec,s})$
 is an anti-pre-Novikov bialgebra if and only if Eqs.~\eqref{CB3}-\eqref{CB13} hold.\end{thm}
\begin{proof} This follows from Definition \ref{DB1}, Definition \ref{DB2} and Proposition \ref{DB4}.\end{proof}

The equations in Eqs.~\eqref{CB3}-\eqref{CB13} are highly complex. 
We therefore investigate a simpler, special case where $s_{\prec}=s_{\succ}=s\in A\otimes A$ and $s$ is skew-symmetric.

The following conclusion is apparent.

\begin{pro} \label{YE4} Assume that $s_{\prec}=s_{\succ}=s\in A\otimes A$ and $s$ is skew-symmetric. Then
\begin{enumerate}
		\item Eqs.~\eqref{CB8}-\eqref{CB13} hold automatically.
\item Eq. \eqref{CB3} holds if and only if
\begin{align*}
(L_{\succ}(x)\otimes I \otimes I-I\otimes L_{\succ}(x)\otimes I)S_{2}
+(I\otimes I\otimes L_{\star}(x))S_{3}=0.
\end{align*}
\item Eq.~\eqref{CB4} holds if and only if
\begin{align*}
(L_{\circ}(x)\otimes I \otimes I+I\otimes L_{\succ}(x)\otimes I)S_{1}+(I\otimes I\otimes L_{\odot}(x))S_{4}
=0.\end{align*}
\item Eq.~\eqref{CB5} holds if and only if
\begin{align*}
(I \otimes I \otimes L_{\star}(x))
S_{5}-(I  \otimes L_{\odot}(x)\otimes I)S_{2}=0.\end{align*}
\item Eq.~\eqref{CB6} holds if and only if
\begin{align*}
 (I\otimes I\otimes L_{\odot}(x))S_{6}+(I\otimes L_{\odot}(x) \otimes I)S_{1}=0.\end{align*}
\item Eq.~\eqref{CB7} holds if and only if
\begin{align*}
(L_{\succ}(x)\otimes I\otimes I)S_{1}-(I\otimes I\otimes L_{\odot}(x))S_{3}
+( I\otimes L_{\succ}(x)\otimes I)S_{7}=0,\end{align*}
\end{enumerate}
where \begin{align*}&S_{1}=s_{12}\prec s_{23}+s_{23}\odot s_{13}+s_{12}\circ s_{13},\ \ \ 
 S_{2}=s_{12}\succ s_{23}-s_{13}\star s_{23}-s_{12}\succ s_{13},\\&
S_{3}=s_{12}\odot s_{23}+s_{13}\prec s_{12}-s_{13}\circ s_{23}+
s_{23}\prec s_{12}+s_{12}\odot s_{13}+s_{23}\circ s_{13},
\\&S_{4}=s_{23}\odot s_{12}-s_{13}\circ s_{12}-s_{13}\star s_{12}-
s_{23}\succ s_{13}+s_{13}\prec s_{23}+s_{23}\succ s_{12},\\&
S_{5}=s_{13}\prec s_{12}+s_{12}\odot s_{23}
-s_{13}\circ s_{23},\ \ \
S_{6}=s_{23}\odot s_{12}+s_{13}\prec s_{23}-s_{13}\circ s_{12},\\&
S_{7}=s_{12}\prec s_{13}-s_{13}\odot s_{23}+s_{12}\circ s_{23}.
\end{align*}
\end{pro}

\begin{rmk} \label{YE5} Assume that $\sigma_{12},\sigma_{13},\sigma_{23},\sigma_{132}:A\otimes A\otimes A\longrightarrow A\otimes A\otimes A$ 
are maps defined respectively by
\begin{align*}&\sigma_{12}(x\otimes y\otimes z)=y\otimes x\otimes z, \ \ \ \sigma_{13}(x\otimes y\otimes z)=z\otimes y\otimes x,
\\&\sigma_{23}(x\otimes y\otimes z)=x\otimes z\otimes y, \ \ \ \sigma_{132}(x\otimes y\otimes z)=z\otimes x\otimes y,~~\forall~x,y,z\in A.
 \end{align*}
If $s$ is skew-symmetric, then
\begin{align*}&S_{2}=-S_{1}-\sigma_{12}S_{1},\ \ \  S_{3}=S_{5}+\sigma_{13}S_{1},\ \ \ S_{4}=-S_{1}-2\sigma_{23}S_{1},\\
&S_{6}=-\sigma_{23}S_{1}, \ \ \ S_{5}=-\sigma_{132}S_{1}, \ \ \ S_{7}=-\sigma_{12}S_{1}.\end{align*}
\end{rmk}

\begin{proof} In view of the definitions of $\odot$ and $\star$, we obtain 
\begin{align*}S_{4}&=s_{23}\odot s_{12}-s_{13}\circ s_{12}-s_{13}\star s_{12}-
s_{23}\succ s_{13}+s_{13}\prec s_{23}+s_{23}\succ s_{12}
\\&=- s_{12}\circ s_{13}-s_{23}\odot s_{13}-s_{12}\prec s_{23}+2s_{23}\odot s_{12}-2s_{13}\circ s_{12}+s_{13}\prec s_{23}
\\&=-S_{1}-2\sigma_{23}S_{1}.\end{align*}
The other case can be verified analogously.
\end{proof}

\begin{thm} \label{BY} Let $s_{\prec}=s_{\succ}=s\in A\otimes A$ and $s$ be skew-symmetric.
Assume that
$\Delta_{\succ,s},\Delta_{\prec,s}$ defined by Eqs.~\eqref{CB1}-\eqref{CB2}.
Then $(A,\succ,\prec,\Delta_{\succ,s},\Delta_{\prec,s})$ is an anti-pre-Novikov bialgebra if and only if the
following equation holds:
\begin{equation} \label{YE6} s_{12}\circ s_{13}+s_{23}\odot s_{13}+s_{12}\prec s_{23}=0.\end{equation}
\end{thm}
\begin{proof} 
Combining Theorem \ref{YE3}, Proposition \ref{YE4} and Remark \ref{YE5}, we can get the conclusion.
\end{proof}

\begin{defi} Let $(A,\succ,\prec)$ be an anti-pre-Novikov algebra and 
$r\in A\otimes A$. Eq.~\eqref{YE6} is called the {\bf anti-pre-Novikov Yang-Baxter equation} or {\bf APN-YBE} in short. 
\end{defi}

In the following, we characterize the skew-symmetric solutions of the APN-YBE in terms of operators.

For a vector space $A$, the isomorphism $A\otimes A^{*}\simeq Hom (A^{*},A)$ identifies an element $s\in A\otimes A$ with a map
$T_{s}:A^{*}\longrightarrow A$. Explicitly, 
\begin{equation} \label{YE7} T_{s}:A^{*}\longrightarrow A,\ \ \  \langle T_{s}(\zeta),\eta\rangle=\langle s,\zeta\otimes\eta\rangle,
\ \ \ \forall~\zeta,\eta\in A^{*}.\end{equation}
It is clearly that $T_{s}^{*}=T_{\tau(s)},~T_{s+\tau(s)}^{*}=T_{s+\tau(s)}$.

\begin{pro} \label{YE80} Let $(A,\succ,\prec)$ be an anti-pre-Novikov
 algebra and 
$s=\sum_{i}a_i\otimes b_i\in A\otimes A$. Then the following conclusions hold:
\begin{enumerate}
		\item $s$ is a solution of the APN-YBE: $P(s)=s_{12}\circ s_{13}+s_{23}\odot s_{13}+s_{12}\prec s_{23}=0$
if and only if
\begin{equation*}T_{\tau(s)}(\zeta)\circ T_{\tau(s)}(\eta)=T_{\tau(s)}(L_{\odot}^*(T_{s}(\zeta))\eta+R_{\prec}^*(T_{\tau(s)}(\eta))\zeta).\end{equation*}
\item $s$ is a solution of the equation:
$ P_1(s)=s_{12}\prec s_{13}-s_{13}\odot s_{23}+s_{12}\circ s_{23}=0$
if and only if
\begin{equation*}T_{\tau(s)}(\zeta)\prec T_{\tau(s)}(\eta)=T_{\tau(s)}(R_{\circ}^*(T_{\tau(s)}(\eta))\zeta-R_{\odot}^*(T_{s}(\zeta))\eta)).\end{equation*}
\item $s$ is a solution of the equation
$P_2(s)=s_{12}\succ s_{13}+s_{13}\star s_{23}-s_{12}\succ s_{23}=0$
if and only if
\begin{equation*}T_{\tau(s)}(\zeta)\succ T_{\tau(s)}(\eta)=T_{\tau(s)}(L_{\star}^*(T_{s}(\zeta))\eta)-R_{\succ}^*(T_{\tau(s)}(\eta))\zeta).\end{equation*}
\item $s$ is a solution of the equation:
$P_3(s)=s_{13}\circ s_{23}-s_{13}\prec s_{12}-s_{12}\odot s_{23}=0$
if and only if
\begin{equation*}T_{s}(\zeta)\circ T_{s}(\eta)=T_{s}(-L_{\odot}^*(T_{s}(\zeta))\eta-R_{\prec}^*(T_{\tau(s)}(\eta))\zeta).\end{equation*}
\item $s$ is a solution of the equation:
$P_4(s)=s_{13}\circ s_{12}-s_{13}\prec s_{23}-s_{23}\odot s_{12}=0$
if and only if
\begin{equation*}T_{s}(\zeta)\prec T_{s}(\eta)=T_{s}(R_{\odot}^*(T_{s}(\zeta))\eta-R_{\circ}^*(T_{\tau(s)}(\eta))\zeta).\end{equation*}
\item $s$ is a solution of the equation:
$P_5(s)=s_{12}\star s_{23}-s_{13}\succ s_{12}-s_{13}\succ s_{23}=0$
if and only if
\begin{equation*}T_{s}(\zeta)\succ T_{s}(\eta)=T_{s}(R_{\succ}^*(T_{\tau(s)}(\eta))\zeta-L_{\star}^*(T_{s}(\zeta))\eta).\end{equation*}
\end{enumerate}
\end{pro}

\begin{proof} According to Eq.~\eqref{YE7}, for all $\zeta,\eta,\theta\in A^{*}$, we have
\begin{align*}\langle\theta\otimes \zeta\otimes \eta,s_{12}\circ s_{13}\rangle
&=\sum_{i,j}\langle \theta\otimes \zeta\otimes \eta,a_i\circ a_j\otimes b_i\otimes b_j\rangle
\\&=\sum_{i,j}\langle \theta,a_i\circ a_j\rangle\langle \zeta,b_i\rangle \langle \eta,b_j\rangle 
=\langle T_{\tau(s)}(\zeta)\circ T_{\tau(s)}(\eta),\theta \rangle
,\end{align*}
\begin{align*}\langle \theta\otimes \zeta\otimes \eta,s_{23}\odot s_{13}\rangle
&=\sum_{i,j}\langle \theta\otimes \zeta\otimes \eta,a_i \otimes a_j\otimes (b_j\odot b_i)\rangle
\\&=\sum_{i,j}\langle \zeta,a_j\rangle \langle \theta,a_i\rangle\langle \eta,b_j\odot b_i\rangle 
=\langle \eta,T_{s}(\zeta) \odot T_{s}(\theta)\rangle
\\&=-\langle L_{\odot}^{*}(T_{s}(\zeta))\eta, T_{s}(\theta)\rangle
=-\langle T_{\tau(s)}(L_{\odot}^{*}(T_{s}(\zeta))\eta), \theta\rangle,\end{align*}
 \begin{align*}\langle \theta\otimes \zeta\otimes \eta,s_{12}\prec s_{23}\rangle
 &=\sum_{i,j}\langle \theta\otimes \zeta\otimes \eta,a_i \otimes (b_i\prec a_j)\otimes b_j\rangle
 \\& =\sum_{i,j}\langle \theta,a_i\rangle\langle \eta,b_j\rangle \langle \zeta,b_i\prec a_j\rangle 
 =\langle \zeta,T_{s}(\theta)\prec T_{\tau(s)}(\eta)\rangle
  \\&=-\langle R_{\prec}^{*}(T_{\tau(s)}(\eta))\zeta,T_{s}(\theta)\rangle
 =-\langle T_{\tau(s)}(R_{\prec}^{*}(T_{\tau(s)}(\eta))\zeta),\theta\rangle.\end{align*}
 Thus, Item (a) holds. The other items can be verified similarly.
\end{proof}

Recall that an $\mathcal O$-operator $T$ on a Novikov
algebra $(A,\circ)$ associated to a representation $(V,l,r)$ is a linear map $T:V\longrightarrow A$ satisfying
$T(u)\circ T(v)=T (l(T(u))v+r(T(v))u)$ for all $u,v\in V$.

 \begin{defi} Let $(A,\succ,\prec)$ be an anti-pre-Novikov algebra and $(V,l_{\succ},r_{\succ},l_{\prec},r_{\prec})$ be
 its representation.
  An $\mathcal O$-operator $T$ on $(A,\succ,\prec)$ associated to $(V,l_{\succ},r_{\succ},l_{\prec},r_{\prec})$
   is a linear map $T:V\longrightarrow A$ satisfying
\begin{equation*}T(u)\succ T(v)=T (l_{\succ}(T(u))v+r_{\succ}(T(v))u),\ \ \ 
T(u)\prec T(v)=T (l_{\prec}(T(u))v+r_{\prec}(T(v))u),~~\forall~u,v\in V.\end{equation*}
In particular, an $\mathcal{O}$-operator $P$ on $(A,\circ )$ associated to
the representation $(A,L_{\succ},R_{\succ},L_{\prec},R_{\prec})$ is called a Rota-Baxter operator, that is, $P:A\longrightarrow A$
is a linear map satisfying 
\begin{equation}
P(x)\succ P(y)=P(P(x)\succ y)+x\succ P(y)),\\ \ P(x)\prec P(y)=P(P(x)\prec y)+x\prec P(y)). \end{equation}
More generally, a linear map $P:A\longrightarrow A$ is called a Rota-Baxter of weight $\lambda$ on an anti-pre-Novikov algebra
$(A,\succ,\prec)$
if
\begin{align*} &P (x) \prec P(y)=P(P(x)\prec y+x\prec P(y)+\lambda x\prec y), \\
&P (x) \succ P(y)=P(P(x)\succ y+x\succ P(y)+\lambda x\succ y),\end{align*}
for all $x, y, z \in A$.
$(A,\succ,\prec,P)$ is called a Rota-Baxter anti-pre-Novikov algebra of weight $\lambda$.\end{defi}

Direct computation shows that
$(A,\succ,\prec,P)$ is a Rota-Baxter anti-pre-Novikov algebra of weight $\lambda$
if and only if
 $(A,\succ,\prec,\tilde{P})$ is, where $\tilde{P}=-\lambda I -P$.
		
\begin{thm} \label{YE8} Let $(A,\succ,\prec)$ be an anti-pre-Novikov
 algebra,
$s=\sum_{i}a_i\otimes b_i\in A\otimes A$ be skew-symmetric and $(A,\circ)$ be the associated Novikov
 algebra of $(A,\succ,\prec)$. Then the following conditions are equivalent:
\begin{enumerate}
		\item $s$ is a solution of the APN-YBE in $(A,\succ,\prec)$, that is, 
\begin{equation*}  s_{12}\circ s_{13}+s_{23}\odot s_{13}+s_{12}\prec s_{23}=0.\end{equation*}
\item $T_{s}$ is an $\mathcal O$-operator on $(A, \circ)$ associated 
to $(A^{*},-L_{\odot}^{*},R_{\prec}^{*})$.
\item $T_{s}$ is an $\mathcal O$-operator on $(A,\succ,\prec)$ associated to $(-L_{\star}^*,-R_{\succ}^*,
R_{\odot}^*,R_{\circ}^*)$.
\end{enumerate}
\end{thm}

\begin{proof} Combining Remark \ref{YE5} and Proposition \ref{YE80}, we get the statement.
\end{proof}

\begin{thm} \label{Yo}
 Let $(A,\succ,\prec)$ be an anti-pre-Novikov algebra and $(V,l_{\succ},r_{\succ},l_{\prec},r_{\prec})$ be
  a representation of $(A,\succ,\prec)$. Suppose that $(V^{*},-(l_{\circ}^*+r_{\circ}^*),-r_{\succ}^*,
r_{\odot}^*,r_{\circ}^*)$ is the dual representation of $A$ given by Proposition \ref{zr}. 
Let $\hat{A}=A\ltimes V^{*}$ and
 $T:V\longrightarrow A$ be a linear map which is identifies an element in $\hat{A}\otimes \hat{A}$ through
 ($Hom(V,A)\simeq A\otimes V^{*}\subseteq \hat{A}\otimes \hat{A}$).
  Then $s=T-\tau(T)$ is a skew-symmetric solution
of the APN-YBE in the anti-pre-Novikov algebra $\hat{A}$ if and only if $T$ is an $\mathcal O$-operator
on $(A,\succ,\prec)$ associated to $(V,l_{\succ},r_{\succ},l_{\prec},r_{\prec})$, where
$r_{\circ}^{*}=r_{\prec}^{*}+r_{\succ}^{*},~l_{\circ}^{*}=l_{\prec}^{*}+l_{\succ}^{*},~$
$l_{\odot}^{*}=r_{\prec}^{*}+l_{\succ}^{*},~r_{\odot}^{*}=l_{\prec}^{*}+r_{\succ}^{*}.$
\end{thm}

\begin{proof}
For all $x+a^{*},y+b^{*}\in \hat{A}$ with $x,y\in A$ and $a^{*},b^{*}\in V^{*}$, the anti-pre-Novikov algebraic structure 
$(\succ,\prec)$ on $\hat{A}$ is defined by
\begin{align}\label{YE9}
 &(x+a^{*})\succ(y+b^{*})=x\succ y-(l_{\circ}^*+r_{\circ}^*)(x)b^{*}-r_{\succ}^{*}(y)a^{*},
 \\& \label{YE10}(x+a^{*})\prec(y+b^{*})=x\prec y+r_{\odot}^*(x)b^{*}+r_{\circ}^*(y)a^{*},\end{align}
and the associated Novikov algebraic structure $\circ$ on $\hat{A}$ is given by
 \begin{align*}
 (x+a^{*})\circ(y+b^{*})=x\circ y-l_{\odot}^*(x)b^{*}+r_{\prec}^{*}(y)a^{*}.\end{align*}
Assume that $\{v_1,v_2,\cdot\cdot\cdot, v_n\}$ is a basis of $V$ and 
$\{v_1^{*},v^{*}_2,\cdot\cdot\cdot, v^{*}_n\}$ is the dual basis of $V^{*}$. Then 
$T=\sum_{i=1}^{n}T(v_i)\otimes v^{*}_i\in T(V)\otimes V^{*}\subseteq \hat{A}\otimes \hat{A}$. Note that
 \begin{align}\label{YE11}
 &l_{\succ}^{*}(T(v_{i}))v_{j}^{*}
 =\sum_{k=1}^{n}\langle -v_{j}^{*},l_{\succ}(T(v_{i})v_k\rangle v_{k}^{*}, \ \ \ r_{\succ}^{*}(T(v_{i}))v_{j}^{*}=
 \sum_{k=1}^{n}\langle -v_{j}^{*},r_{\succ}(T(v_{i})v_k\rangle v_{k}^{*},\\&
 \label{YE12} l_{\prec}^{*}(T(v_{i}))v_{j}^{*}=
 \sum_{k=1}^{n}\langle -v_{j}^{*},l_{\prec}(T(v_{i})v_k\rangle v_{k}^{*}
, \ \ \ r_{\prec}^{*}(T(v_{i}))v_{j}^{*}= \sum_{k=1}^{n}\langle -v_{j}^{*},r_{\prec}(T(v_{i})v_k\rangle v_{k}^{*}.\end{align}
 Using Eqs.~\eqref{YE9}-\eqref{YE12}, we have
 \begin{align*}
 s_{12}\circ s_{13}
 &=\sum_{i,j=1}^{n}T(v_i)\circ T(v_j)\otimes v_{i}^{*}\otimes v_{j}^{*}-
 T(v_i)\circ v_{j}^{*}\otimes v_{i}^{*}\otimes T(v_{j})
-v_{i}^{*}\circ T(v_j)\otimes T(v_{i})\otimes v_{j}^{*}
\\&=\sum_{i,j=1}^{n}T(v_i)\circ T(v_j)\otimes v_{i}^{*}\otimes v_{j}^{*}
+[l_{\odot}^{*}(T(v_i))v_{j}^{*}]\otimes v_{i}^{*}\otimes T(v_{j})
-r_{\prec}^{*}(T(v_j))v_{i}^{*}\otimes T(v_{i})\otimes v_{j}^{*}
\\&=\sum_{i,j=1}^{n}T(v_i)\circ T(v_j)\otimes v_{i}^{*}\otimes v_{j}^{*}
-v_{j}^{*}\otimes v_{i}^{*}\otimes T(l_{\odot}(T(v_i))v_{j})
+v_{i}^{*}\otimes T(r_{\prec}(T(v_j))v_{i})\otimes v_{j}^{*},\end{align*}
\begin{align*}
 s_{12}\prec s_{23}
 &=\sum_{i,j=1}^{n}T(v_i)\otimes (v_{i}^{*}\prec  T(v_j))\otimes v_{j}^{*}+
v_{i}^{*}\otimes (T(v_{i})\prec v_{j}^{*})\otimes T(v_{j})-
v_{i}^{*}\otimes (T(v_{i})\prec T(v_{j}))\otimes v_{j}^{*}
\\&=\sum_{i,j=1}^{n}T(v_i)\otimes v_{i}^{*}\otimes [r_{\circ}^*(T(v_j))v_{j}^{*}]
+v_{i}^{*}\otimes r_{\odot}^*(T(v_{i}))v_{j}^{*} \otimes T(v_{j})
-v_{i}^{*}\otimes (T(v_{i})\prec T(v_{j}))\otimes v_{j}^{*}
\\&=\sum_{i,j=1}^{n}-T(r_{\circ}(T(v_j))v_i)\otimes v_{i}^{*}\otimes v_{j}^{*}
-v_{i}^{*}\otimes v_{j}^{*} \otimes T(r_{\odot}(T(v_{i})v_{j}))
-v_{i}^{*}\otimes (T(v_{i})\prec T(v_{j}))\otimes v_{j}^{*},\end{align*}
\begin{align*}
 s_{23}\odot s_{13}
 &=\sum_{i,j=1}^{n}-T(v_j)\otimes v_{i}^{*}\otimes (T(v_i)\odot  v_{j}^{*})-
v_{j}^{*}\otimes T(v_{i})\otimes (v_{i}^{*}\odot T(v_j))
+v_{j}^{*}\otimes v_{i}^{*}\otimes (T(v_i)\odot T(v_j))
\\&=\sum_{i,j=1}^{n}T(v_j)\otimes v_{i}^{*}\otimes [l_{\circ}^{*}(T(v_i))
 v_{j}^{*}]-v_{j}^{*}\otimes T(v_{i})\otimes l_{\prec}^{*}(T(v_j))v_{i}^{*}
+v_{j}^{*}\otimes v_{i}^{*}\otimes (T(v_i)\odot T(v_j))
\\&=\sum_{i,j=1}^{n}-T(l_{\circ}(T(v_i))v_j)\otimes v_{i}^{*}\otimes 
 v_{j}^{*})+v_{j}^{*}\otimes T(l_{\prec}(T(v_j))v_{i})\otimes v_{i}^{*}
+v_{j}^{*}\otimes v_{i}^{*}\otimes (T(v_i)\odot T(v_j))
.\end{align*}
Therefore, $s=T-\tau(T)$ is a skew-symmetric solution
of the APN-YBE in the anti-pre-Novikov algebra $\hat{A}$ if and only if the following equations hold:
 \begin{align}\label{YE13}
&T(v_i)\circ T(v_j)=T(r_{\circ}(Tv_j)v_i)+T(l_{\circ}(T(v_i))v_j),
\\&\label{YE14}
 T(v_i)\odot T(v_j)=T(l_{\odot}(T(v_i))v_{j})+T(r_{\odot}(T(v_{j})v_{i}),
\\&\label{YE15}
T(v_{i})\prec T(v_{j})=T(r_{\prec}(T(v_j))v_{i})+T(l_{\prec}(T(v_i))v_{j}).\end{align}
It is easy to check that Eqs.~(\ref{YE13})-(\ref{YE15}) hold if and only if $T$ is an $\mathcal O$-operator
on $(A,\succ,\prec)$ associated with $(V,l_{\succ},r_{\succ},l_{\prec},r_{\prec})$.
The proof is finished.
\end{proof}

\begin{ex} Let $(A,\succ,\prec)$ be a 2-dimensional anti-pre-Novikov algebra with a basis $\{e_1, e_2 \}$ whose
non-trivial
multiplication is given by $e_1 \succ e_1 = ae_2~(~\forall~a\in k)$. Define a linear map $T:A\longrightarrow A$ by a matrix
 $\begin {bmatrix}
 t_{1}&t_{2}\\
t_{3}&t_{4}
\end {bmatrix}$
with respect to the basis $\{e_1, e_2 \}$, where $t_i\in k~(i=1,2,3,4)$. 
Then $T$ is an $\mathcal{O}$-operator on $(A,\succ,\prec)$ associated to the representation
$(A,L_{\succ},R_{\succ},L_{\prec},R_{\prec})$ if and only if $t_{2}=0,~t_{1}(t_{1}-2t_{4})=0$.
Denote the dual basis of $A^{*}$ by $\{e_1^{*}, e_2^{*} \}$. 
The semi-direct product $A\ltimes A^{*}$ of $(A,\succ,\prec)$ and its representation $(A^{*},-(L_{\prec_A}^{*}+L_{\succ_A}^{*}+R_{\prec_A}^{*}+R_{\succ_A}^{*}),-R_{\succ_A}^{*},(R_{\succ_A}^{*}+L_{\prec_A}^{*}),
 (R_{\succ_A}^{*}+R_{\prec_A}^{*}))$ is an anti-pre-Novikov algebra with the non-trivial binary operation given by 
   \begin{flalign*}
&e_1\succ e_1=ae_2, \ \ \ e_1\succ e_2^{*}= 2ae_1^{*}, \ \ \ e_2^{*}\succ e_1= ae_1^{*},\ \ \
e_1\prec e_2^{*}= e_2^{*}\prec e_1=-ae_1^{*}.
\end{flalign*}
 In the light of Theorem \ref{Yo}, 
 \begin{align*}s&=\sum_{i,j=1}^{2}T(e_i)\otimes e_i^{*}-e_j^{*}\otimes T(e_j)\\&=(t_1e_1+t_3e_2)\otimes e_1^{*}
 +t_4e_2\otimes e_2^{*}- e_1^{*}\otimes (t_1e_1+t_3e_2)-t_4e_2^{*}\otimes e_2.
 \end{align*}
 is a skew-symmetric solution of the APN-YBE in the anti-pre-Novikov algebra $(A\ltimes A^{*},\succ,\prec)$. By Theorem \ref{BY}, 
$ (A\ltimes A^{*},\succ,\prec,\Delta_{\succ,s},\Delta_{\prec,s})$ is an anti-pre-Novikov bialgebra with the linear maps
$\Delta_{\succ,s},\Delta_{\prec,s}:A\ltimes A^{*}\longrightarrow (A\ltimes A^{*})\otimes (A\ltimes A^{*})$ defined respectively by
 \begin{align*}&
\Delta_{\succ,s}(x)=(I\otimes L_{\star}(x)-L_{\succ}(x)\otimes I)s,
\\&\Delta_{\prec,s}(x)=(L_{\circ}(x)\otimes I-I\otimes L_{\odot}(x))s,~~\forall~x\in A\ltimes A^{*}.
\end{align*}
Explicitly, 
 \begin{align*}&
\Delta_{\succ,s}(e_1)=a(t_4-t_1)e_2\otimes e_1^{*}+a(2t_4-t_1)e_1^{*}\otimes e_2,\\
&\Delta_{\succ,s}(e_2^{*})=-2at_1e_1^{*}\otimes e_1^{*},\\&
\Delta_{\prec,s}(e_1)=a(t_1-t_4)e_2\otimes e_1^{*}+a(t_1-t_4)e_1^{*}\otimes e_2,\\
&\Delta_{\succ,s}(e_2^{*})=-at_1e_1^{*}\otimes e_1^{*}.
\end{align*}
\end{ex}

\section{Quasi-triangular anti-pre-Novikov bialgebras and factorizable anti-pre-Novikov bialgebras}
 By Theorem \ref{BY}, a skew-symmetric solution of the APN-YBE gives rise to  
an (coboundary) anti-pre-Novikov bialgebra. In the present section, we examine how solutions of the APN-YBE 
produce anti-pre-Novikov bialgebras, showing these structures are no longer constrained to possess skew-symmetry.

\subsection{Quasi-triangular anti-pre-Novikov bialgebras }

\begin{defi} \label{In}
 Let $(A,\succ,\prec)$ be an anti-pre-Novikov algebra and $s\in A\otimes A$. Then $s$ is called {\bf invariant} if 
 \begin{align}&\label{IE3}
(I\otimes L_{\star}(x)-L_{\succ}(x)\otimes I)s=0,
\\&\label{IE4}(L_{\circ}(x)\otimes I-I\otimes L_{\odot}(x))s=0,~\forall~x\in A.
\end{align}
\end{defi}
 
 \begin{lem} \label{In1}
 Let $(A,\succ,\prec)$ be an anti-pre-Novikov algebra and $s\in A\otimes A$. Then $s$ is invariant if and only if
 \begin{align}&\label{IE5}
L_{\star}(x) T_{s}(\zeta)+T_{s}(L_{\succ}^{*}(x)\zeta)=0,
\\&\label{IE6}L_{\odot}(x) T_{s}(\zeta)+T_{s}(L_{\circ}^{*}(x)\zeta)=0,~~\forall~x\in A,\zeta\in A^{*}.
\end{align}
Moreover, Eqs. (\ref{IE5})-(\ref{IE6}) hold if and only if the following equations hold:
\begin{align}&\label{IE55}
L_{\star}^{*} (T_{s}(\zeta))\eta=R_{\succ}^{*}(T_{\tau(s)}(\eta))\zeta),
\\&\label{IE66}R_{\odot}^{*} (T_{s}(\zeta))\eta=R_{\circ}^{*}(T_{\tau(s)}(\eta))\zeta,~~\forall~x\in A,\zeta\in A^{*}.
\end{align}
\end{lem}
\begin{proof} 
For all $~x\in A,\zeta,\eta\in A^{*}$, we have
\begin{align*}\langle(I\otimes L_{\star}(x)-L_{\succ}(x)\otimes I)s,\zeta\otimes \eta\rangle
=&\langle s,L_{\succ}^{*}(x)\zeta\otimes \eta\rangle-\langle s,\zeta\otimes L_{\star}^{*}(x)\eta\rangle
\\=&\langle T_{s}(L_{\succ}^{*}(x))\zeta,\eta\rangle-\langle T_{s}(\zeta), L_{\star}^{*}(x)\eta\rangle
\\=&\langle T_{s}(L_{\succ}^{*}(x))\zeta,\eta\rangle+\langle L_{\star}(x)T_{s}(\zeta), \eta\rangle,\end{align*}
\begin{align*} \langle L_{\circ}(x)\otimes I-I\otimes L_{\odot}(x))s,\zeta\otimes \eta\rangle
=&\langle -s,L_{\circ}^{*}(x)\zeta\otimes \eta\rangle+\langle s,\zeta\otimes L_{\odot}^{*}(x)\eta\rangle
\\=&-\langle T_{s}(L_{\circ}^{*}(x))\zeta,\eta\rangle+\langle T_{s}(\zeta), L_{\odot}^{*}(x)\eta\rangle
\\=&-\langle T_{s}(L_{\circ}^{*}(x))\zeta,\eta\rangle-\langle L_{\odot}(x) T_{s}(\zeta), \eta\rangle.
\end{align*}
Thus, Eqs. (\ref{IE3})-(\ref{IE4}) hold if and only if Eqs. (\ref{IE5})-(\ref{IE6}) hold. Analogously,
Eqs. (\ref{IE5})-(\ref{IE6}) hold if and only if Eqs. (\ref{IE55})-(\ref{IE66}) hold. 
\end{proof}

\begin{pro} \label{Si}
 Let $(A,\succ,\prec)$ be an anti-pre-Novikov algebra and $s\in A\otimes A$. Then the following
 conditions are equivalent:
  \begin{enumerate}
\item $s+\tau(s)$ is invariant.
\item The following equations hold:
\begin{align}&\label{IE7}T_{s+\tau(s)}(L_{\succ}^{*}(x)\zeta)+
L_{\star}(x)T_{s+\tau(s)}(\zeta)=0, \ \ \ T_{s+\tau(s)}(L_{\circ}^{*}(x)\zeta)+
L_{\odot}(x)T_{s+\tau(s)}(\zeta)=0.\end{align} 
\item The following equations hold:
\begin{align}&\label{IE8}L_{\succ}(x)T_{s+\tau(s)}(\zeta)+T_{s+\tau(s)}(L_{\star}^{*}(x)\zeta)=0,
 \ \ \ T_{s+\tau(s)}(L_{\odot}^{*}(x)\zeta)+L_{\circ}(x)T_{s+\tau(s)}(\zeta)=0.\end{align}
\item The following equations hold:
\begin{align}&\label{LR1} 
L_{\star}^{*} (T_{s+\tau(s)}(\zeta))\eta=R_{\succ}^{*}(T_{s+\tau(s)}(\eta))\zeta,\ \ \
R_{\odot}^{*} (T_{s+\tau(s)}(\zeta))\eta=R_{\circ}^{*}(T_{s+\tau(s)}(\eta))\zeta,
\end{align}
\end{enumerate}
for all $x\in A,~\zeta,\eta\in A^{*}$
 \end{pro}
  
 \begin{proof} In view of Lemma \ref{In1}, $s+\tau(s)$ is invariant if and only if Eq.~(\ref{IE7}) holds.
 Since $T_{s+\tau(s)}=T_{s+\tau(s)}^{*}$, Eq.~(\ref{IE7}) holds if and only if Eq.~ (\ref{IE8}) holds.
By Eqs.~(\ref{IE55})-(\ref{IE66}), $s+\tau(s)$ is invariant if and only if  Eq.~(\ref{LR1}) holds. 
The proof is completed. 
\end{proof}
 Combining Eqs.~(\ref{IE7})-(\ref{LR1}), we obtain for all~$x\in A,~\zeta,\eta\in A^{*}$,
\begin{align}&\label{IE9}
T_{s+\tau(s)}(L_{\prec}^{*}(x)\zeta)=
R_{\odot}(x)T_{s+\tau(s)}(\zeta), \ \ \
L_{\prec}(x)T_{s+\tau(s)}(\zeta)=T_{s+\tau(s)}R_{\odot}^{*}(x)\zeta),\\
&\label{IE10}T_{s+\tau(s)}(R_{\prec}^{*}(x)\zeta)=R_{\circ}(x)T_{s+\tau(s)}(\zeta), \ \ \
R_{\succ}(x)T_{s+\tau(s)}(\zeta)=-T_{s+\tau(s)}R_{\succ}^{*}(x)\zeta),
\\
&\label{IE11}T_{s+\tau(s)}(R_{\succ}^{*}(x)\zeta)=-R_{\succ}(x)T_{s+\tau(s)}(\zeta), \ \ \
R_{\prec}(x)T_{s+\tau(s)}(\zeta)=T_{s+\tau(s)}R_{\circ}^{*}(x)\zeta),
\\&\label{LR2} 
L_{\odot}^{*} (T_{s+\tau(s)}(\zeta))\eta=-R_{\prec}^{*}(T_{s+\tau(s)}(\eta))\zeta, \ \ \
L_{\circ}^{*} (T_{s+\tau(s)}(\zeta))\eta=-L_{\prec}^{*}(T_{s+\tau(s)}(\eta))\zeta.
\end{align}

\begin{thm}\label{Ya1} Let $(A,\succ,\prec)$ be an anti-pre-Novikov
 algebra and $s=\sum_{i}a_i\otimes b_i\in A\otimes A$. 
Assume that $s+\tau(s)$ is invariant. Then the following conditions are equivalent:
\begin{enumerate}
		\item $s$ is a solution of the APN-YBE $P(s)=0$.
\item $s$ is a solution of the equation $P_3(s)=0$.
\item $s$ is a solution of the equations $P_1(s)=P_2(s)=0$. 
\item $s$ is a solution of the equations $P_4(s)=P_5(s)=0$.
\item $s$ is a solution of the equation $P(\tau(s))=s_{21}\circ s_{31}+s_{21}\prec s_{32}+s_{32}\odot s_{31}=0$,
\end{enumerate}
where the notations $P_i(s)~(i=1,2,3,4,5)$ appeared in Proposition \ref{YE80}.
\end{thm} 

\begin{proof} 
For all $\zeta,\eta,\theta\in A^{*}$,
we have
\begin{align}&\label{Ey1}\langle \zeta\otimes\eta\otimes\theta,\sum_i(R_{\prec}(a_i)\otimes I\otimes I)(\tau\otimes I)[b_i\otimes (s+\tau(s))]\rangle
\\=&-\sum_i\langle \eta \otimes R_{\prec}^{*}(a_i)\zeta\otimes\theta,b_i\otimes (s+\tau(s))\rangle
\nonumber\\=&-\langle R_{\prec}^{*}(T_{\tau(s)}(\eta))\zeta\otimes\theta,s+\tau(s)\rangle
=-\langle T_{s+\tau(s)}(R_{\prec}^{*}(T_{\tau(s)}(\eta)\zeta)),\theta\rangle\nonumber
,\\&\label{Ey2}\langle \zeta\otimes\eta\otimes\theta,\sum_ia_i\otimes(L_{\odot}(b_i)\otimes I)(s+\tau(s))\rangle
=-\langle T_{s+\tau(s)}(L_{\odot}^{*}(T_{s}(\zeta)\eta),\theta\rangle.
\end{align}
Using Eqs.~(\ref{IE8}) and (\ref{IE10}), we get
\begin{align}&\label{Ey3}\langle \zeta\otimes\eta\otimes\theta,(I\otimes I\otimes R_{\circ}(b_i))(\tau\otimes I)[a_i\otimes (s+\tau(s))]
\rangle
\\=&\langle R_{\circ}^{*}(T_{s}(\eta))T_{s+\tau(s)}(\zeta),\theta\rangle\nonumber
=\langle T_{s+\tau(s)}(R_{\prec}^{*}(T_{s}(\eta)\zeta),\theta\rangle
,\\&\label{Ey4}\langle \zeta\otimes\eta\otimes\theta,b_i\otimes(I\otimes L_{\circ}(a_i))(s+\tau(s))
\rangle
\\=&\langle L_{\circ}^{*}(T_{\tau(s)}(\zeta))T_{s+\tau(s)}(\eta),\theta\rangle\nonumber
=-\langle T_{s+\tau(s)}(L_{\odot}^{*}(T_{\tau(s)}(\zeta)\eta),\theta\rangle.
\end{align}
Combining Eqs.~(\ref{LR2}) and (\ref{Ey1})-(\ref{Ey4}), we derive
\begin{align*}&P_{3}(s)=s_{13}\circ s_{23}-s_{13}\prec s_{12}-s_{12}\odot s_{23}
 \\=&\sigma_{132}P(s)-(s_{13}+s_{31})\prec s_{12}-s_{12}\odot(s_{23}+s_{32})+(s_{13}+s_{31})\circ s_{23}-s_{31}\circ (s_{23}+s_{32})
 \\=&\sigma_{132}P(s)-\sum_i(R_{\prec}(a_i)\otimes I\otimes I)(\tau\otimes I)[b_i\otimes (s+\tau(s))]-
 a_i\otimes(L_{\odot}(b_i)\otimes I)(s+\tau(s))\\&+(I\otimes I\otimes R_{\circ}(b_i))(\tau\otimes I)[a_i\otimes (s+\tau(s))]
 - b_i\otimes(I\otimes L_{\circ}(a_i))(s+\tau(s))
 \\=&\sigma_{132}P(s).
\end{align*}
Thus, Item (a) $\Longleftrightarrow $ Item (b).
 In the following, we prove that Item (a) $\Longleftrightarrow $ Item (d). In fact,
 \begin{align*}&P_{4}(s)=s_{13}\circ s_{12}-s_{13}\prec s_{23}-s_{23}\odot s_{12}\\=&
\sigma_{23}(P(s))-s_{13}\prec(s_{23}+s_{32})-(s_{23}+s_{32})\odot s_{12}
\\=&\sigma_{23}(P(s))-\sum_ia_i\otimes(I\otimes L_{\prec}(b_i)+R_{\odot}(b_i)\otimes I)(s+\tau(s)),
\\&P_5(s)=s_{12}\star s_{23}-s_{13}\succ s_{12}-s_{13}\succ s_{23}
\\=&-\sigma_{12}[P(s)+P_4(s)]+(s_{12}+s_{21})\star s_{23}-s_{13}\succ (s_{12}+s_{21})
\\=&-\sigma_{12}[P(s)+P_4(s)]+\sum_i(I\otimes L_{\star}(a_i)-L_{\succ}\otimes I)(s+\tau(s))\otimes b_i.
\end{align*}
Combining Definition \ref{In}, Item (a) $\Longleftrightarrow $ Item (d). Item (a) $\Longleftrightarrow $Item (c) and
Item (a) $\Longleftrightarrow $ Item (e)
can be verified analogously.
\end{proof}

\begin{pro} \label{Da1}
Let $(A,\succ,\prec,\Delta_{\succ,s},\Delta_{\prec,s})$ be a coboundary anti-pre-Novikov bialgebra and $s\in A\otimes A$, 
where the comultiplications
$\Delta_{\succ,s},\Delta_{\prec,s}$ are defined by the following equations respectively:
 \begin{align}&\label{CB01}
\Delta_{\succ,s}(x)=(I\otimes L_{\star}(x)-L_{\succ}(x)\otimes I)s,
\\&\label{CB02}\Delta_{\prec,s}(x)=(L_{\circ}(x)\otimes I-I\otimes L_{\odot}(x))s,
\end{align}
for all $x\in A$. Then the anti-pre-Novikov algebra structure $(\succ_s,\prec_s)$ on $A^{*}$ is
given by 
\begin{align}&\label{IE1}\zeta\succ_{s}\eta=R_{\succ}^{*}(T_{\tau(s)}(\eta))\zeta-L_{\star}^{*}(T_{s}(\zeta))\eta,
\\&\label{IE2}\zeta\prec_{s}\eta=R_{\odot}^{*}(T_{s}(\zeta))\eta-R_{\circ}^{*}(T_{\tau(s)}(\eta))\zeta.
\end{align}
And the associated Novikov algebra structure on  $\circ_r$ on $A^{*}$ is
given by 
\begin{align}\label{IE0}\zeta\circ_{r}\eta=-R_{\prec}^{*}(T_{\tau(s)}(\eta))\zeta-L_{\odot}^{*}(T_{s}(\zeta))\eta,
\end{align}
where $\circ=\succ+\prec,~
 L_{\star}=L_{\circ}+R_{\circ},~R_{\odot}=R_{\succ}+L_{\prec}$.
\end{pro}
\begin{proof} For all $\zeta,\eta\in A^{*}$ and $x\in A$,
\begin{align*}\langle \zeta\succ_{s}\eta,x\rangle
&=\langle \zeta\otimes\eta,\Delta_{\succ,s}(x)\rangle
=\langle \zeta\otimes\eta,(I\otimes L_{\star}(x)-L_{\succ}(x)\otimes I)s\rangle
\\&=\langle  \eta\otimes L_{\succ}^{*}(x)\zeta,\tau(s)\rangle
-\langle  \zeta \otimes L_{\star}^{*}(x)\eta,s\rangle
=\langle T_{\tau(s)}(\eta),L_{\succ}^{*}(x)\zeta\rangle
-\langle T_{s}(\zeta),L_{\star}^{*}(x)\eta
\\&=-\langle x\succ T_{\tau(s)}(\eta),\zeta\rangle
+\langle x\star T_{s}(\zeta),\eta\rangle
=\langle R_{\succ}^{*}(T_{\tau(s)}(\eta))\zeta,x\rangle-\langle L_{\star}^{*}(T_{s}(\zeta))\eta,x\rangle,
\end{align*}
which indicates that Eq. (\ref{IE1}) holds.
By the same token, Eq. (\ref{IE2}) holds.
\end{proof}

\begin{thm} Let $(A,\succ,\prec)$ be an anti-pre-Novikov algebra and $s=\sum_{i}a_i\otimes b_i\in A\otimes A$. 
 Assume that
$\Delta_{\succ,s},\Delta_{\prec,s}$ are given by Eqs. (\ref{CB01})-(\ref{CB02}). If $s$ is a solution of the 
APN-YBE in $(A,\succ,\prec)$ and $s+\tau(s)$ is invariant.
Then $(A,\succ,\prec,\Delta_{\succ,s},\Delta_{\prec,s})$ is an anti-pre-Novikov bialgebra.
\end{thm}
\begin{proof}
Since $s$ is a solution of the 
APN-YBE in $(A,\succ,\prec)$ and $s+\tau(s)$ is invariant, $P(s)=0$ and
Eqs.~(\ref{IE8})-(\ref{IE11}) hold. Note that
\begin{align*}A_{11}=&(I\otimes I\otimes L_{\star}(x))(s_{12}\odot s_{23}
+s_{13}\prec s_{12}-s_{23}\prec s_{21}-s_{21}\odot s_{13}
+s_{23}\circ s_{13}-s_{13}\circ s_{23})\\=&(I\otimes I\otimes L_{\star}(x))[\sigma_{13}P(s)-\sigma_{132}P(s)+s_{12}\odot(s_{23}+s_{32})
+(s_{13}+s_{31})\prec s_{12}+(s_{13}+s_{31})\circ s_{32}\\&-s_{13}\circ(s_{23}+s_{32})
-(s_{23}+s_{32})\prec s_{21}-s_{21}\odot(s_{13}+s_{31})+s_{23}\circ(s_{13}+s_{31})
-(s_{23}+s_{32})\circ s_{31}]
\\=&\sum_{i}a_i\otimes
( L_{\odot}(b_i)\otimes L_{\star}(x))( s+ \tau (s))+
(R_{\prec}(a_i)\otimes I \otimes L_{\star}(x))(\tau\otimes I)(b_i\otimes( s+ \tau (s)))
\\&+(I\otimes I \otimes L_{\star}(x)R_{\circ}(a_i))(\tau\otimes I)(b_i\otimes( s+ \tau (s)))
-a_i\otimes
(I\otimes L_{\star}(x)L_{\circ}(b_i))( s+ \tau (s))
\\&-b_i\otimes
(R_{\prec}(b_i)\otimes L_{\star}(x))( s+ \tau (s))
-(L_{\odot}(b_i)\otimes I \otimes L_{\star}(x))(\tau\otimes I)(a_i\otimes( s+ \tau (s)))
\\&+(I\otimes I \otimes L_{\star}(x)L_{\circ}(b_i))(\tau\otimes I)(a_i\otimes( s+ \tau (s)))
-b_i\otimes
(I\otimes L_{\star}(x)R_{\circ}(a_i))( s+ \tau (s))
.\end{align*}
By Eqs.~(\ref{IE8}) and (\ref{IE10}), for all $\zeta,\eta\in A^{*}$ we obtain
\begin{align*}&\langle
( L_{\odot}(b_i)\otimes L_{\star}(x))( s+ \tau (s)),\zeta\otimes \eta\rangle
=\langle s+ \tau (s), L_{\odot}^{*}(b_i)\zeta\otimes  L_{\star}^{*}(x)\eta\rangle
\\=&\langle T_{s+ \tau (s)}( L_{\odot}^{*}(b_i)\zeta),  L_{\star}^{*}(x)\eta\rangle
=\langle L_{\star}(x)L_{\circ}(b_i)T_{s+ \tau (s)}( \zeta),  \eta\rangle,
\\&\langle
( R_{\prec}(b_i)\otimes L_{\star}(x))( s+ \tau (s)),\zeta\otimes \eta\rangle
=\langle s+ \tau (s), R_{\prec}^{*}(b_i)\zeta\otimes  L_{\star}^{*}(x)\eta\rangle
\\=&\langle T_{s+ \tau (s)}( R_{\prec}^{*}(b_i)\zeta),  L_{\star}^{*}(x)\eta\rangle
=-\langle L_{\star}(x)R_{\circ}(b_i)T_{s+ \tau (s)}(\zeta),  \eta\rangle,
\\&\langle (I \otimes L_{\star}(x)L_{\circ}(b_i))
( s+ \tau (s)),\zeta\otimes \eta\rangle
=\langle s+ \tau (s), \zeta\otimes L_{\circ}^{*}(b_i) L_{\star}^{*}(x)\eta\rangle
\\=&\langle T_{s+ \tau (s)}( \zeta),  L_{\circ}^{*}(b_i) L_{\star}^{*}(x)\eta\rangle
=\langle L_{\star}(x)L_{\circ}(b_i)T_{s+ \tau (s)}(\zeta),  \eta\rangle
,\\&\langle (I \otimes L_{\star}(x)R_{\circ}(b_i))
( s+ \tau (s)),\zeta\otimes \eta\rangle
=\langle s+ \tau (s), \zeta\otimes R_{\circ}^{*}(b_i) L_{\star}^{*}(x)\eta\rangle
\\=&\langle T_{s+ \tau (s)}( \zeta),  R_{\circ}^{*}(b_i) L_{\star}^{*}(x)\eta\rangle
=\langle L_{\star}(x)R_{\circ}(b_i)T_{s+ \tau (s)}(\zeta),  \eta\rangle
.\end{align*}
Thus, $A_{11}=0$.
Analogously,
\begin{align*}&
[L_{\succ}(x)\otimes I \otimes I-I\otimes L_{\succ}(x)\otimes I](s_{12}\succ s_{23}+s_{21}\succ s_{13}-s_{13}\star s_{23})
\\&+\sum_i
[I\otimes L_{\prec}(x\succ a_i)-L_{\prec}(x\succ a_i)\otimes I]( s+ \tau (s))\otimes b_i
=0.\end{align*}
Therefore, Eq.~(\ref{CB30}) holds. Taking the same procedure, one can prove that
Eqs.~(\ref{CB40})-(\ref{CB130}) hold.
The proof is completed.
\end{proof}

\begin{defi} \label{Qt1}
 Let $(A,\succ,\prec)$ be an anti-pre-Novikov algebra and $s\in A\otimes A$. If $s$ is a solution of the APN-YBE in $(A,\succ,\prec)$
 and $s+\tau(s)$ is invariant, then the anti-pre-Novikov
  bialgebra $(A, \succ,\prec,\Delta_{\succ,s},\Delta_{\prec,s})$ induced by $s$ is called a {\bf quasi-triangular} anti-pre-Novikov
 bialgebra. In particular, if $s$ is skew-symmetric, $(A, \succ,\prec,\Delta_{\succ,s},\Delta_{\prec,s})$
is called a {\bf triangular} anti-pre-Novikov bialgebra, where $\Delta_{\succ,s}$ and $\Delta_{\prec,s}$ 
are given by Eqs.~(\ref{CB01})-(\ref{CB02}) respectively.
\end{defi}

\begin{thm}\label{QB0}  Let $(A,\succ,\prec)$ be an anti-pre-Novikov algebra and
$s\in A\otimes A$. Suppose that $s+\tau(s)$ is invariant. Then the following conditions are equivalent:
 \begin{enumerate}
\item $s$ is a solution of the APN-YBE in $(A,\succ,\prec)$.
 \item $(A^{*},\circ_s)$ is a Novikov algebra and the linear maps
$T_{s},-T_{\tau(s)}$ are both Novikov algebra
 homomorphisms from $(A^{*},\circ_s)$ to $(A,\circ)$.
 \item $(A^{*},\succ_s,\prec_s)$ is an anti-pre-Novikov algebra and the linear maps
$T_{s},-T_{\tau(s)}$ are both anti-pre-Novikov algebra
 homomorphisms from $(A^{*},\succ_s,\prec_s)$ to $(A,\succ,\prec)$.
 \end{enumerate}
\end{thm}
\begin{proof} 
Combining Proposition \ref{YE80}, Theorem \ref{Ya1} and Proposition \ref{Da1},
 we get the conclusion.
\end{proof}

\begin{cor} \label{QB} Let $(A,\succ,\prec,\Delta_{\succ,s},\Delta_{\prec,s})$ be a quasi-triangular anti-pre-Novikov bialgebra.
 Then  \begin{enumerate}
\item $T_{s},-T_{\tau(s)}$ are both Novikov algebra
 homomorphisms from $(A^{*},\circ_s)$ to $(A,\circ)$.
 \item $T_{s},-T_{\tau(s)}$ are anti-pre-Novikov algebra
 homomorphisms from $(A^{*},\succ_s,\prec_s)$ to $(A,\succ,\prec)$.
 \end{enumerate}
\end{cor}

\begin{proof}
It follows by Theorem \ref{QB0} and Definition \ref{Qt1}.
\end{proof}

\subsection{Factorizable anti-pre-Novikov bialgebras}
In this section, we introduce the notion of factorizable anti-pre-Novikov
bialgebras, which is a special class of quasi-triangular anti-pre-Novikov bialgebras.
We show that the double of an anti-pre-Novikov bialgebra
is naturally a factorizable anti-pre-Novikov bialgebra.

\begin{defi} \label{Qt} A quasi-triangular anti-pre-Novikov
 bialgebra $(A, \succ,\prec, \Delta_{\succ,s}, \Delta_{\prec,s})$ is
called factorizable if  
 the linear map $T_{s+\tau(s)}:A^{*}\longrightarrow A$ given by Eq. (\ref{YE7}) is a linear isomorphism of vector spaces, 
where $\Delta_{\succ,s}, \Delta_{\prec,s} $
are defined by Eqs. (\ref{CB01})-(\ref{CB02}).\end{defi}

In view of Item (e) of Theorem \ref{Ya1} and Definition \ref{Qt}, we have the following statement:

\begin{pro} \label{Qf} Assume that $(A, \succ,\prec, \Delta_{\succ,s}, \Delta_{\prec,s})$ is a quasi-triangular
 (factorizable) anti-pre-Novikov bialgebra. Then $(A, \succ,\prec, \Delta_{\succ,\tau(s)}, \Delta_{\prec,\tau(s)})$
 is also a quasi-triangular
 (factorizable) anti-pre-Novikov bialgebra.\end{pro}

\begin{pro}
Let $(A, \succ,\prec, \Delta_{\succ,s}, \Delta_{\prec,s})$ be a factorizable anti-pre-Novikov bialgebra.
Then $\mathrm{Im}(T_{s}\oplus T_{\tau(s)})$ is an anti-pre-Novikov subalgebra of the direct sum anti-pre-Novikov
algebra $A\oplus A$,
which is isomorphic to the anti-pre-Novikov algebra $(A^{*},\succ_{s},\prec_{s})$. Furthermore, any $x\in A$ has a unique
decomposition $x=x_1-x_2$, where $(x_1,x_2)\in \mathrm{Im}(T_{s}\oplus T_{\tau(s)})$ and
\begin{equation*}T_{s}\oplus T_{\tau(s)}:A^{*}\longrightarrow A\oplus A,~~(T_{s}\oplus T_{\tau(s)})(\zeta)=(T_{s}(\zeta),
 -T_{\tau(s)}(\zeta)),~\forall~\zeta\in A^{*}. \end{equation*}
\end{pro}

\begin{proof} In view of Definition \ref{Qt} and Corollary \ref{QB}, $T_{s}, -T_{\tau(s)}$ are anti-pre-Novikov algebra 
homomorphisms and $T_{s+\tau(s)}$ is a linear isomorphism. It follows that
$T_{s}\oplus T_{\tau(s)}$ is an anti-pre-Novikov algebra homomorphism and $\mathrm{Ker}(T_{s}\oplus T_{\tau(s)})=0$.
Therefore, $\mathrm{Im}(T_{s}\oplus T_{\tau(s)})$ is isomorphic to $(A^{*},\succ_{s},\prec_{s})$ as anti-pre-Novikov algebras.
For all $x\in A$, 
\begin{equation*}x=T_{s+\tau(s)}T_{s+\tau(s)}^{-1}(x)=T_{s}T_{s+\tau(s)}^{-1}(x)+T_{\tau(s)}T_{s+\tau(s)}^{-1}(x)=x_1-x_2
,\end{equation*}
where $x_1=T_{s}T_{s+\tau(s)}^{-1}(x)\in \mathrm{Im}(T_{s}),
x_2=-T_{\tau(s)}T_{s+\tau(s)}^{-1}(x)\in \mathrm{Im}(-T_\tau(s))$.
The proof is finished.
\end{proof}
\begin{pro}
 Let $(A, \succ,\prec, \Delta_{\succ,s}, \Delta_{\prec,s})$ 
 be a factorizable anti-pre-Novikov bialgebra. Then the double anti-pre-Novikov algebra $(D=A\oplus A^{*}, \succ_D,\prec_D)$
is isomorphic to the direct sum $A\oplus A$ of anti-pre-Novikov algebras.
\end{pro}

\begin{proof} By Definition \ref{Qt1} and Definition \ref{Qt}, 
$T_{s+\tau(s)}$ is a linear isomorphism and $s+\tau(s)$ is invariant.
Define $\varphi:A\oplus A^{*}\longrightarrow A\oplus A$ by
\begin{equation}\label{FD1}\varphi(x,\zeta)=(x+T_{s}(\zeta),x-T_{\tau(s)}(\zeta)),~\forall~(x,\zeta)\in A\oplus A^{*}.\end{equation}
Then $\varphi$ is a bijection.
Using Eqs. (\ref{IE1}) and (\ref{LR1}), we have
\begin{align*}&\langle\eta,-R_{\succ_{A^*}}^{*}(\zeta)x-T_{s}(L_{\star_A}^{*}(x)\zeta)\rangle
= \langle \eta\succ_{s} \zeta,x\rangle+\langle x\star_{A}T_{\tau(s)}(\eta),\zeta\rangle
\\=& \langle R_{\succ}^{*}(T_{\tau(s)}(\eta))\zeta-L_{\star}^{*}(T_{s}(\zeta))\eta,x\rangle
+\langle x\star_{A}T_{\tau(s)}(\eta),\zeta\rangle
\\=& \langle R_{\succ}^{*}(T_{s+\tau(s)}(\zeta))\eta-L_{\star}^{*}(T_{s+\tau(s)}(\eta))\zeta-R_{\succ}^{*}(T_{s}(\zeta))\eta+
L_{\star}^{*}(T_{\tau(s)}(\eta))\zeta,x\rangle
-\langle L_{\star_{A}}^{*}(T_{\tau(s)}(\eta))\zeta,x\rangle
\\=& \langle\eta,x\succ_A T_{s}(\zeta)\rangle, 
\end{align*}
which implies that 
\begin{equation}\label{FD2} -R_{\succ_{A^*}}^{*}(\zeta)x-T_{s}(L_{\star_A}^{*}(x)\zeta)
=x\succ_A T_{s}(\zeta).\end{equation}
Analogously,  
\begin{equation}\label{FD3}-R_{\succ_{A^*}}^{*}(\zeta)x+T_{\tau(s)}(L_{\star_A}^{*}(x)\zeta)
=-x\succ_A T_{\tau(s)}(\zeta).\end{equation}
According to Eqs. (\ref{Db1}) and (\ref{FD1})-(\ref{FD3}),
\begin{align*}&\varphi(x\succ_D \zeta)=\varphi(-R_{\succ_A^{*}}^{*}(\zeta)x
-L_{\star_A}^{*}(x)\zeta)\\
=&(-R_{\succ_{A^*}}^{*}(\zeta)x-T_{s}(L_{\star_A}^{*}(x)\zeta),-R_{\succ_{A^*}}^{*}(\zeta)x+ T_{\tau(s)}(L_{\star_A}^{*}(x)\zeta))
\\=&(x\succ_A T_{s}(\zeta),-x\succ_A T_{\tau(s)}(\zeta))
=(x,x)\succ (T_{s}(\zeta),-T_{\tau(r)}(\zeta))
\\=&\varphi(x)\succ_A \varphi(\zeta).
\end{align*}
By the same token, $\varphi(\zeta\succ_D x)=\varphi(\zeta)\succ_A \varphi(x) $. Thus,
$\varphi((x,\zeta)\succ_D (y,\eta))=\varphi(x,\zeta)\succ \varphi(y,\eta) $ for all
$(x,\zeta), (y,\eta)\in A\oplus A^{*}$.
Taking the same procedure, we can prove that
$\varphi((x,\zeta)\prec_D (y,\eta))=\varphi(x,\zeta)\prec \varphi(y,\eta) $.
Hence, $\varphi$ is an isomorphism of anti-pre-Novikov algebras.
The proof is completed.
\end{proof}

\begin{thm} \label{Ft}
Let $(A, \succ,\prec, \Delta_{\succ}, \Delta_{\prec})$ be an anti-pre-Novikov
 bialgebra. Assume that $\{e_1,...,e_n\}$ is a basis of $A$ and $\{e^{*}_1,...,e^{*}_n\}$ is the dual basis.
 Let \begin{equation*}s=\sum_{i=1}^{n}e_{i}\otimes e_{i}^{*}\in A\otimes A^{*}\subseteq D\otimes D.\end{equation*}
Then $(D,\succ_D,\prec_D, \Delta_{\succ,s},\Delta_{\prec,s})$ with 
$\Delta_{\succ,s},\Delta_{\prec,s}$ given by Eqs.~(\ref{CB01})-(\ref{CB02})
 is a factorizable anti-pre-Novikov bialgebra.
\end{thm}
\begin{proof} Firstly, we prove that $s+\tau(s)=\sum_{i,j=1}^{n}(e_{i}\otimes e_{i}^{*}+e_{j}^{*}\otimes e_{j})$
 is invariant. In fact,
by Eqs.~(\ref{Db1}) and (\ref{Db2}), for all $x\in A$, we have 
\begin{align*}&(I\otimes L_{\star_D}(x)-L_{\succ_D}(x)\otimes I)\sum_{i,j=1}^{n}(e_{i}\otimes e_{i}^{*}+e_{i}^{*}\otimes e_{i})
\\=&\sum_{i,j=1}^{n}e_{i}\otimes x\star_{D}e_{i}^{*}-x\succ_{D}e_{i}\otimes e_{i}^{*}
+e_{j}^{*}\otimes x\star_{D}e_{j}-x\succ_{D}e_{j}^{*}\otimes e_{j}
\\=&\sum_{i,j=1}^{n}-e_{i}\otimes L_{\succ_{A^*}}^{*}(e_{i}^{*})x-e_{i}\otimes L_{\succ}^{*}(x)e_{i}^{*}
-x\succ e_{i}\otimes e_{i}^{*}+e_{j}^{*}\otimes x\star e_j+R_{\succ_{A^*}}^{*}(e_{j}^{*})x\otimes e_{j}
+L_{\star}^{*}(x)e_{j}^{*}\otimes e_{j}.
\end{align*}
Note that
\begin{align*}& \sum_{i,j=1}^{n}R_{\succ_{A^*}}^{*}(e_{j}^{*})x\otimes e_{j}-e_{i}\otimes L_{\succ_{A^*}}^{*}(e_{i}^{*})x=0,
\\&\sum_{i,j=1}^{n}e_{i}\otimes L_{\succ_A}^{*}(x)e_{i}^{*}
+x\succ_{A}e_{i}\otimes e_{i}^{*}=0,
\\&\sum_{i,j=1}^{n}e_{j}^{*}\otimes x\star e_j
+L_{\star}^{*}(x)e_{j}^{*}\otimes e_{j}=0.
\end{align*}
Thus, $(I\otimes L_{\star_{D}}(x)-L_{\succ_{D}}(x)\otimes I)(s+\tau(s))=0.$
Similarly, $(L_{\circ_{D}}(x)\otimes I-I\otimes L_{\odot_{D}}(x))(s+\tau(s))=0.$
By duality, we get
\begin{equation*}(I\otimes L_{\star_{D}}(\zeta)-L_{\succ_{D}}(\zeta)\otimes I)(s+\tau(s))=0,
\ \ \
(L_{\circ_{D}}(\zeta)\otimes I-I\otimes L_{\odot_{D}}(\zeta))(s+\tau(s))=0\end{equation*}
for all $\zeta\in A^{*}$. Therefore, according to Definition \ref{In},  $s+\tau(s)$ is invariant.
Secondly, we prove that $s$ is a solution of the APN-YBE in $(D,\succ_D,\prec_D)$.
 In view of Eqs.~(\ref{Db1}) and (\ref{Db2}), we get
\begin{align*}& s_{12}\circ_D s_{13}+s_{23}\odot_D s_{13}+s_{12}\prec_D s_{23}
\\=&\sum_{i,j=1}^{n}e_{i}\circ_D e_{j}\otimes e_{i}^{*}\otimes e_{j}^{*}+e_{j}\otimes e_{i} \otimes e_{i}^{*}\odot_D e_{j}^{*}
+e_{i}\otimes e_{i}^{*}\prec_D e_j\otimes e_{j}^{*}
\\=&\sum_{i,j=1}^{n}e_{i}\circ_A e_{j}\otimes e_{i}^{*}\otimes e_{j}^{*}+e_{j}\otimes e_{i} \otimes e_{i}^{*}\odot_{A^*} e_{j}^{*}
+e_{i} \otimes (R_{\odot_A{^*}}^{*}(e_{i}^{*})e_{j}+R_{\circ}^{*}(e_{j})e_{i}^{*})\otimes e_{j}^{*}
=0.\end{align*}
In all, $(D,\succ_D,\prec_D, \Delta_{\succ,s},\Delta_{\prec,s})$ is a quasi-triangular anti-pre-Novikov bialgebra.
 Furthermore, the linear maps $T_{s},T_{\tau(s)}:D^{*}\longrightarrow D$ are respectively defined by 
$T_{s}(\zeta,x)=\zeta,~T_{\tau(s)}(\zeta,x)=-x$ for all $x\in A,\zeta\in A^{*}$. Thus,
$T_{s+\tau(s)}(\zeta,x)=(\zeta,-x)$ is a linear isomorphism. Hence, $(D,\succ_D,\prec_D, \Delta_{\succ,s},\Delta_{\prec,s})$
is a factorizable anti-pre-Novikov bialgebra.
\end{proof}

\section{Relative Rota-Baxter operators and quadratic Rota-Baxter anti-pre-Novikov algebras }

\subsection{Relative Rota-Baxter operators and the APN-YBE}
This section interprets solutions to the APN-YBE possessing 
invariant symmetric parts via relative Rota–Baxter operators of weight on anti-pre-Novikov algebras.

\begin{defi}\label{RRa} Let $(A,\succ,\prec)$ and $(V,\succ_V,\prec_V)$ be anti-pre-Novikov algebras. Assume that
 $(V,l_{\succ},$ \ \ \ $r_{\succ},l_{\prec},r_{\prec})$ is a representation and the following conditions are satisfied:
 \begin{align*}&
 (l_{\circ}(x)a-r_{\circ}(x)a)\succ_V b=a\succ l_{\succ}(x)b-l_{\succ}(x)(a\succ_V b),\\&
 l_{\prec}(x)(a\circ_V b)=r_{\succ}(x)a\prec_V b-l_{\prec}(x)a\prec_V b-a\succ_V(l_{\prec}(x)b),\\&
 l_{\circ}(x)a\succ_V b=-l_{\succ}(x)b\prec_V a, \ \ \ l_{\prec}(x)a\prec_V b=l_{\prec}(x)b\prec_V a,
 \\
 &(l_{\circ}(x)a-r_{\circ}(x)a)\prec_V b=l_{\succ}(x)(a\circ_Vb)-a\succ_V(l_{\circ}(x)b),
 \\&r_{\succ}(x)(a\circ_V b-b\circ_V a)=b\succ_V(r_{\succ}(x)a)-a\succ_V(r_{\succ}(x)b),
 \\&
 a \prec_V(r_{\circ}(x)b)=r_{\prec}(x)(b\succ_V a -a\prec_V b)-b\succ_V(r_{\prec}(x)a),\\&
 r_{\prec}(x)(a\circ_V b-b\circ_V a)=a\succ_V(r_{\circ}(x)b)-b\succ_V(r_{\circ}(x)a),
 \\&a\prec_V(l_{\circ}(x)b)=l_{\succ}(x)a\prec_V b-r_{\prec}(x)a\prec_V b-l_{\succ}(x)(a\prec_V b),\\&
  r_{\succ}(x)(a\circ_V b)=-r_{\succ}(x)a\prec_V b, \ \   r_{\prec}(x)(a\prec_V b)=r_{\prec}(x)a\prec_V b,\ \
  \\&r_{\prec}(x)(a\succ_V b)=-r_{\circ}(x)a\succ_V b
\end{align*}
 for all $x\in A$ and $a,b\in V$. Then $(V,\succ_V,\prec_V,l_{\succ},r_{\succ},l_{\prec},r_{\prec})$ is called
  an A-anti-pre-Novikov algebra, where $\circ=\succ+\prec,~\circ_V=\succ_V+\prec_V$.
\end{defi}
 It is easy to check that $(V,\succ_V,\prec_V,l_{\succ},r_{\succ},l_{\prec},r_{\prec})$ is an A-
 anti-pre-Novikov algebra
 if and only if $(A\oplus V,\succeq,\preceq)$ is an anti-pre-Novikov algebra, where
 \begin{align*}&
 (x+a)\succeq (y+b)=x\succ y+l_{\succ}(x)b+r_{\succ}(y)a+a\succ_V b,\\&
 (x+a)\preceq (y+b)=x\prec y+l_{\prec}(x)b+r_{\prec}(y)a+a\prec_V b.
  \end{align*}

 \begin{pro}
Let $(A,\succ,\prec)$ be an anti-pre-Novikov algebra and $s\in A\otimes A$ be symmetric and invariant. Define the
multiplications $\succeq_s,\preceq_s:A^{*}\otimes A^{*}\longrightarrow A^{*}$ by
\begin{align}&\label{AD1} \zeta\succeq_s\eta=-L_{\star}^{*}(T_{s}(\zeta))\eta=-R_{\succ}^{*}(T_{s}(\eta))\zeta, 
 \\&\label{AD2}\zeta\preceq_s\eta=R_{\odot}^{*}(T_{s}(\zeta))\eta=R_{\circ}^{*}(T_{s}(\eta))\zeta
,~~\forall~\zeta,\eta\in A^{*}.
\end{align}
Then $(A^{*},\succeq_s,\preceq_s,-L_{\star}^{*},-R_{\succ}^{*},R_{\odot}^{*}, R_{\circ}^{*})$ is an A-anti-pre-Novikov algebra
and $(A^{*},\circ_s,-L_{\odot}^{*},R_{\prec}^{*})$ is an A-Novikov algebra, where
\begin{equation}\label{AD3} \zeta\circ_s\eta=-L_{\odot}^{*}(T_s(\zeta))\eta=R_{\prec}^{*}(T_s(\eta))\zeta,~~\forall~\zeta,\eta\in A^{*}.
\end{equation}
\end{pro}

\begin{proof} In view of Eqs.~(\ref{Nr4}), (\ref{IE8})-(\ref{IE9}) and (\ref{AD1})-(\ref{AD3}),
for all $\zeta,\eta,\theta\in A^{*}$ we have
\begin{align*}(\zeta\circ_s\eta-\eta\circ_s\zeta)\succeq_{s}\theta
=&-L_{\star}^{*}(T_{s}(\zeta\circ_s\eta-\eta\circ_s\zeta))\theta
\\=&L_{\star}^{*}(T_{s}(L_{\odot}^{*}(T_s(\zeta))\eta-L_{\odot}^{*}(T_s(\eta))\zeta))\theta
\\=&L_{\star}^{*}(T_s(\eta)\circ T_s(\zeta)-T_s(\zeta)\circ T_s(\eta))\theta
\\=&L_{\star}^{*}(T_s(\eta))L_{\star}^{*}( T_s(\zeta))\theta-L_{\star}^{*}(T_s(\zeta)) L_{\star}^{*}(T_s(\eta))\theta
\\=&\eta\succeq_s(\zeta\succeq_s\theta)-\zeta\succeq_s(\eta\succeq_s\theta),
\end{align*}
which yields that Eq.~(\ref{Aa1}) holds for $\succeq_s$ and $\circ_s$. Analogously,
using Eqs. (\ref{rp1})-(\ref{rp10}), (\ref{IE8})-(\ref{IE9}) and (\ref{AD1})-(\ref{AD3}), we can prove that 
Eqs. (\ref{Aa2})-(\ref{Aa5}) and all equations in Definition \ref{RRa} hold for $\succeq_s,\preceq_s$ and $\circ_s$. We finish the proof.
\end{proof}

\begin{defi} Let $(A,\succ,\prec)$ be an anti-pre-Novikov algebra and 
$(V,\succ_V,\prec_V,l_{\succ},r_{\succ},l_{\prec},r_{\prec})$ be an A-anti-pre-Novikov algebra.
A relative Rota-Baxter operator $T$ of weight $\lambda$ on $(A,\succ,\prec)$ associated to
 $(V,\succ_V,\prec_V,l_{\succ},r_{\succ},l_{\prec},r_{\prec})$
   is a linear map $T:V\longrightarrow A$ satisfying
\begin{align*}&T(u)\succ T(v)=T (l_{\succ}(T(u))v+r_{\succ}(T(v))u+\lambda u\succ_V v),\\& 
T(u)\prec T(v)=T (l_{\prec}(T(u))v+r_{\prec}(T(v))u+\lambda u\prec_V v),~~\forall~u,v\in V.\end{align*}
When $u\succ_Vv=u\prec_Vv=0$ for all $u,v\in V$, then $T$ is simply a relative Rota-Baxter operator ( $\mathcal O$-operator) on
$(A,\succ,\prec)$ associated to a representation $(V,l_{\succ},r_{\succ},l_{\prec},r_{\prec})$.
\end{defi}

The following result ia a generalization of Theorem \ref{YE8}.
\begin{thm}
Let $(A,\succ,\prec)$ be an anti-pre-Novikov algebra and $s\in A\otimes A$. Assume that $s+\tau(s)$ is invariant.
Then the following
conditions are equivalent.
 \begin{enumerate}
\item $s$ is a solution of the APN-YBE in $(A,\succ,\prec)$
 such that $(A,\succ,\prec,\Delta_{\succ,s},\Delta_{\prec,s})$
with $\Delta_{\succ,s},\Delta_{\prec,s}$ given by Eqs. (\ref{CB01})-(\ref{CB02}) is a quasi-triangular anti-pre-Novikov bialgebra.
\item $T_s$ is a relative Rota–Baxter operator of weight $-1$ on $(A,\succ,\prec)$
 with respect to the A-anti-pre-Novikov algebra $(A^{*},\succeq_{s+\tau(s)},\preceq_{s+\tau(s)},-L_{\star}^{*},-R_{\succ}^{*},R_{\odot}^{*}, R_{\circ}^{*}
 )$,
 that is,
 \begin{align}\label{AD5}&T_{s}(\zeta)\prec T_{s}(\eta)=T_{s}(R_{\odot}^*(T_{s}(\zeta))\eta
 +R_{\circ}^*(T_{s}(\eta))\zeta-\zeta\preceq_{s+\tau(s)}\eta), \\&
 \label{AD6}T_{s}(\zeta)\succ T_{s}(\eta)=T_{s}(-L_{\star}^*(T_{s}(\zeta))\eta-R_{\succ}^*(T_{s}(\eta))\zeta
-\zeta\succeq_{s+\tau(s)}\eta). 
 \end{align}
\item $T_s$ is a relative Rota–Baxter operator of weight $-1$ on $(A,\circ)$
 with respect to the A-Novikov algebra $(A^{*},\circ_{s+\tau(s)},-L_{\odot}^{*},R_{\prec}^{*})$, that is,
 \begin{align}\label{AD8} &
 T_{s}(\zeta)\circ T_{s}(\eta)=T_{s}(-L_{\odot}^*(T_{s}(\zeta))\eta+R_{\prec}^*(T_{s}(\eta))\zeta-\zeta\circ_{s+\tau(s)}\eta), 
 \end{align}
  \end{enumerate}
for all $\zeta,\eta\in A^{*}$, where $\succeq_{s+\tau(s)},~\preceq_{s+\tau(s)}$ and $\circ_{s+\tau(s)}$ are given respectively by Eqs.~(\ref{AD1})-(\ref{AD3})
\end{thm}
  
 \begin{proof}
 In the light of Proposition \ref{YE80} and Theorem \ref{Ya1}, if $s+\tau(s)$ is invariant, then
 $s$ is a solution of the APN-YBE in $(A,\succ,\prec)$ if and only if
 \begin{equation*}T_{s}(\zeta)\circ T_{s}(\eta)=T_{s}(-L_{\odot}^*(T_{s}(\zeta))\eta-R_{\prec}^*(T_{\tau(s)}(\eta))\zeta).\end{equation*}
 Observe that
\begin{align*}T_{s}(\zeta)\circ T_{s}(\eta)&=T_{s}(-L_{\odot}^*(T_{s}(\zeta))\eta+R_{\prec}^*(T_{s}(\eta))\zeta-R_{\prec}^*(T_{s+\tau(s)}(\eta))\zeta)
\\&=T_{s}(-L_{\odot}^*(T_{s}(\zeta))\eta+R_{\prec}^*(T_{s}(\eta))\zeta-\zeta\circ_{s+\tau(s)}\eta).\end{align*}
Thus, Item (a) $\Longleftrightarrow$ Item (c).
Similarly, Item (a) $\Longleftrightarrow$ Item (b).
  \end{proof}
 If $s\in A\otimes A$ is skew-symmetric, then $s+\tau(s)=0$. Thus, Theorem \ref{YE8} is obtained.
 
\subsection{Quadratic Rota-Baxter anti-pre-Novikov algebras and factorizable anti-pre-Novikov bialgebras}
In this section, we first introduce a notion of quadratic Rota-Baxter anti-pre-Novikov algebras. 
Then we characterize the relationship between factorizable anti-pre-Novikov 
bialgebras and quadratic Rota-Baxter anti-pre-Novikov algebras.

\begin{defi} Let $(A,\succ,\prec,P)$ be a Rota-Baxter anti-pre-Novikov algebra of weight $\lambda$
and $(A,\succ,\prec,\omega)$ a quadratic anti-pre-Novikov algebra. Then $(A,\succ,\prec,P,\omega)$
is called a \textbf{quadratic Rota-Baxter anti-pre-Novikov algebra of weight $\lambda$} if the following condition holds:
\begin{equation} \label{Fs}\omega (P(x),y)+\omega(x, P(y))+\lambda\omega(x,y)=0, ~\forall~x, y \in A.\end{equation}
\end{defi}

\begin{defi} Let $(A,\circ,P)$ be a Rota-Baxter Novikov algebra of weight $\lambda$ and
 $(A,\circ,\omega )$ be a symmetric quasi-Frobenius Novikov algebra if the following condition holds: 
 \begin{equation} \label{Fs1}\omega (P(x),y)+\omega(x, P(y))+\lambda\omega(x,y)=0, ~\forall~x, y \in A.\end{equation}
Then $(A,\circ ,P,\omega)$ is called a \textbf{symmetric Rota-Baxter quasi-Frobenius Novikov algebra of 
weight $\lambda$}.
\end{defi}

Since  \begin{align*} &\lambda\omega(x,y)+\omega (-\lambda(x)- P(x),y)+\omega(x, -\lambda(y)- P(y))\\=&
-\lambda\omega(x,y)-\omega (P(x),y)-\omega(x,  P(y)), ~\forall~x, y \in A,\end{align*}
the following conclusions are clear.

\begin{pro} \label{Fb2} Let $(A,\succ,\prec,\omega)$ be a quadratic anti-pre-Novikov algebra and let $P:A\longrightarrow A$ be a linear
map. Then $(A,\succ,\prec,P,\omega)$ is a quadratic Rota–Baxter anti-pre-Novikov algebra of weight $\lambda$ if and only if
$(A,\succ,\prec,-\lambda I-P,\omega)$ is a quadratic Rota–Baxter anti-pre-Novikov algebra of weight $\lambda$.
\end{pro}

\begin{pro} Let $(A,\circ,\omega)$ be a symmetric quasi-Frobenius Novikov algebra and let $P:A\longrightarrow A$ be a linear
map. Then $(A,\circ ,P,\omega)$ is a symmetric Rota-Baxter quasi-Frobenius Novikov algebra of weight $\lambda$ if and only if
$(A,\circ ,-\lambda I-P,\omega)$ is a symmetric Rota-Baxter quasi-Frobenius Novikov algebra of weight $\lambda$.
\end{pro}

\begin{thm}\label{Fb0} Suppose that $(A,\circ,P,\omega)$ is a symmetric Rota-Baxter quasi-Frobenius Novikov algebra of 
weight $\lambda$. Then $(A,\succ,\prec,P,\omega)$ is a quadratic Rota-Baxter anti-pre-Novikov 
algebra of weight $\lambda$, where $\succ,\prec$ are defined by Eq.~(\ref{C2}).
 On the other hand, let $(A,\succ,\prec,P,\omega)$ be a quadratic Rota-Baxter anti-pre-Novikov algebra of weight $\lambda$.
Then $(A,\circ,P,\omega)$ is a symmetric Rota-Baxter 
quasi-Frobenius Novikov algebra of weight $\lambda$, where $\circ=\succ+\prec$.
\end{thm}

\begin{proof} Let $(A,\circ,P,\omega)$ is a symmetric Rota-Baxter quasi-Frobenius Novikov algebra of 
weight $\lambda$. By Theorem \ref{Am3}, $(A,\succ,\prec,\omega)$ is a quadratic anti-pre-Novikov algebra.
Using Eqs.~(\ref{C2}) and (\ref{Fs1}), for all $x,y,z\in A$, we have
\begin{align*}& \omega( P(x)\succ P(y)-P(P(x)\succ y+x\succ P(y)+\lambda x\succ y),z)
\\=&\omega( P(x)\star z, P(y))+\omega(P(x)\succ y, P(z))+\lambda\omega(P(x)\succ y, z)
+\omega(x\succ P(y), P(z))\\&+\lambda\omega(x\succ P(y), z)
+\lambda\omega(x\succ y, P(z))+\lambda^{2}\omega(x\succ y, z)
\\=&\omega( P(x)\star z, P(y))+\omega( P(x)\star  P(z), y)+\lambda \omega( P(x)\star z, y)
+\omega( x\star  P(z), P(y))\\&+\lambda \omega( x\star z, P(y))
+\lambda \omega( x\star P(z), y)
+\lambda^{2} \omega( x\star z, y)
\\=&\omega( P(x)\star z+x\star  P(z)+\lambda x\star z, P(y))+\omega( P(x)\star  P(z), y)
+\lambda \omega( P(x)\star z+x\star P(z)+\lambda x\star z, y)
\\=&0,\end{align*}
which implies that $P(x)\succ P(y)=P(P(x)\succ y+x\succ P(y)+\lambda x\succ y)$.
Analogously, $P(x)\prec P(y)=P(P(x)\prec y+x\prec P(y)+\lambda x\prec y)$.
Thus, $(A,\succ,\prec,P,\omega)$ is a quadratic Rota-Baxter anti-pre-Novikov algebra of weight $\lambda$.
The other hand is apparently.
We complete the proof.
\end{proof}

Let $\omega$ be a non-degenerate bilinear form on a vector space $A$. Then there is an isomorphism
$\omega^{\sharp}:A\longrightarrow A^{*}$ given by
\begin{equation} \omega(x,y)=\langle\omega^{\sharp}(x),y \rangle,~~\forall~x,y\in A.\end{equation}
Define an element $s_{\omega}\in A\otimes A$ such that $T_{s_{\omega}}=(\omega^{\sharp})^{-1}$,
that is, 
\begin{equation}\label{Nd1} \langle T_{s_{\omega}}(\zeta),\eta\rangle=\langle s_{\omega},\zeta \otimes\eta\rangle=
\langle (\omega^{\sharp})^{-1}(\zeta), \eta\rangle, \ \ \forall~\zeta,\eta\in A^{*}. \end{equation}

\begin{lem}\label{Fb1} Let $(A, \succ,\prec)$ be an anti-pre-Novikov algebra and $\omega$ be a non-degenerate bilinear form on $A$. Then
$(A, \succ,\prec,\omega)$ is a quadratic anti-pre-Novikov algebra if and only if the corresponding $s_{\omega}\in A\otimes A$ given by
Eq.~(\ref{Nd1}) is symmetric and invariant.\end{lem}
\begin{proof}
It is obvious that $\omega$ is symmetric if and only if $s_{\omega}$ is symmetric. 
For all $x,y\in A$, put $\omega^{\sharp}(x)=\zeta,\omega^{\sharp}(y)=\eta,
\omega^{\sharp}(z)=\theta$ with $\zeta,\eta,\theta\in A^{*}$. If $\omega$ is invariant, we have
\begin{align*} \omega (x \prec y, z)+\omega(x, z\circ y)
=&\langle \omega^{\sharp}(z), (\omega^{\sharp})^{-1}(\zeta)\prec y\rangle
+\langle \omega^{\sharp}(x), (\omega^{\sharp})^{-1}(\theta)\circ y\rangle
\\=&\langle \theta, T_{s_{\omega}}(\zeta)\prec y\rangle
+\langle \zeta, T_{s_{\omega}}(\theta)\circ y\rangle
\\=&\langle -L_{\prec}^{*}(T_{s_{\omega}}(\zeta))\theta-L_{\circ}^{*}(T_{s_{\omega}}(\theta))\zeta,  y\rangle=0,\end{align*}
\begin{align*} \omega(x \succ y, z)-\omega(x\circ z+z\circ x, y)
=&\langle \omega^{\sharp}(z),x\succ(\omega^{\sharp})^{-1}(\eta)\rangle-
\langle \omega^{\sharp}(y),x\circ(\omega^{\sharp})^{-1}(\theta)+(\omega^{\sharp})^{-1}(\theta)\circ x
\rangle
\\=&\langle \theta,x\succ T_{s_{\omega}}(\eta)\rangle-
\langle \eta,x\circ T_{s_{\omega}}(\theta)+T_{s_{\omega}}(\theta)\circ x
\rangle
\\=&\langle L_{\star}^{*}(T_{s_{\omega}}(\theta))\eta-R_{\succ}^{*}(T_{s_{\omega}}(\eta))\theta,x\rangle=0,
\end{align*}
Thus,
\begin{align}&\label{Sd1}L_{\prec}^{*}(T_{s_{\omega}}(\zeta))\theta=-L_{\circ}^{*}(T_{s_{\omega}}(\theta))\zeta,
\ \ \ L_{\star}^{*}(T_{s_{\omega}}(\theta))\eta=R_{\succ}^{*}(T_{s_{\omega}}(\eta))\theta.\end{align}
It follows that 
\begin{align}\label{Sd2}
R_{\odot}^{*}(T_{s_{\omega}}(\zeta))\theta=R_{\circ}^{*}(T_{s_{\omega}}(\theta))\zeta.\end{align}
Combining Proposition \ref{Si} and Eqs.~(\ref{Sd1})-(\ref{Sd2}), $s_{\omega}\in A\otimes A$ is invariant.
The converse part can be checked similarly.
 \end{proof}
 
\begin{pro} \label{QF1} Let $(A, \succ,\prec,\omega)$ be a quadratic anti-pre-Novikov algebra and $s\in A\otimes A$. 
Assume that $s+\tau(s)$ is invariant. Define a linear map 
\begin{equation}\label{Nd2}P:A\longrightarrow A,\ \ \ P(x)=T_{s}\omega^{\sharp}(x),~~\forall~x\in A.\end{equation}
Then $s$ is a solution of the APN-YBE in $(A, \succ,\prec)$ if and only if $P$ satisfies 
\begin{align}&\label{Nd3}P(x)\succ P(y)= P(P(x)\succ y+x\succ P(y)-x\succ T_{s+\tau(s)}\omega^{\sharp}(y)),
\\&\label{Nd4}P(x)\prec P(y)= P(P(x)\prec y+x\prec P(y)-x\prec T_{s+\tau(s)}\omega^{\sharp}(y)),~~\forall~x,y\in A.
 \end{align}
\end{pro}

\begin{proof} By Lemma \ref{Fb1}, $s_{\omega}$ is symmetric and invariant. In the light of Proposition \ref{Si},
 Eqs.~(\ref{IE8}), (\ref{IE10}) and (\ref{AD1}) ,
 for all $x,y\in A$, put $\omega^{\sharp}(x)=\zeta,\omega^{\sharp}(y)=\eta$ with 
$\zeta,\eta\in A^{*}$, we get
\begin{align*}& P(x)\succ P(y)=T_{s}\omega^{\sharp}(x)\succ T_{s}\omega^{\sharp}(y)=T_{s}(\zeta)\succ T_{s}(\eta),\\
&P(P(x)\succ y)=T_{s}\omega^{\sharp}(T_{s}(\zeta)\succ T_{s_{\omega}}(\eta))=
-T_{s}\omega^{\sharp} T_{s_{\omega}}(L_{\star}^{*}(T_{s}(\zeta))\eta)
=-T_{s}(L_{\star}^{*}(T_{s}(\zeta))\eta),
\\
&P(x\succ P(y))=T_{s}\omega^{\sharp}(T_{s_{\omega}}(\zeta)\succ T_{s}(\eta))=
-T_{s}\omega^{\sharp}T_{s_{\omega}}(R_{\succ}^{*}(T_{s}(\eta))\zeta)
=-T_{s}(R_{\succ}^{*}(T_{s}(\eta))\zeta),
\\
&P(x\succ T_{s+\tau(s)}\omega^{\sharp}(y))
=T_{s}\omega^{\sharp}(T_{s_{\omega}}(\zeta)\succ T_{s+\tau(s)}(\eta))
=-T_{s}\omega^{\sharp}  T_{s_{\omega}}( R_{\succ}^{*}(T_{s+\tau(s)}(\eta))\zeta)
\\=&-T_{s}(( R_{\succ}^{*}(T_{s+\tau(s)}(\eta))\zeta)
=T_{s}(\zeta\succeq_{s+\tau(s)}\eta).\end{align*}
 Thus, Eq.~(\ref{Nd3}) holds if and only if Eq.~(\ref{AD6}) holds.
 Analogously, Eq.~(\ref{Nd4}) holds if and only if Eq.~(\ref{AD5}) holds. The proof is finished.
 \end{proof}
 
\begin{lem} \label{QF2} Let $A$ be a vector space and $\omega$ be a non-degenerate symmetric bilinear form. 
Let $s\in A\otimes A$,
$\lambda\in k$ and $P$ be given by Eq.~(\ref{Nd2}). Then $s$ satisfies
\begin{equation} \label{Nd5} s+\tau(s)=-\lambda s_{\omega} \end{equation}
if and only if $P$ satisfies Eq.~(\ref{Fs}).
\end{lem}
\begin{proof} For all $x,y\in A$, put $\omega^{\sharp}(x)=\zeta,\omega^{\sharp}(y)=\eta,~~\zeta,\eta\in A^{*}$.
\begin{align*}& \omega(P(x),y)=\omega(y,P(x))=\langle\omega^{\sharp}(y),T_{s}\omega^{\sharp}(x)\rangle
=\langle \eta,T_{s}(\zeta)\rangle
=\langle s,\zeta\otimes\eta\rangle,\\
&\omega(x,P(y))=\langle \omega^{\sharp}(x),T_{s}\omega^{\sharp}(y)\rangle
=\langle \zeta,T_{s}(\eta)\rangle
=\langle s,\eta\otimes\zeta\rangle
=\langle \tau(s),\zeta\otimes\eta\rangle,
\\
&\lambda\omega(x,y)=\lambda \omega(y,x)
=\lambda\langle \omega^{\sharp}(y),(\omega^{\sharp})^{-1}\omega^{\sharp}(x)\rangle
=\lambda\langle s_{\omega},\zeta\otimes\eta\rangle
.\end{align*}
Thus, Eq.~(\ref{Nd5}) holds if and only if Eq.~(\ref{Fs}) holds.
 \end{proof}
 
\begin{cor} \label{Fb3} Let $(A, \succ,\prec,P,\omega)$ be a quadratic Rota–Baxter anti-pre-Novikov algebra of weight 0.
 Then there is a triangular
anti-pre-Novikov bialgebra $(A, \succ,\prec, \Delta_{\succ,s},\Delta_{\prec,s})$ with $\Delta_{\succ,s},\Delta_{\prec,s}$ given
 by Eqs.~(\ref{CB01})-(\ref{CB02}), where $s\in A\otimes A$ given by 
$T_{s}(\zeta)=P(\omega^{\sharp})^{-1}$ for all $\zeta\in A^{*}$.
\end{cor}

 \begin{proof} Since $s+\tau(s)=-\lambda s_{\omega}=0$, $s$ is skew-symmetric.
 In the light of Proposition \ref{QF1} and Lemma \ref{QF2}, we get the conclusion. \end{proof}
 
\begin{thm} \label{Fb3} Let $(A, \succ,\prec, \Delta_{\succ,s},\Delta_{\prec,s})$ be a factorizable 
anti-pre-Novikov bialgebra
with $s\in A\otimes A$. Then $(A, \succ,\prec,P,\omega)$
 is a quadratic Rota-Baxter anti-pre-Novikov algebra of weight $\lambda$ with $P$ given by Eq. (\ref{Nd2}), and $\omega$ is given by
\begin{equation} \label{Nd6} \omega(x,y)=-\lambda\langle T_{s+\tau(s)}^{-1}(x),y \rangle,~~\forall~x,y\in A.\end{equation}
Conversely, let $(A, \succ,\prec,P,\omega)$ be a quadratic Rota-Baxter anti-pre-Novikov algebra of weight $\lambda~(\lambda\neq 0)$.
Then there is a factorizable anti-pre-Novikov bialgebra $(A, \succ,\prec,\Delta_{\succ,s},\Delta_{\prec,s})$
 with $\Delta_{\succ,s},\Delta_{\prec,s}$ defined by Eqs.~(\ref{CB01})-(\ref{CB02}), where $s\in A\otimes A$ is
given through the operator form $T_s=P(\omega^{\sharp})^{-1}$.
\end{thm}

\begin{proof} On the one hand, since $(A, \succ,\prec, \Delta_{\succ,s},\Delta_{\prec,s})$ is a factorizable 
anti-pre-Novikov bialgebra, $s+\tau(s)$ is invariant and $T_{s+\tau(s)}$ is a linear isomorphism. 
By Proposition \ref{QF1} and Lemma \ref{QF2}, we get that $(A, \succ,\prec,P,\omega)$
 is a quadratic Rota-Baxter anti-pre-Novikov algebra of weight $\lambda$, where $\omega^{\sharp}=-\lambda T_{s+\tau(s)}^{-1}$. 
 Conversely, assume that $(A, \succ,\prec,P,\omega)$ is a quadratic Rota-Baxter anti-pre-Novikov algebra of weight $\lambda~(\lambda\neq 0)$.
 In the light of Lemma \ref{Fb1},  Lemma \ref{QF2} and Proposition \ref{QF1}, 
 $s+\tau(s)$ is invariant, $T_{s+\tau(s)}=-\lambda (\omega^{\sharp})^{-1}$ is a linear isomorphism
 and $s$ is a solution of the APN-YBE in $(A, \succ,\prec)$. Thus,
  $(A, \succ,\prec,\Delta_{\succ,s},\Delta_{\prec,s})$ is a factorizable anti-pre-Novikov bialgebra.
\end{proof}

\begin{pro} Let $(A, \succ,\prec,P)$ be a Rota–Baxter anti-pre-Novikov algebra of weight $\lambda$. Then 
$(A\ltimes A^{*},\omega,P-(P^{*}+\lambda I))$
is a quadratic Rota–Baxter anti-pre-Novikov algebra of weight $\lambda$, where the bilinear form $\omega$ on $A\oplus A^{*}$
is given by
\begin{equation*} \omega(x+\zeta,y+\eta)=\langle x,\eta\rangle+\langle y,\zeta\rangle,~~\forall~x,y\in A,~\zeta,\eta\in A^{*}.\end{equation*}
Then $(A\ltimes A^{*}, \Delta_{\succ,s},\Delta_{\prec,s})$
is a factorizable anti-pre-Novikov bialgebra
with $\Delta_{\succ,s},\Delta_{\prec,s}$ defined by Eqs.~(\ref{CB01})-(\ref{CB02}) with $s$ given by $T_s=P(\omega^{\sharp})^{-1}$.
  Explicitly, assume that $\{e_1,\cdot\cdot\cdot, e_n\}$
is a basis of $A$ and $\{e_{1}^{*},\cdot\cdot\cdot, e^{*}_n\}$
is the dual basis, where $s=\sum_{i}e_{i}^{*}\otimes P(e_{i})-(P+\lambda I)(e_{i})\otimes e_{i}^{*}$
 
\end{pro}

\begin{proof} By direct computations, $(A\ltimes A^{*},\omega,P-(P^{*}+\lambda I))$
is a quadratic Rota–Baxter anti-pre-Novikov algebra of weight $\lambda$.
For all $x\in A$ and $\zeta\in A^{*}$, $\omega^{\sharp}(x+\zeta)=x+\zeta$. By Corollary \ref{Fb3}, we have a linear map
$T_{s}:A\oplus A^{*}\longrightarrow A\oplus A^{*}$ defined by
\begin{equation*} 
T_{s}(x+\zeta)=(P-(P^{*}+\lambda I))(\omega^{\sharp})^{-1}(x+\zeta)=P(x)-(P^{*}+\lambda I)(\zeta).\end{equation*}
Thus, 
\begin{align*}&
\sum_{i,j}\langle s,e_{i}\otimes e_{j}^{*}\rangle=\sum_{i,j}\langle T_{s}(e_{i}), e_{j}^{*}\rangle=\sum_{i,j}\langle P(e_{i}), e_{j}^{*}\rangle
\\&\sum_{i,j}\langle s,e_{i}^{*}\otimes e_{j}\rangle=\sum_{i,j}\langle T_{s}(e_{i}^{*}), e_{j}\rangle=-\sum_{i,j}\langle (P^{*}+\lambda I)(e_{i}^{*}), e_{j}\rangle=-\sum_{i,j}\langle e_{i}^{*}, (P+\lambda I)(e_{j})\rangle.\end{align*}
It follows that $s=\sum_{i}e_{i}^{*}\otimes P(e_{i})-(P+\lambda I)(e_{i})\otimes e_{i}^{*}$.
\end{proof}

According to Theorem \ref{Fb3}, quadratic Rota-Baxter anti-pre-Novikov algebras of non-zero weight $\lambda$ 
are one to one correspondence to factorizable anti-pre-Novikov bialgebras. Combining  Theorem \ref{Fb0}, 
symmetric Rota-Baxter quasi-Frobenius Novikov algebras of non-zero weight correspond to factorizable anti-pre-Novikov bialgebras.

In view of Proposition \ref{Fb2} and Theorem \ref{Fb3}, 
since $(A,\succ,\prec,P,\omega)$ is a quadratic Rota–Baxter anti-pre-Novikov algebra of non-zero weight corresponds to a 
factorizable anti-pre-Novikov bialgebra, then the quadratic Rota–Baxter anti-pre-Novikov algebra $(A,\succ,\prec,-\lambda I-P,\omega)$
of non-zero weight also corresponds to a factorizable anti-pre-Novikov bialgebra.
In the light of Proposition \ref{Qf} and Theorem \ref{Fb3}, if $(A, \succ,\prec, \Delta_{\succ,s},\Delta_{\prec,s})$ is a factorizable 
anti-pre-Novikov bialgebra, then the factorizable 
anti-pre-Novikov bialgebra $(A, \succ,\prec, \Delta_{\succ,\tau(s)},\Delta_{\prec,\tau(s)})$ gives rise to a 
quadratic Rota–Baxter anti-pre-Novikov algebra of non-zero weight $\lambda$.
In fact, we have

\begin{pro}
Let $(A, \succ,\prec, \Delta_{\succ,s},\Delta_{\prec,s})$ be a factorizable 
anti-pre-Novikov bialgebra which corresponds to a quadratic Rota-Baxter anti-pre-Novikov algebras of non-zero weight $\lambda$. 
Then the factorizable 
anti-pre-Novikov bialgebra $(A, \succ,\prec, \Delta_{\succ,\tau(s)},\Delta_{\prec,\tau(s)})$ corresponds to the 
quadratic Rota–Baxter anti-pre-Novikov algebra $(A,\succ,\prec,-\lambda I-P,\omega)$ of non-zero weight $\lambda$. 
\end{pro}

\begin{proof} 
By Proposition \ref{Qf} and Theorem \ref{Fb1}, the factorizable 
anti-pre-Novikov bialgebra $(A, \succ,\prec, \Delta_{\succ,\tau(s)},\Delta_{\prec,\tau(s)})$ corresponds to a
quadratic Rota–Baxter anti-pre-Novikov algebra $(A,$ \ \ \ $\succ,\prec,P',\omega')$ of non-zero weight $\lambda$.
By Theorem \ref{Fb3},
\begin{equation*} 
\omega'(x,y)=-\lambda\langle T_{s+\tau(s)}^{-1}(x),y \rangle=\omega(x,y).\end{equation*}
Using Eqs.~(\ref{Nd2}) and (\ref{Nd6}), 
\begin{align*} 
P'(x)= T_{\tau(s)}{\omega^{'}}^{\sharp}(x)=T_{\tau(s)}\omega^{\sharp}(x)
&=-\lambda T_{\tau(s)}T_{s+\tau(s)}^{-1}(x)=\lambda (T_{s}-T_{s+\tau(s)})T_{s+\tau(s)}^{-1}(x)
\\&=\lambda T_{s}T_{s+\tau(s)}^{-1}(x)-\lambda (x)
=-T_{s}\omega^{\sharp}(x)-\lambda (x)=-P(x)-\lambda (x)
.\end{align*}
Thus, the factorizable 
anti-pre-Novikov bialgebra $(A, \succ,\prec, \Delta_{\succ,\tau(s)},\Delta_{\prec,\tau(s)})$ generates a
quadratic Rota-Baxter anti-pre-Novikov algebra $(A,\succ,\prec,-\lambda I-P,\omega)$ of non-zero weight $\lambda$.
Analogously, the converse part holds.
\end{proof} 


\begin{center}{\textbf{Acknowledgments}}
\end{center}
This work was supported by the Natural Science
Foundation of Zhejiang Province of China (LY19A010001), the Science
and Technology Planning Project of Zhejiang Province
(2022C01118).

\begin{center} {\textbf{Statements and Declarations}}
\end{center}
 All datasets underlying the conclusions of the paper are available
to readers. No conflict of interest exits in the submission of this
manuscript.


\end {document}